\documentclass[a4paper,11pt]{article}
\usepackage[dvips]{graphicx}
\usepackage{amsmath}
\usepackage{amsthm}
\usepackage{amssymb}

\theoremstyle{break}
\newtheorem{theo}{Theorem}[section]
\newtheorem{defi}{Definition}[section]
\newtheorem{prop}{Proposition}[section]
\newtheorem{lemm}{Lemma}[section]
\newtheorem{cor}{Corollary}[section]
\newtheorem{rem}{Remark}[section]

\makeatletter

\@addtoreset{equation}{section}
\makeatother

\makeatletter
\@addtoreset{figure}{section}

\makeatother

\begin{document}

\begin{center}
\LARGE
\textbf{Parametric Stokes Phenomenon for the second Painlev$\acute{\rm \textbf{e}}$ 
equation with a large parameter} \\[+1.em]

\large
Kohei Iwaki \\[+.5em]

Research Institute for Mathematical Sciences,  
Kyoto University, \\[+.2em]

Kyoto, 606-8502 Japan \\[+.5em]

iwaki@kurims.kyoto-u.ac.jp \\[+1.5em]
\end{center}

\begin{abstract}
The second Painlev\'e equation with a large parameter ($P_{\rm II}$) 
is analyzed by using the exact WKB analysis. 
The purpose of this study is to investigate the problem of the degeneration of 
$P$-Stokes geometry of ($P_{\rm II}$), 
which relates to a kind of Stokes phenomena for asymptotic (formal) solutions of ($P_{\rm II}$). 
We call this Stokes phenomenon a ``parametric Stokes phenomenon". 
We formulate the connection formula for this Stokes phenomenon, and confirm it in two ways: 
the first one is by computing the ``Voros coefficient" of ($P_{\rm II}$), 
and the second one is by using the isomonodromic deformation theory. 
Our main claim is that the connection formulas derived by these 
two completely different methods coincide. 
\end{abstract}

\tableofcontents

\section{Introduction and main results}

The theory of the exact WKB analysis of Painlev\'e equations 
with a large parameter has been established 
in the series of papers \cite{KT WKB Painleve I}, \cite{AKT Painleve WKB} and  
\cite{KT WKB Painleve III}. 
The main result of these papers is that any 2-parameter (formal) solution of the 
$J$-th Painlev\'e equation ($J$ = II, $\cdots$, VI), which has been 
constructed through multiple-scale analysis in \cite{AKT Painleve WKB}, 
can be transformed formally to a 2-parameter solution of the first Painlev\'e  
equation ($P_{\rm I}$) near a simple $P$-turning point. 
(In this paper we call turning points (resp. Stokes curves) of Painlev\'e equations, 
the definition of which is given in \cite{KT WKB Painleve I}, 
$P$-turning points (resp. $P$-Stokes curves). 
These terminologies have been already used in the exact WKB analysis of higher order 
Painlev\'e equations: cf.\ \cite{KT WKB higher Painleve}.)
Moreover, the connection formulas for 2-parameter solutions of ($P_{\rm I}$)
on $P$-Stokes curves were also discussed in \cite{Takei Painleve} 
by using the isomonodromic deformation method,
that is, by analyzing linear differential equations associated with ($P_{\rm I}$) 
through the exact WKB analysis.

The analysis presented in the above papers is concerned with 
the local theory near a simple $P$-turning point. 
But, for ($P_{J}$) ($J$ = II, $\cdots$, VI), two or more $P$-turning points appear in general  
and ``a degeneration of $P$-Stokes geometry" sometimes occurs;  
that is,  a $P$-Stokes curve may connect two $P$-turning points
when parameters contained in ($P_J$) take some special values.  
For example, the second Painlev\'e equation with a large parameter $\eta$
\[
(P_{\rm II}) : \frac{d^2 \lambda}{dt^2} =  \eta^2 (2 \lambda^3 + t \lambda + c) 
\]
has three $P$-turning points if $c \ne 0$, 
and Figure \ref{fig:P_II,argc=0.5Pi-epsilon} $\sim$ \ref{fig:P_II,argc=0.5Pi+epsilon} 
describe the $P$-Stokes curves of ($P_{\rm II}$) near arg \hspace{-.5em} $c = \frac{\pi}{2}$.
  \begin{figure}[h]
  \begin{minipage}{0.32\hsize}
  \begin{center}
  \includegraphics[width=45mm]{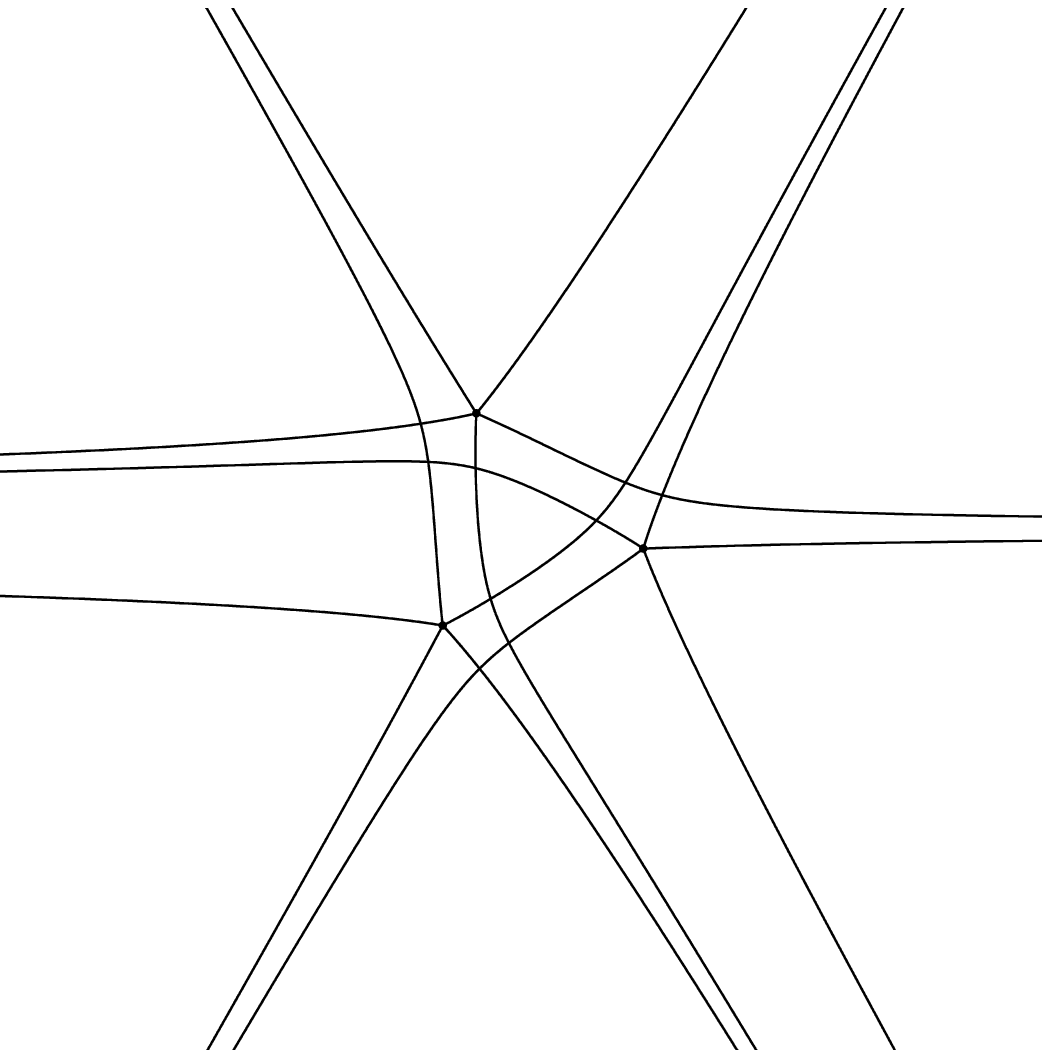}
  \end{center}
  \caption{\small{$P$-Stokes curves when arg \hspace{-.5em} $c = \frac{\pi}{2}-\varepsilon$.}}
  \label{fig:P_II,argc=0.5Pi-epsilon}
  \end{minipage} \hspace{+.3em}
  \begin{minipage}{0.32\hsize}
  \begin{center}
  \includegraphics[width=45mm]{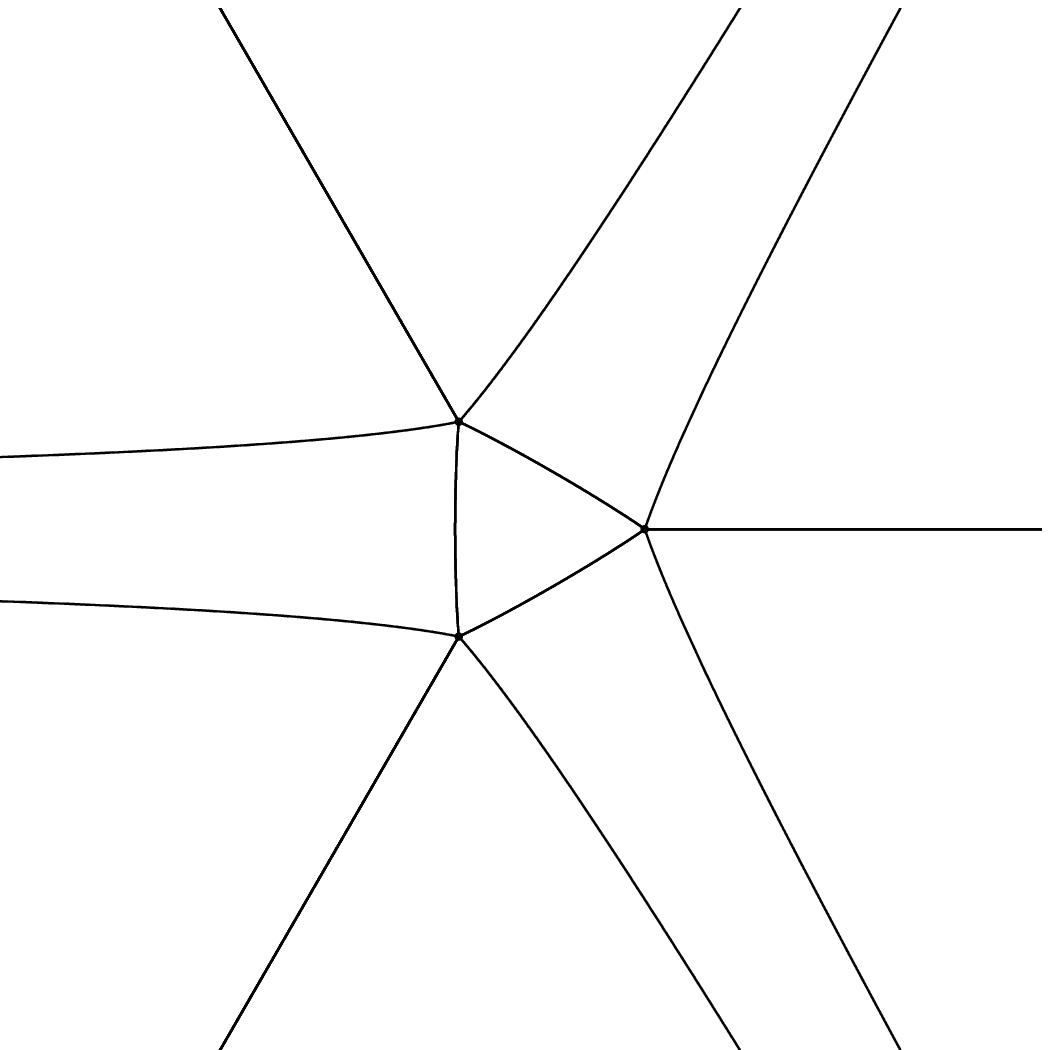}
  \end{center}
  \caption{\small{$P$-Stokes curves when arg \hspace{-.5em} $c = \frac{\pi}{2}$.}}
  \label{fig:P_II,argc=0.5Pi}
  \end{minipage} \hspace{+.3em}
  \begin{minipage}{0.32\hsize}
  \begin{center}
  \includegraphics[width=45mm]{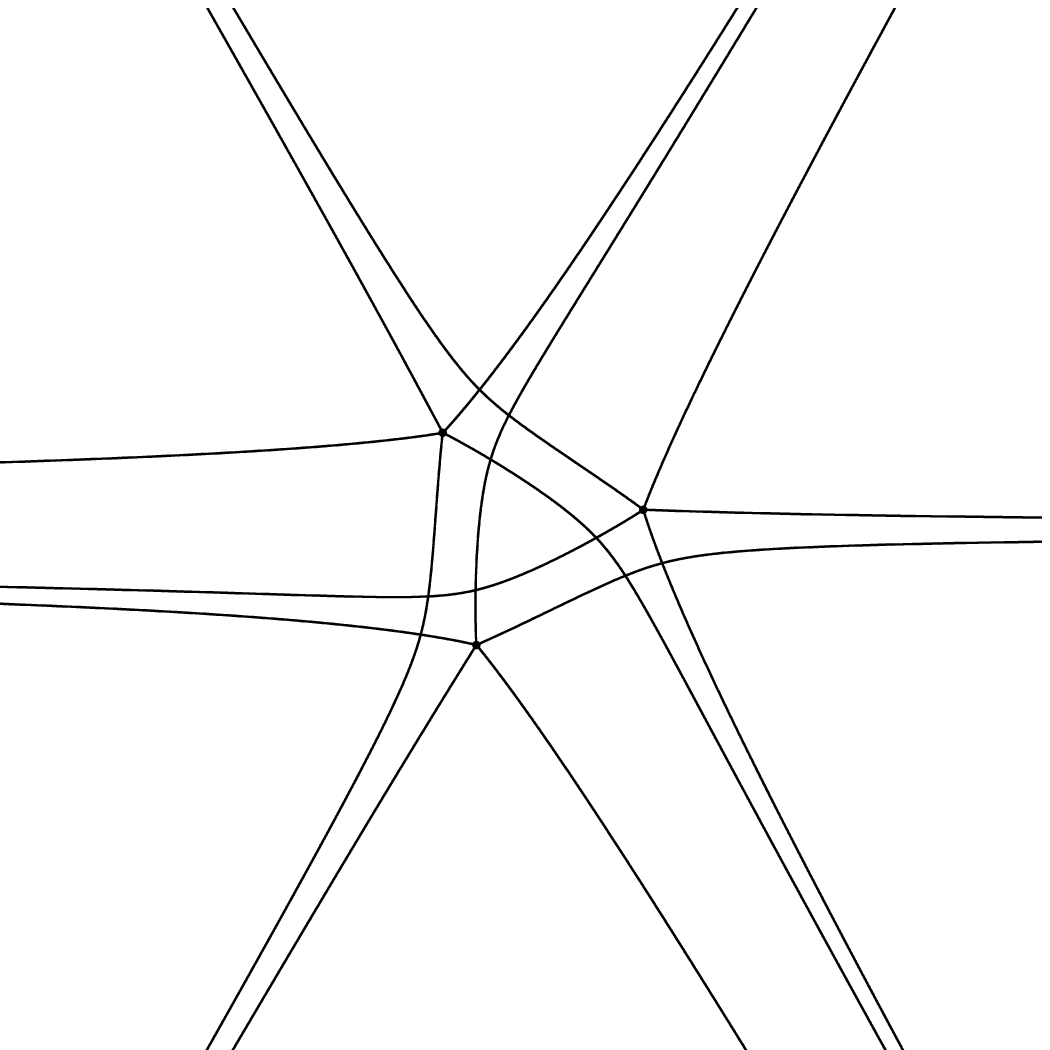}
  \end{center}
  \caption{\small{$P$-Stokes curves when arg \hspace{-.5em} $c = \frac{\pi}{2}+\varepsilon$}.}
  \label{fig:P_II,argc=0.5Pi+epsilon}
  \end{minipage}
  \end{figure}
(See Section 2.2 for the definitions of $P$-turning points and $P$-Stokes curves of ($P_{\rm II}$).) 
As is clear from Figure \ref{fig:P_II,argc=0.5Pi-epsilon} $\sim$ \ref{fig:P_II,argc=0.5Pi+epsilon}, 
a degeneration of $P$-Stokes geometry is observed when arg \hspace{-.4em} $c = \frac{\pi}{2}$. 
This degeneration suggests that a kind of Stokes phenomena 
occurs when $c$ varies near arg \hspace{-.4em} $c = \frac{\pi}{2}$, that is, 
the correspondence between asymptotic solutions and 
true solutions of ($P_{\rm II}$) 
changes discontinuously before and after the degeneration. 
We call this phenomenon ``a parametric Stokes phenomenon" because 
this Stokes phenomenon (or the degeneration of $P$-Stokes geometry) occurs when 
the parameter $c$ contained in ($P_{\rm II}$) varies. 
The purpose of this paper is to investigate the parametric Stokes phenomenon  
for ($P_{\rm II}$) in an explicit manner. 

In this paper we mainly discuss the parametric Stokes phenomenon for the 
following 1-parameter family of transseries solutions (1-parameter solutions) of ($P_{\rm II}$):  
\begin{equation}
\lambda(t,c,\eta;\alpha)  =  \lambda^{(0)}(t,c,\eta) + 
\alpha \eta^{-\frac{1}{2}} \lambda^{(1)}(t,c,\eta) \hspace{+.2em} {{e}}^{\eta \phi_{\rm{II}}} + 
(\alpha \eta^{-\frac{1}{2}})^2 \lambda^{(2)}(t,c,\eta) \hspace{+.2em} {{e}}^{2 \eta \phi_{\rm{II}}} 
+ \cdots . 
\label{eq:1-parameter solution}
\end{equation}
Here $\alpha$ is a free parameter, 
$\lambda^{(k)}(t,c,\eta) = \lambda^{(k)}_0(t,c) + 
\eta^{-1} \lambda^{(k)}_1(t,c) + \eta^{-2} \lambda^{(k)}_2(t,c) + \cdots$ 
($k \ge 0$) are formal power series of $\eta^{-1}$ and 
$\phi_{\rm II} = \phi_{\rm II}(t,c)$ is some function. 
(Note that $\lambda^{(0)}(t,c,\eta)$ itself is a formal power series solution of ($P_{\rm II}$), 
called 0-parameter solution.) 
In Section 3.1 we introduce two normalizations of 1-parameter solutions;  
let $\lambda_{\infty}(t,c,\eta;\alpha)$(resp. $\lambda_{\tau}(t,c,\eta;\alpha)$)
be a 1-parameter solution which is normalized at $t = \infty$ 
(resp. at some $P$-turning point $t = \tau$).
(A construction of 1-parameter solutions are briefly reviewed in Section 2.1.) 
Our aim is to explicitly determine the connection formulas 
for the parametric Stokes phenomenon occurring 
to 1-parameter solutions $\lambda_{\infty}(t,c,\eta;\alpha)$ and $\lambda_{\tau}(t,c,\eta;\alpha)$  
when arg \hspace{-.4em} $c = \frac{\pi}{2}$. 

The motivation of this paper comes from the result of Takei 
for the exact WKB analysis of the Weber equation
\begin{equation}
\Bigl( \frac{d^2}{d x^2}-\eta^2(E-\frac{1}{4}x^2) \Bigr) \psi = 0, 
    \label{eq:Weber equation1}
\end{equation}
which is a linear ordinary differential equation 
having two simple turning points at $x = \pm 2 \sqrt{E}$ if $E \ne 0$. 
A degeneration of Stokes geometry of \eqref{eq:Weber equation1} occurs 
when arg \hspace{-.3em} $E = 0$. 
In \cite{Takei Sato conjecture} Takei investigated the parametric Stokes phenomenon 
(with respect to the parameter $E$) of \eqref{eq:Weber equation1} 
and obtained a connection formula for WKB solutions of \eqref{eq:Weber equation1}
(\cite[Theorem 2.1]{Takei Sato conjecture}).
According to \cite{Takei Sato conjecture}, the parametric Stokes phenomenon  
is caused by some singularities of the Borel transform of WKB solutions, 
and furthermore, these singularities originate from the singularities of the Borel transform of
the ``Voros coefficient" of \eqref{eq:Weber equation1}. 
The Voros coefficient is a formal power series of $\eta^{-1}$ 
defined by some integral (cf.\ \cite{Takei Sato conjecture}). 
The explicit representation of the Voros coefficient of \eqref{eq:Weber equation1} 
(see \eqref{eq:Voros coefficient of Weber equation})
was first conjectured by Mikio Sato, and a proof of it based on the use of 
the creation operator of \eqref{eq:Weber equation1} was given by Takei in \cite{Takei Sato conjecture}. 
The connection formula for the parametric Stokes phenomenon was obtained as a 
corollary of it. In this paper we extend the analysis presented in \cite{Takei Sato conjecture}  
to ($P_{\rm II}$). 

Having the results of \cite{Takei Sato conjecture} for \eqref{eq:Weber equation1} in mind, 
we will introduce the ``Voros coefficient of ($P_{\rm II}$)" 
(or the ``$P$-Voros coefficient", for short) in Section 3.1, 
and obtain the explicit representation of it in Section 3.2. 
The following is one of the main results of this paper: 
\begin{theo} [{Theorem \ref{main theorem 1}}] \label{main theorem 1 in introduction} 
The $P$-Voros coefficient $W(c, \eta)$ is represented explicitly as follows:
\begin{equation}
W(c,\eta) 
 = - \sum_{n = 1}^{\infty}  
 \frac{2^{1-2n}-1}{2n(2n-1)}B_{2n} (c \hspace{+.1em} \eta)^{1 - 2n} ,
\label{eq:expression of PIIVoros1}
\end{equation}
where $B_{2n}$ is the 2n-th Bernoulli number defined by
\begin{equation}
\frac{w}{e^w - 1} = 1 - \frac{w}{2} + \sum_{n = 1}^{\infty} \frac{B_{2n}}{(2n)!} \hspace{+.2em} w^{2n}  .
\label{eq:Bernoulli number}
\end{equation}
\end{theo}

\noindent
Using this expression, 
we can derive the following connection formula for the parametric Stokes phenomenon 
under the assumption of Borel summability of 1-parameter solutions:  \\[+.6em]
\noindent
\textbf{Connection formula for 1-parameter solutions of ($P_{\rm II}$).} \hspace{+.1em}
\textit{Let $\varepsilon$ be a sufficiently small positive number. }

\noindent
\textit{ (i)  If the true solutions represented by 
$\lambda_{\infty}(t,c,\eta;\alpha)$ for arg c $ = \frac{\pi}{2} - \varepsilon$ and those by 
$\lambda_{\infty}(t,c,\eta;\tilde{\alpha})$ for arg c $ = \frac{\pi}{2} + \varepsilon$ coincide, 
then the following holds: }
\begin{equation}
\tilde{\alpha} = \alpha . 
\label{connection formula for 1-parameter solution normalized at infinity in introduction}
\end{equation}

\noindent
\textit{ (ii) If the true solutions represented by 
$\lambda_{\tau}(t,c,\eta;\alpha)$ for arg c $ = \frac{\pi}{2} - \varepsilon$ and those by 
$\lambda_{\tau}(t,c,\eta;\tilde{\alpha})$ for arg c $ = \frac{\pi}{2} + \varepsilon$ coincide,
then the following holds: }
\begin{equation}
\tilde{\alpha} = (1 + e^{2 \pi i c \eta}) \hspace{+.2em} \alpha . 
\label{connection formula for 1-parameter solution normalized at tau_1 in introduction}
\end{equation}

On the other hand, we can also derive the connection formula for the parametric 
Stokes phenomenon by using the isomonodromic deformation of ($SL_{\rm II}$), 
which is a second order linear ordinary differential equation relevant to ($P_{\rm II}$):  
\[
(SL_{\rm{II}}) : \hspace{+.2em}
\Bigl( \frac{\partial^2}{\partial x^2} - \eta^2 Q_{\rm{II}} \Bigr) \hspace{+.2em} \psi = 0 ,
\]
where
\[
Q_{\rm{II}} = x^4 + t x^2 +2 c x + 2 K_{\rm{II}} 
 - \eta^{-1} \frac{\nu}{x-\lambda} + \eta^{-2} \frac{3}{4(x-\lambda)^2} ,
\]
\[
K_{\rm{II}} = \frac{1}{2} 
 [ \nu^2 - (\lambda^4 + t \lambda^2 + 2 c \lambda) ] . 
\]
We compute the Stokes multipliers of ($SL_{\rm II}$) around $x = \infty$ by using the 
``Iso-Monodromic" WKB solutions $\psi_{\pm, \rm IM}$, which satisfy both 
($SL_{\rm II}$) and it's deformation equation ($D_{\rm II}$). 
(The construction of $\psi_{\pm, \rm IM}$ is explained in Section 4.2.) 
Let ${\mathfrak s}_j$ (resp. ${{\mathfrak s}_j}'$) be the Stokes multipliers computed 
by using $\psi_{\pm, {\rm IM}}$ when arg \hspace{-.5em} $c = \frac{\pi}{2} - \varepsilon$ 
(resp. arg \hspace{-.5em} $c = \frac{\pi}{2} + \varepsilon$) with a sufficiently 
small positive number $\varepsilon$ ($1 \leq j \leq 6$). 
The results of the computations (those are presented in Section 6.3) 
are given by the following lists: 
\\[-.5em]

\noindent
\textbf{Stokes multipliers of ($SL_{\rm II}$) around $x = \infty$.} 
\\[-1.em]

\noindent
\textit{(i) If the 1-parameter solution substituted into the coefficients of ($SL_{\rm II}$) 
and ($D_{\rm II}$) is normalized at $\infty$, 
the Stokes multipliers of ($SL_{\rm II}$) are given by the following}:
\begin{align}
\begin{cases}
{\mathfrak s}_{1}  =  i \hspace{+.3em} (1 + {{e}}^{2 \pi i c \eta}) \hspace{+.2em} {{e}}^{U - 2V} \\
{\mathfrak s}_{2}  =  i \hspace{+.3em}{{e}}^{- 2 \pi i c \eta} \hspace{+.2em} {{e}}^{2V - U} \\
{\mathfrak s}_{3}  = 
i \hspace{+.3em} (1 + {{e}}^{2 \pi i c \eta}) \hspace{+.2em} {{e}}^{- 2 \pi i c \eta} \hspace{+.2em}
{{e}}^{U - 2V} \hspace{+2.3em} \\ 
{\mathfrak s}_{4}  =  - 2 \sqrt{\pi} \hspace{+.1em} \alpha \\
{\mathfrak s}_{5}  =  0 \\
{\mathfrak s}_{6}  =  2 \sqrt{\pi} \hspace{+.1em} \alpha + i \hspace{+.3em} {{e}}^{2V - U} .
\end{cases}
 & \hspace{+.5em}
\begin{cases}
{\mathfrak s}_{1}'  =  i \hspace{+.3em} {{e}}^{U - 2V} \\
{\mathfrak s}_{2}'  =  
i \hspace{+.3em} (1 + {{e}}^{2 \pi i c \eta}) \hspace{+.2em} {{e}}^{- 2 \pi i c \eta} \hspace{+.2em} 
{{e}}^{2V - U} \hspace{+2.3em} \\
{\mathfrak s}_{3}'  = i \hspace{+.3em} {{e}}^{- 2 \pi i c \eta} \hspace{+.2em} {{e}}^{U - 2V} \\ 
{\mathfrak s}_{4}'  =  - 2 \sqrt{\pi} \hspace{+.1em} \alpha \\
{\mathfrak s}_{5}'  =  0 \\
{\mathfrak s}_{6}'  =  2 \sqrt{\pi} \hspace{+.1em} \alpha + 
i \hspace{+.3em} (1 + {{e}}^{2 \pi i c \eta}) \hspace{+.2em} {{e}}^{2V - U} .
\end{cases}
\label{eq:list1}
\end{align}

\noindent
\textit{(ii) If the 1-parameter solution substituted into the coefficients of ($SL_{\rm II}$) 
and ($D_{\rm II}$) is normalized at $\tau$, 
the Stokes multipliers of ($SL_{\rm II}$) are given by the following}:
\begin{align}
\begin{cases}
{\mathfrak s}_{1}  =  i \hspace{+.3em} (1 + {{e}}^{2 \pi i c \eta}) \hspace{+.2em} {{e}}^{U - 2V} \\
{\mathfrak s}_{2}  =  i \hspace{+.3em}{{e}}^{- 2 \pi i c \eta} \hspace{+.2em} {{e}}^{2V - U} \\
{\mathfrak s}_{3}  = 
i \hspace{+.3em} (1 + {{e}}^{2 \pi i c \eta}) \hspace{+.2em} {{e}}^{- 2 \pi i c \eta} \hspace{+.2em} 
{{e}}^{U - 2V} \hspace{+2.3em} \\ 
{\mathfrak s}_{4}  =  - 2 \sqrt{\pi} \hspace{+.1em} \alpha \hspace{+.2em} {e}^{W} \\
{\mathfrak s}_{5}  =  0 \\
{\mathfrak s}_{6}  =  2 \sqrt{\pi} \hspace{+.1em} \alpha \hspace{+.1em} {e}^{W} + 
i \hspace{+.3em} {{e}}^{2V - U} .
\end{cases}
 & \hspace{+.5em}
\begin{cases}
{\mathfrak s}_{1}'  =  i \hspace{+.3em} {{e}}^{U - 2V} \\
{\mathfrak s}_{2}'  =  
i \hspace{+.3em} (1 + {{e}}^{2 \pi i c \eta}) \hspace{+.2em} {{e}}^{- 2 \pi i c \eta} \hspace{+.2em} 
{{e}}^{2V - U} \hspace{+2.3em} \\
{\mathfrak s}_{3}'  = i \hspace{+.3em} {{e}}^{- 2 \pi i c \eta} \hspace{+.2em} {{e}}^{U - 2V} \\ 
{\mathfrak s}_{4}'  =  - 2 \sqrt{\pi} \hspace{+.1em} \alpha \hspace{+.2em} {e}^{W} \\
{\mathfrak s}_{5}'  =  0 \\
{\mathfrak s}_{6}'  =  
2 \sqrt{\pi} \hspace{+.1em} \alpha \hspace{+.1em} {e}^{W} + 
i \hspace{+.3em} (1 + {{e}}^{2 \pi i c \eta}) \hspace{+.2em} {{e}}^{2V - U} .
\end{cases}
\label{eq:list2}
\end{align}
\textit{Here $\alpha$ is the free parameter contained 
in the 1-parameter solution substituted into 
the coefficients of ($SL_{\rm II}$) and ($D_{\rm II}$), 
$U = U(t,c,\eta;\alpha)$ is given by the following integral 
\begin{equation}
U = \eta \int_{\infty}^{t} \bigl( \lambda(t,c,\eta;\alpha) - \lambda^{(0)}_0(t,c) \bigr) \hspace{+.2em} dt , 
\label{eq:U in introduction}
\end{equation}
$V = V(t,c,\eta;\alpha)$ is the Voros coefficient of ($SL_{\rm II}$) 
and $W = W(c,\eta)$ is the $P$-Voros coefficient.} \\[-.5em]

Precisely speaking, the Stokes multipliers are the Borel sums of the formal series in the above lists. 
They should be independent of $t$ because $\psi_{\pm, {\rm IM}}$ 
satisfies the deformation equation ($D_{\rm II}$).  
This fact can be confirmed by the following theorem, 
which will be verified through the analysis of 
the Voros coefficient $V$ of ($SL_{\rm II}$) in Section 5.  
\begin{theo} [{Theorem \ref{main theorem 2}}] 
The Voros coefficient $V$ of ($SL_{\rm II}$) and $U$ given by 
\eqref{eq:U in introduction} are related as follows: 
\begin{equation}
2V(t,c,\eta) - U(t,c,\eta) = 
- \sum_{n = 1}^{\infty} \frac{2^{1-2n} - 1}{2n(2n - 1)} B_{2n} (c \hspace{+.1em} \eta)^{1-2n} ,
\label{eq:2V-U in introduction}
\end{equation}
where $B_{2n}$ is the 2n-th Bernoulli number defined by \eqref{eq:Bernoulli number}.
\end{theo}
If two 1-parameter solutions of ($P_{\rm II}$) are given and 
true solutions of ($P_{\rm II}$) represented by these 1-parameter solutions coincide, 
then the corresponding Stokes multipliers of ($SL_{\rm II}$) should coincide. 
Thus, making use of the expression \eqref{eq:2V-U in introduction}, 
we can derive the connection formulas which describe the parametric Stokes phenomenon  
occurring to 1-parameter solutions substituted into the coefficients of ($SL_{\rm II}$) and ($D_{\rm II}$). 
The details will be explained in Section 6.4. 
Furthermore, the connection formulas obtained in this way coincide with 
\eqref{connection formula for 1-parameter solution normalized at infinity in introduction} and 
\eqref{connection formula for 1-parameter solution normalized at tau_1 in introduction},
that is, the connection formulas 
\eqref{connection formula for 1-parameter solution normalized at infinity in introduction} and 
\eqref{connection formula for 1-parameter solution normalized at tau_1 in introduction} 
describing the parametric Stokes phenomenon for 1-parameter solutions 
can be confirmed also by the isomonodromic deformation method. 
This is our main claim. 

\subsection*{Acknowledgment}
The author is very grateful to Professor Yoshitsugu Takei, Professor Takahiro Kawai, 
Professor Takashi Aoki, Professor Tatsuya Koike and Professor Yousuke Ohyama
for helpful advices and constant encouragements.


\section{1-parameter solutions and the $P$-Stokes geometry of ($P_{\rm II}$)}

The general theory of the exact WKB analysis of Painlev$\acute{\rm e}$ equations ($P_{\rm J}$) 
(${\rm J = I, \cdots ,VI}$) was developed in a series of papers \cite{KT WKB Painleve I}, 
\cite{AKT Painleve WKB}, \cite{KT WKB Painleve III} and 1-parameter solutions of ($P_J$) 
were discussed in \cite{KT WKB Deformation of SL_J}. (See also \cite[$\S$4]{KT iwanami}.) 
In this section, we review the core part of the exact WKB analysis of ($P_J$) in the case of ($P_{\rm II}$). 
\[
(P_{\rm II}) : \frac{d^2 \lambda}{dt^2} =  \eta^2 (2 \lambda^3 + t \lambda + c). 
\]


\subsection{0-parameter solutions and 1-parameter solutions of ($H_{\rm II}$)}

($P_{\rm II}$) is equivalent to the following Hamiltonian system ($H_{\rm II}$):
\begin{eqnarray*}
(H_{\rm II}) : 
 \left\{
\begin{array}{ll}
\displaystyle \frac{d \lambda}{d t} = \eta \hspace{+.2em} \nu \hspace{+.2em}, \\[+1.em]
\displaystyle \frac{d \nu}{d t} = \eta \hspace{+.2em} ( 2 \lambda^3 + t \lambda + c) .
\end{array} \right. 
\end{eqnarray*}

\noindent
($H_{\rm II}$) has a formal power series solution ($\lambda^{(0)}, \nu^{(0)}$) of $\eta^{-1}$  
called a 0-parameter solution of ($H_{\rm II}$). 
\begin{eqnarray*}
\left\{
\begin{array}{ll}
\lambda^{(0)} (t,c,\eta) = \lambda^{(0)}_0(t,c) + \eta^{-1} \lambda^{(0)}_1(t,c) + 
\eta^{-2} \lambda^{(0)}_2(t,c) + \cdots  \hspace{+.2em},  \\[+.5em]
\nu^{(0)} (t,c,\eta) = \nu^{(0)}_0(t,c) + \eta^{-1} \nu^{(0)}_1(t,c) 
+ \eta^{-2} \nu^{(0)}_2(t,c) + \cdots \hspace{+.2em}.
\end{array} \right. 
\end{eqnarray*}

\noindent
Here $\lambda^{(0)}_0$ satisfies 
\begin{equation}
2 {\lambda^{(0)}_0}^3 + t \lambda^{(0)}_0 + c = 0   \label{eq:lambda_0}
\end{equation}
and 
\begin{equation}
\nu^{(0)}_0 = 0 .  \label{eq:nu_0}
\end{equation}
In what follows we abbreviate $\lambda^{(0)}_0$ and $\nu^{(0)}_0$ 
to $\lambda_0$ and $\nu_0$, respectively.
Once the branch of $\lambda_0$ is fixed,
the coefficients ($\lambda^{(0)}_k$, $\nu^{(0)}_k$) 
of $\eta^{-k}$ in ($\lambda^{(0)}$ , $\nu^{(0)}$) 
are determined by the following recursive relations:
\begin{equation}
(6 \lambda_0^2 + t) \lambda^{(0)}_k + 2 \sum_{
\tiny
\begin{array}{ll}
k_1+k_2+k_3 = k  \\ 
\hspace{+1.4em} 0 \le k_j < k
\end{array} 
}
\normalsize
\lambda^{(0)}_{k_1} \lambda^{(0)}_{k_2} \lambda^{(0)}_{k_3} = 
\frac{d^2 \lambda^{(0)}_{k-2}}{d t^2} \hspace{+1.em} ( k \ge 1) \hspace{+.1em},
\label{eq:recurrence relation of lambda(0)}
\end{equation}
\begin{equation}
\nu^{(0)}_k = \frac{d \lambda^{(0)}_{k-1}}{dt} \hspace{+1.em}(k \ge 1) .
\label{eq:relation between lambda(0)_k and nu(0)_k}
\end{equation}

We can also construct a 1-parameter family of formal solutions, 
called 1-parameter solution, of ($H_{\rm II}$) of the following form:
\begin{eqnarray*}
\left\{
\begin{array}{ll}
\lambda(t,c,\eta;\alpha)  =  \lambda^{(0)}(t,c,\eta) + 
\alpha \eta^{-\frac{1}{2}} \lambda^{(1)}(t,c,\eta) {{e}}^{\eta \phi_{\rm{II}}} + 
(\alpha \eta^{-\frac{1}{2}})^2 \lambda^{(2)}(t,c,\eta) {{e}}^{2 \eta \phi_{\rm{II}}} 
+ \cdots \hspace{+.2em},
\label{eq:1-parameterlambda} \\[+.4em]
\nu(t,c,\eta;\alpha)  =  \nu^{(0)}(t,c,\eta) + 
\alpha \eta^{-\frac{1}{2}} \nu^{(1)}(t,c,\eta) {{e}}^{\eta \phi_{\rm{II}}} + 
(\alpha \eta^{-\frac{1}{2}})^2 \nu^{(2)}(t,c,\eta) {{e}}^{2 \eta \phi_{\rm{II}}} 
+ \cdots \hspace{+.2em},
\label{eq:1-parameternu}
\end{array} \right.
\end{eqnarray*}
where $\alpha$ is a free parameter, 
$\lambda^{(k)}$ and $\nu^{(k)}$ are formal power series of $\eta^{-1}$,
\[
\lambda^{(k)}(t,c,\eta) 
 = \lambda_0^{(k)}(t,c) + \eta^{-1} \lambda_1^{(k)}(t,c) + \eta^{-2} \lambda_2^{(k)}(t,c) 
 + \cdots  ,
\]
\[
\nu^{(k)}(t,c,\eta) 
 = \nu_0^{(k)}(t,c) + \eta^{-1} \nu_1^{(k)}(t,c) + \eta^{-2} \nu_2^{(k)}(t,c)  
 + \cdots  ,
\]
and 
\[
\phi_{\rm{II}}(t,c) = \int^t \sqrt{\Delta(t,c)} \hspace{+.2em} dt  ,
\]
\[
\Delta(t,c) = 6 \lambda_0(t,c)^2 + t .
\]
In what follows 1-parameter solutions are considered in a domain 
where the real part of $\phi_{\rm II}$ is negative, i.e.,  
$e^{\eta \phi_{\rm II}}$ is exponentially small when $\eta \rightarrow \infty$.

Here we briefly recall the construction of a 1-parameter solution. 
(It is similar to the construction of the so-called transseries solution; cf. \cite[pp.90-91]{Costin}.) 
First, $\lambda^{(0)}$ is a 0-parameter solution constructed above. Second, if 
$\alpha \eta^{-\frac{1}{2}} \lambda^{(1)}(t,c,\eta) {{e}}^{\eta \phi_{\rm{II}}}$ 
is denoted by $\tilde{\lambda}^{(1)}$, 
then $\tilde{\lambda}^{(1)}$ is a solution of the following second order linear differential equation: 
\begin{equation}
\frac{d^2 \tilde{\lambda}^{(1)}}{d t^2} = 
\eta^2 (6 {\lambda^{(0)}(t,c,\eta)}^2 + t) \hspace{+.2em} \tilde{\lambda}^{(1)} ,
\label{eq:lambda(1)} 
\end{equation}
that is, the Fr$\acute{\rm e}$chet derivative of ($P_{\rm II}$) at $\lambda = \lambda^{(0)}$.
Thus $\tilde{\lambda}^{(1)}$ can be taken as a WKB solution 
(see \cite[$\S$2.1]{KT iwanami} for example) of \eqref{eq:lambda(1)} of the form
\begin{eqnarray}
\tilde{\lambda}^{(1)}
 & = & 
\alpha \hspace{+.1em} \frac{1}{\sqrt{R_{\rm{odd}}(t,c,\eta)}} \hspace{+.2em} {\rm{exp}} 
\Bigl( \int^t R_{\rm{odd}}(t,c,\eta) dt \Bigr)  
\label{eq:lambda(1)tilde}  \\
 & = & 
\alpha \eta^{- \frac{1}{2}}
(\lambda_0^{(1)}(t,c) + \eta^{-1} \lambda_1^{(1)}(t,c) + \eta^{-2} \lambda_2^{(1)}(t,c) + \cdots) 
{e}^{\eta \phi_{\rm{II}}} ,
    \label{eq:expansion of lambda(1)tilde}
\end{eqnarray}
where $R_{\rm odd}$ is the odd part (in the sense of Remark \ref{odd part} below) of 
a formal power series solution $R = \eta R_{-1} + R_0 + \eta^{-1} R_1 + \cdots$ 
of the Riccati equation 
\begin{equation}
R^2 + \frac{dR}{dt} = \eta^2 (6 {\lambda^{(0)}(t,c,\eta)}^ 2 + t)  \label{eq:R}
\end{equation}
associated with \eqref{eq:lambda(1)}.
\begin{rem}\label{odd part} \normalfont
The odd part of $R$ is defined as follows.
The coefficients $R_k$ of $\eta^{-k}$ in $R$ are determined by the following recursive relations:
\[
R_{-1}^2 = \Delta = 6 \lambda_0^2 + t,
\]
\begin{equation}
2 R_{-1} R_{k+1} + \hspace{-1.em}
\sum_{
\tiny
\begin{array}{ll}
k_1+k_2 = k  \\ 
\hspace{+1.5em} 0 \le k_j 
\end{array} 
\normalsize} \hspace{-1.em}
R_{k_1} R_{k_2} 
 + \frac{d R_k}{d t} 
 = 6 \hspace{-1.em}
\sum_{
\tiny
\begin{array}{ll}
l_1+l_2 = k + 2  \\ 
\hspace{+2.em} 0 \le l_j
\end{array} 
\normalsize} \hspace{-1.em}
\lambda^{(0)}_{l_1} \lambda^{(0)}_{l_2}  \hspace{+.7em}\hspace{+.3em}(k \ge -1).
\label{eq:R_k+1}
\end{equation}
Once the branch of 
\begin{equation}
R_{-1} = \sqrt{\Delta} \hspace{+.2em} 
\label{eq:R_-1}
\end{equation}
is fixed, $R_k$ ($k \ge 0$) is determined uniquely by \eqref{eq:R_k+1}. 
Hence we obtain a formal solution $R = R(t,c,\eta)$ of \eqref{eq:R}.  
Similarly we obtain a formal solution $R^{\dagger} = R^{\dagger}(t,c,\eta)$ of \eqref{eq:R}, 
starting with $R^{\dagger}_{-1} = - \sqrt{\Delta}$. 
Then, we define the odd and even part of $R$ by
\begin{eqnarray}
R_{\rm odd}(t,c,\eta) & = & \frac{1}{2} \bigl( R(t,c,\eta) - R^{\dagger}(t,c,\eta) \bigr), \\
R_{\rm even}(t,c,\eta) & = & \frac{1}{2} \bigl( R(t,c,\eta) + R^{\dagger}(t,c,\eta) \bigr).
\end{eqnarray}
We also note that,
as in \cite[$\S$2.1]{KT iwanami}, 
\begin{equation}
R_{\rm{even}} = - \frac{1}{2} \frac{1}{R_{\rm odd}} \frac{d R_{\rm odd}}{d t}
 = - \frac{1}{2} \frac{d}{dt} \hspace{.2em} {\rm log} R_{\rm{odd}}.
\label{eq:R_odd and R_even}
\end{equation}

\end{rem}

An important fact is that, once we fix a normalization (i.e., the lower endpoint) of the integral of 
$R_{\rm odd}$ in \eqref{eq:lambda(1)tilde}, 
the coefficients $\lambda^{(k)}_\ell$ of $\lambda^{(k)}$ ($k \ge 2$, $\ell \ge 0$) 
are determined uniquely by the following recursive relations.
\begin{eqnarray}
(k^2 - 1) \Delta(t,c) \hspace{+.2em} \lambda_\ell^{(k)} & = & 
6 \hspace{-1.em} \sum_{
\tiny
\begin{array}{ll}
\ell_1 + \ell_2 + \ell_3 = \ell \\
\hspace{+2.em} \ell_3 < \ell
\end{array} } \hspace{-1.em}
\lambda_{\ell_1}^{(0)} \lambda_{\ell_2}^{(0)} \lambda_{\ell_3}^{(k)}  \nonumber \\
&  &
+ \hspace{+.2em}
2 \hspace{-1.em} \sum_{
\tiny
\begin{array}{ll}
k_1 + k_2 + k_3 = k  \\
\ell_1 + \ell_2 + \ell_3 = \ell \\
\hspace{+2.em} k_j < k
\end{array} } \hspace{-1.em}
\lambda_{\ell_1}^{(k_1)} \lambda_{\ell_2}^{(k_2)} \lambda_{\ell_3}^{(k_3)} \nonumber \\
&   &
-
2 k \frac{d \phi_{\rm{II}}}{dt} \frac{d \lambda_{\ell - 1}^{(k)}}{dt}
-
k \frac{d^2 \phi_{\rm{II}}}{dt^2} \lambda_{\ell - 1}^{(k)}
-
\frac{d^2 \lambda_{\ell - 2}^{(k)}}{dt^2} .
\label{eq:lambda_(k)l}
\end{eqnarray}
Thus we can construct a 1-parameter solution $\lambda(t,c,\eta;\alpha)$ 
including a free parameter $\alpha$. 
(Normalization of the integral in \eqref{eq:lambda(1)tilde} will be discussed in Section 3.)
Since $\nu = \eta^{-1} \frac{d \lambda}{dt}$ follows from ($H_{\rm II}$), 
the formal power series $\nu^{(k)}$ ($k \ge 0$) are determined by
\begin{equation}
\nu^{(k)} = k \frac{d \phi_{\rm{II}}}{dt} \lambda^{(k)} + 
\eta^{-1} \frac{\hspace{+.5em} d \lambda^{(k)}}{dt} .
\label{eq:relation between lambda(k) and nu(k) in full order}
\end{equation}
Especially, since $\tilde{\lambda}^{(1)} = \alpha \eta^{-\frac{1}{2}} \lambda^{(1)} e^{\eta \phi_{\rm II}}$ 
can be written also as
\[
\tilde{\lambda}^{(1)} = \alpha \eta^{-\frac{1}{2}} C(\eta) \hspace{+.2em} {\rm exp}
\Bigl( \int^{t} R(t,c,\eta) \hspace{+.2em} dt \Bigr) 
\]
with a formal power series $C(\eta)$ of $\eta^{-1}$ 
whose coefficients are independent of $t$, we have 
\begin{equation}
\nu^{(1)} = \eta^{-1} R \hspace{+.2em} \lambda^{(1)} .  \label{eq:nu(1)}
\end{equation}


\subsection{$P$-Stokes geometry of ($P_{\rm II}$)}

Next, we recall the definition of turning points (``$P$-turning points") and 
Stokes curves (``$P$-Stokes curves") of ($P_{\rm II}$).

\begin{defi} [{\cite[$\S$4, Definition 4.5]{KT iwanami}}]
\normalfont
(i) A point $t$ is called a \textit{$P$-turning point of $(P_{\rm II})$} if $t$ satisfies 
$\Delta = 6 \lambda_0^2 + t = 0$. 

\noindent
(ii) For a $P$-turning point $t = \tau$, a real one-dimensional curve defined by
\[
{\rm Im} \int_{\tau}^t \sqrt{\Delta(t,c)} \hspace{+.2em} dt = 0
\]
is said to be a \textit{$P$-Stokes curve of $(P_{\rm II})$}.
\end{defi}

The $P$-turning points and the $P$-Stokes curves are the 
turning points and the Stokes curves of the linear equation \eqref{eq:lambda(1)}. 
Since $P$-turning points are also zeros of discriminant of the algebraic equation \eqref{eq:lambda_0}
for $\lambda_0$, there are three $P$-turning points at 
$t = \tau_j := -6 \hspace{+.2em} \bigl( c / 4 \bigr)^{2/3} \omega^j$ 
$(\omega = e^{\frac{2 \pi i}{3}}, j=1,2,3)$ in the case $c \ne 0$. 
Figure \ref{fig:P_II,argc=0} describes $P$-Stokes curves when arg \hspace{-.5em} $c = 0$. 
 \begin{figure}[h]
 \begin{center}
 \includegraphics[width=50mm]{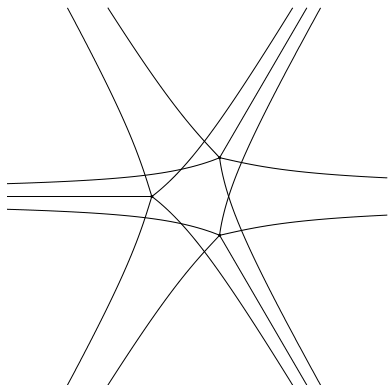}
 \end{center}
 \caption{$P$-Stokes curve when arg \hspace{-.5em} $c =0$.}
 \label{fig:P_II,argc=0}
 \end{figure}
As we saw in Section 1, some degeneration of $P$-Stokes geometry occurs 
when arg \hspace{-.4em} $c = \frac{\pi}{2}$
(Figure \ref{fig:P_II,argc=0.5Pi-epsilon} $\sim$ \ref{fig:P_II,argc=0.5Pi+epsilon}). 
This degeneracy can be analytically confirmed by the relation 
\[
\int_{\tau_1}^{\tau_2} \sqrt{\Delta} \hspace{+.2em} dt = \pm 2 \pi i c , 
\]
which we will show in Proposition \ref{integral of top terms} in Section 4. 
(The choice of the sign on the right-hand side of the above relation
depends on the determination of branch of $\sqrt{\Delta}$.) 
\begin{rem} \label{remark for the branch of sqrt-Delta}
\normalfont
Since $P$-turning points and $P$-Stokes curves 
are defined in terms of $\Delta = 6 \lambda_0^2 + t$, 
it is natural to lift them onto the Riemann surface of $\lambda_0$. 
Figure \ref{fig:Riemann surface of lambda_0} 
describes the lift of $P$-Stokes curves 
onto the Riemann surface of $\lambda_0$ when arg \hspace{-.5em} $c = \frac{\pi}{2}$. 
  \begin{figure}[h]
  \begin{minipage}{0.32\hsize}
  \begin{center}
  \includegraphics[width=45mm]{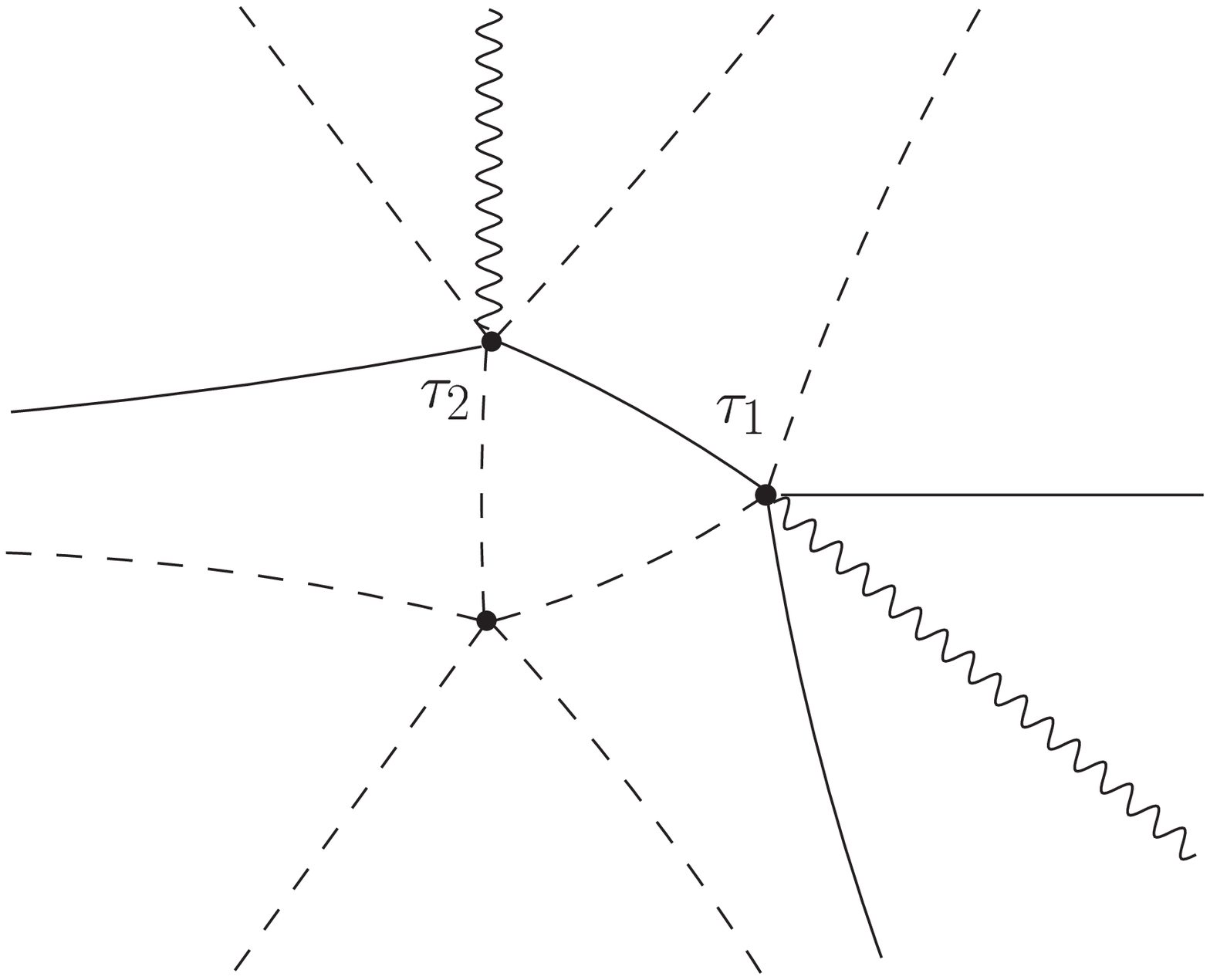}
  \end{center}
  \end{minipage}
  \begin{minipage}{0.32\hsize}
  \begin{center}
  \includegraphics[width=45mm]{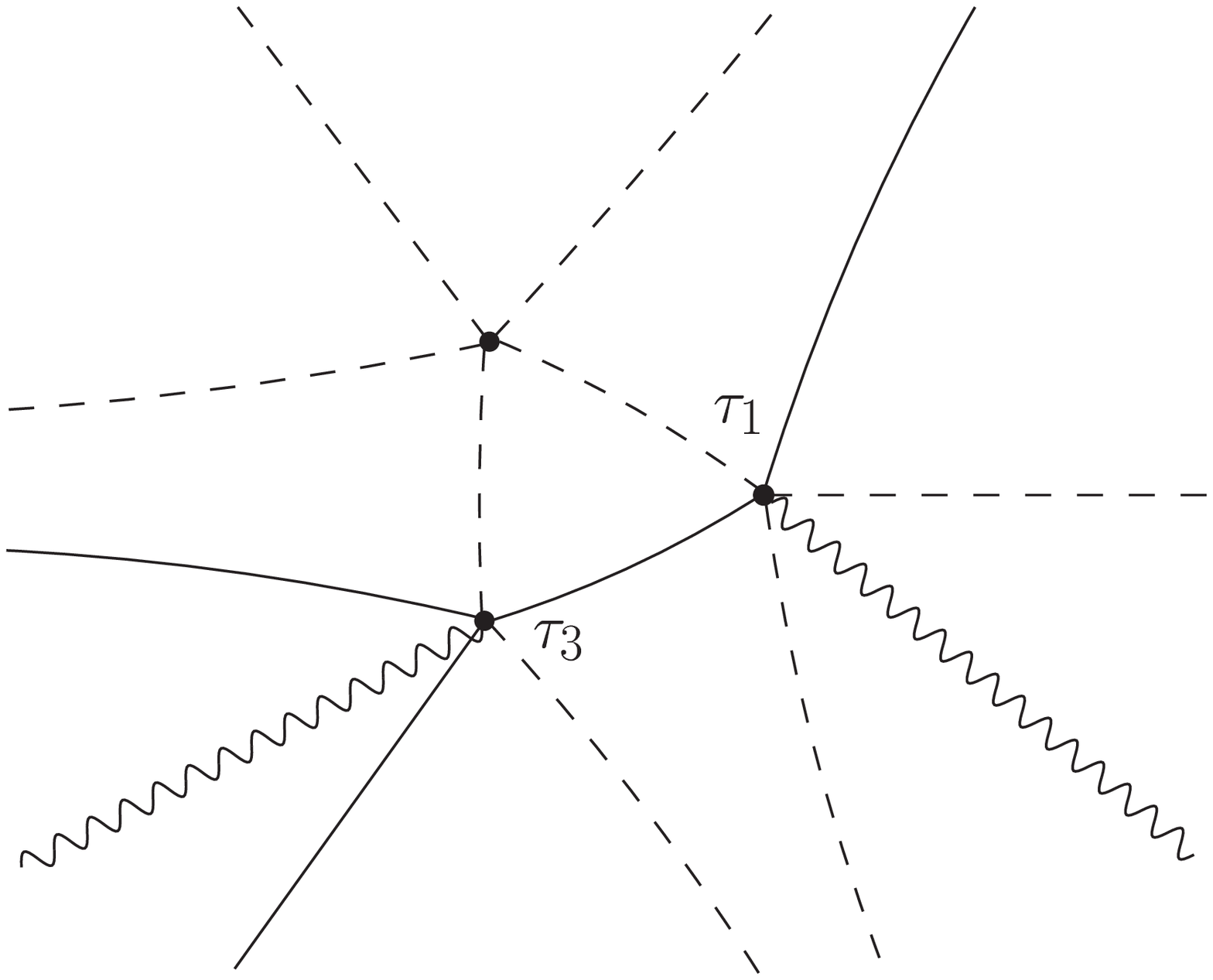}
  \end{center}
  \caption{\small{The lift of $P$-Stokes curves onto the Riemann surface of $\lambda_0$}}
  \label{fig:Riemann surface of lambda_0}
  \end{minipage}
  \begin{minipage}{0.32\hsize}
  \begin{center}
  \includegraphics[width=45mm]{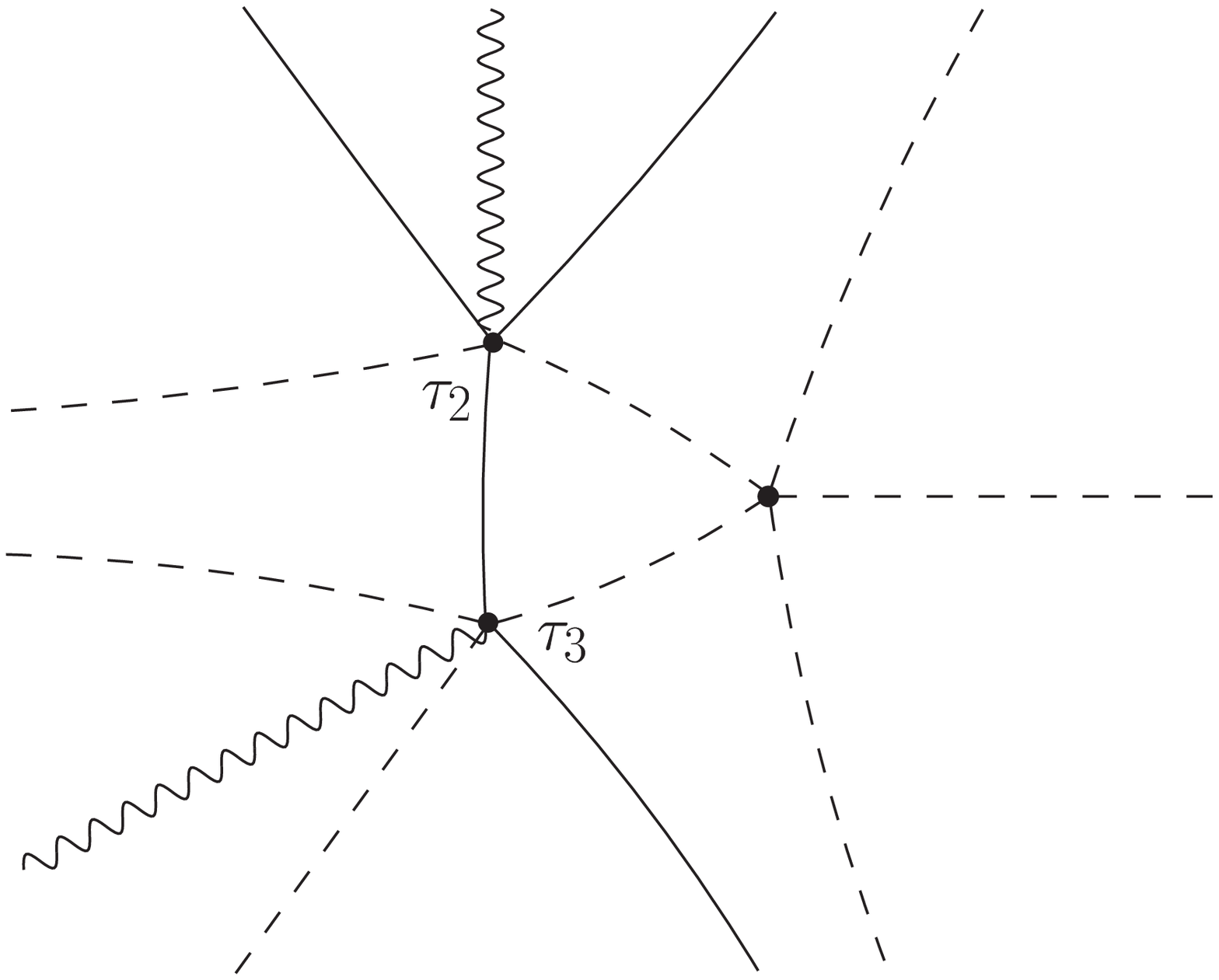}
  \end{center}
  \end{minipage}
  \end{figure}
Wiggly lines, solid lines and dotted lines in Figure \ref{fig:Riemann surface of lambda_0}
represent cuts to define the Riemann surface of $\lambda_0$, 
$P$-Stokes curves on the sheet under consideration 
and $P$-Stokes curves on the other sheets, respectively. 
In this paper we only consider the situation 
where arg \hspace{-.5em} $c$ is sufficiently close to $\frac{\pi}{2}$ 
and $t$ moves in the shaded domain in Figure \ref{fig:Riemann surface of sqrt Delta} below. 
  \begin{figure}[h]
  \begin{minipage}{0.32\hsize}
  \begin{center}
  \includegraphics[width=45mm]{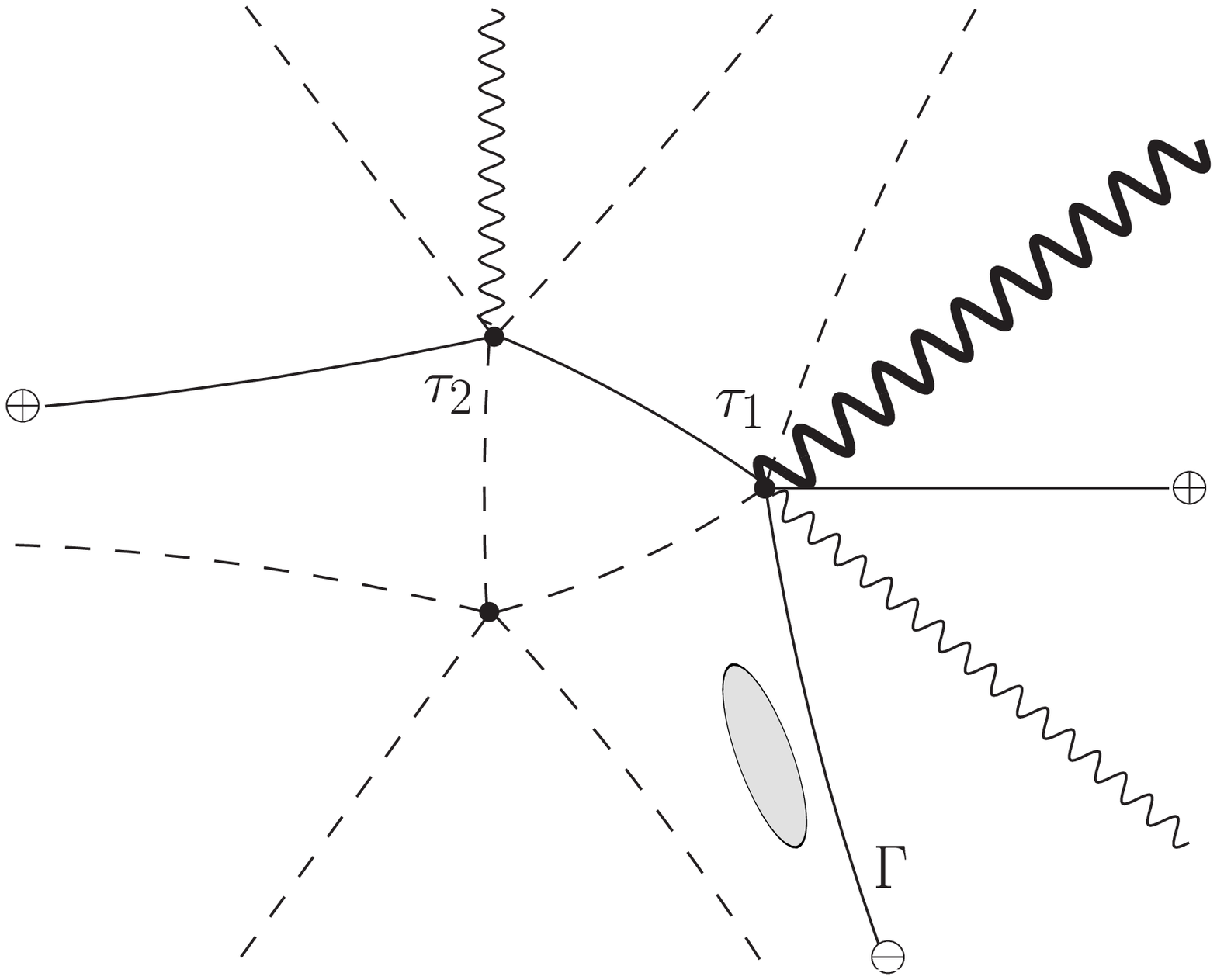}
  \end{center}
  \end{minipage}
  \begin{minipage}{0.32\hsize}
  \begin{center}
  \includegraphics[width=45mm]{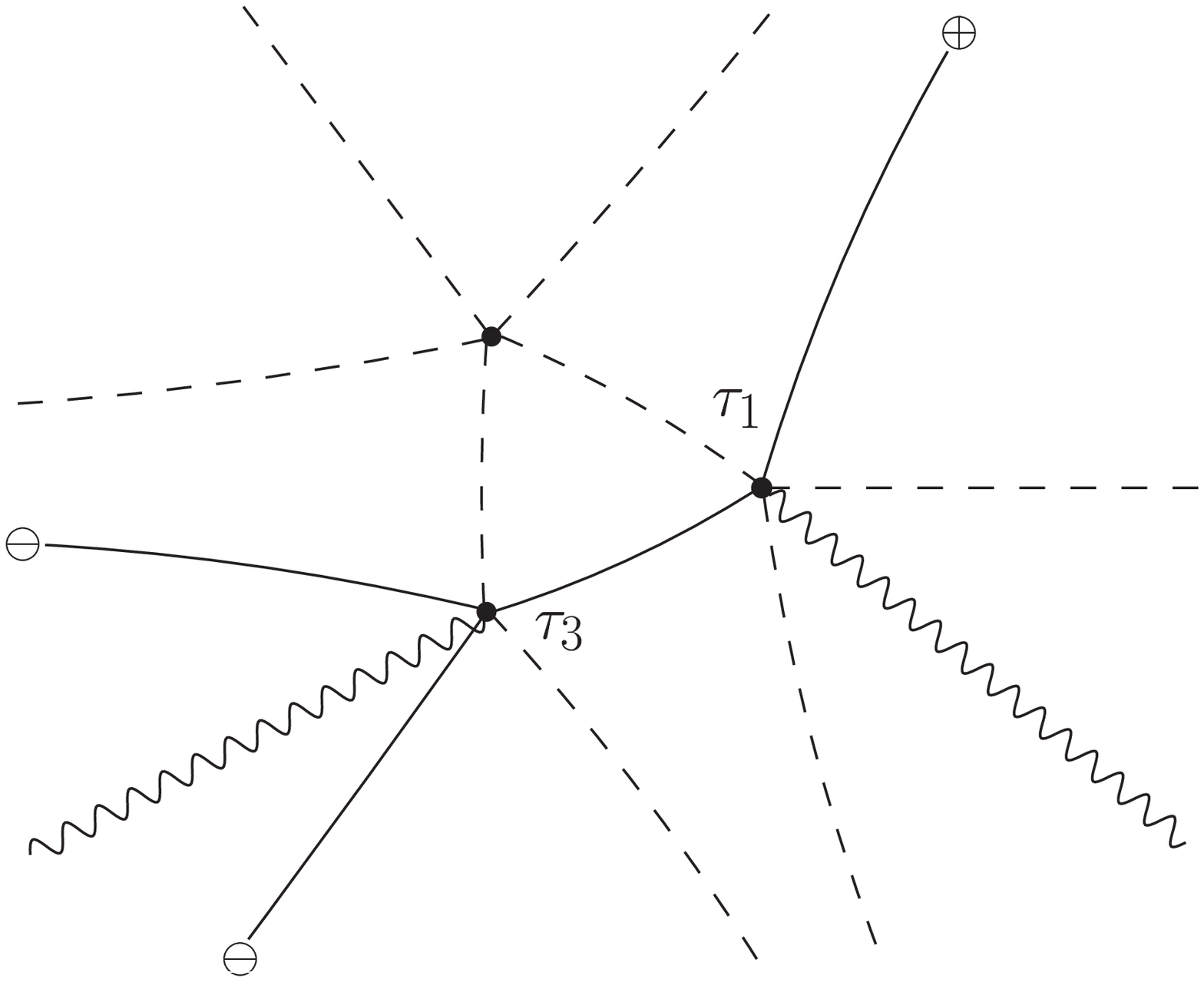}
  \end{center}
  \caption{\small{Domain of $t$.}}
  \label{fig:Riemann surface of sqrt Delta}
  \end{minipage}
  \begin{minipage}{0.32\hsize}
  \begin{center}
  \includegraphics[width=45mm]{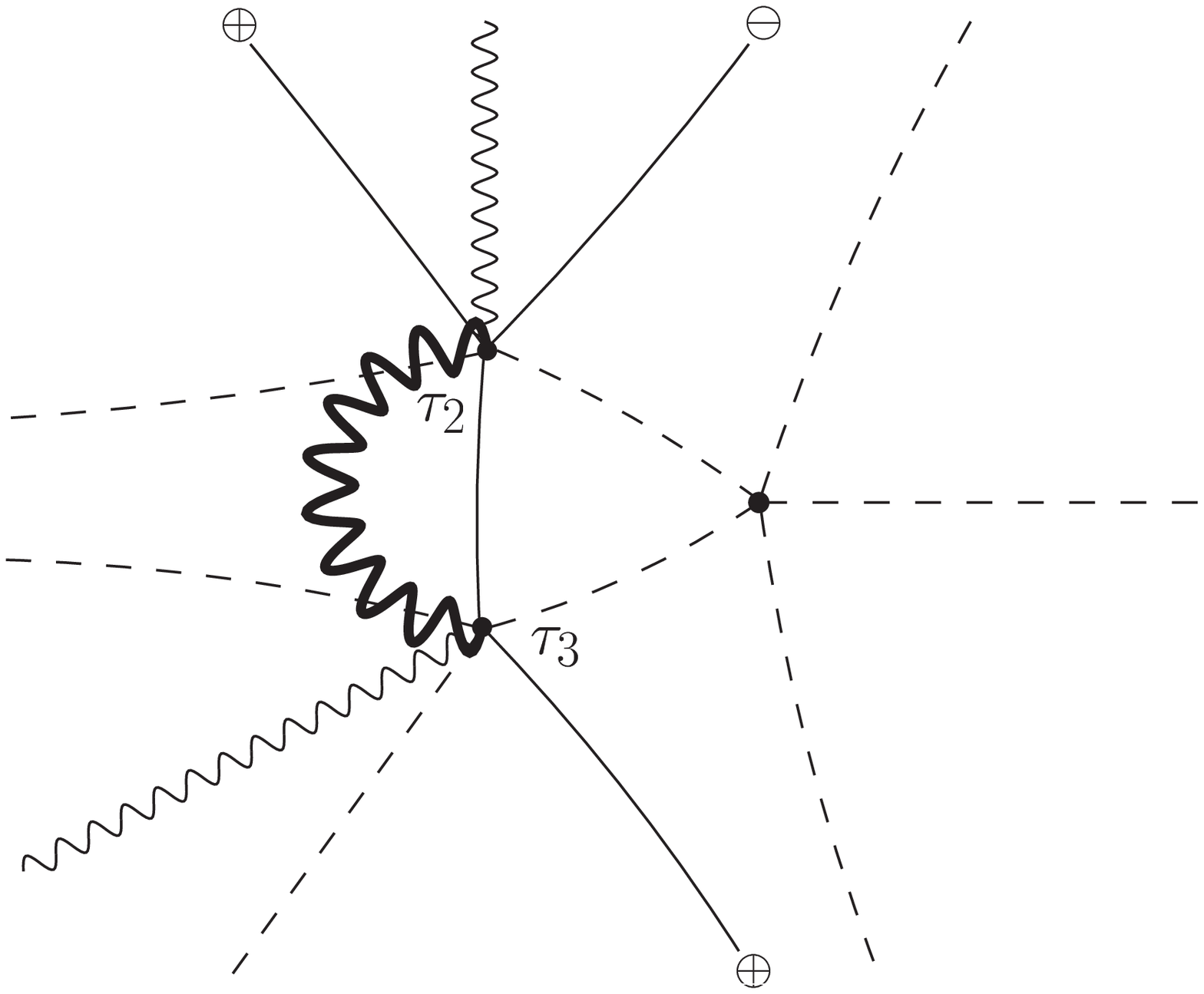}
  \end{center}
  \end{minipage}
  \end{figure}
In Figure \ref{fig:Riemann surface of sqrt Delta} thick wiggly lines designate cuts 
for the determination of the branch of $\sqrt{\Delta}$ 
and the symbols $\oplus$ and $\ominus$ represent the ``sign of $P$-Stokes curves". 
Here the sign of  a $P$-Stokes curve is defined by the sign of 
\[
{\rm Re} \int_{\tau}^t \sqrt{\Delta} \hspace{+.2em} dt ,
\]
where $t$ is a point on the $P$-Stokes curve in question and 
$\tau$ is a $P$-turning point which the $P$-Stokes curve emanates from. 
\end{rem}

As we mentioned in Introduction, 
the degeneration of $P$-Stokes geometry observed when arg \hspace{-.3em} $c = \frac{\pi}{2}$ 
suggests that the ``parametric Stokes phenomenon" occurs when $c$ varies 
near arg \hspace{-.5em} $c = \frac{\pi}{2}$. 
To formulate the connection formula for this Stokes phenomenon,
we define and analyze ``the Voros coefficient of ($P_{\rm II}$)" in the next section. 
The Voros coefficient plays an important role in the analysis 
of the parametric Stokes phenomenon.

The following lemma will be used in the next section. 
\begin{lemm} \label{behavior of coeff of 1-parameter solution}
We have the following asymptotic behaviors when $t \rightarrow \infty$ 
along the $P$-Stokes curve $\Gamma$ in Figure \ref{fig:Riemann surface of sqrt Delta}. 
\begin{eqnarray}
\lambda_0(t,c) & = & 
- \frac{i}{\sqrt{2}} \hspace{+.1em} t^{\frac{1}{2}} + \frac{1}{2} \hspace{+.1em} c \hspace{+.1em} t^{-1}  
- \frac{3 \sqrt{2} \hspace{+.2em} i}{8} \hspace{+.1em} c^2 t^{-\frac{5}{2}} + {O}(t^{-4}) .  
    \label{eq:behavior of lambda_0}  \\ 
\lambda^{(0)}_2(t,c) & = &  
- \frac{\sqrt{2} \hspace{+.2em} i}{16} \hspace{+.1em} t^{-\frac{5}{2}} + {O}(t^{-4}) .  \\
\lambda^{(0)}_{2k}(t,c) & = & {O}(t^{-\frac{11}{2}}) \hspace{+1.em} (k \ge 2) .
\end{eqnarray}
\begin{eqnarray}
\nu^{(0)}_1(t,c) & = & 
- \frac{i}{2\sqrt{2}} \hspace{+.1em} t^{-\frac{1}{2}} - \frac{1}{2} \hspace{+.1em} c \hspace{+.1em} t^{-2} 
+ \frac{15 \sqrt{2} \hspace{+.2em} i}{16} \hspace{+.1em} c^2 t^{-\frac{7}{2}} + {O}(t^{-5}) . \\
\nu^{(0)}_3(t,c) & = &  
\frac{5 \sqrt{2} \hspace{+.2em} i}{32} \hspace{+.1em} t^{-\frac{7}{2}} + {O}(t^{-5}) . \\
\nu^{(0)}_{2k + 1}(t,c) & = & {O}(t^{-\frac{13}{2}}) \hspace{+1.em} (k \ge 2) .
\end{eqnarray}
\begin{eqnarray}
R_{-1}(t,c) & = &  
- \sqrt{2} \hspace{+.1em} i \hspace{+.1em} t^{\frac{1}{2}}  
+ \frac{3}{2} \hspace{+.1em} c \hspace{+.1em} t^{-1}  
- \frac{21 \sqrt{2} \hspace{+.2em} i}{16} \hspace{+.1em} c^2 t^{-\frac{5}{2}} + {O}(t^{-4}) .  
    \label{eq:behavior of R_-1}  \\
R_0(t,c) & = & 
- \frac{1}{4} \hspace{+.1em} t^{-1}  
+ \frac{9 \sqrt{2} \hspace{+.2em} i}{16} \hspace{+.1em} c \hspace{+.1em} t^{-\frac{5}{2}} + {O}(t^{-4})  . \\
R_1(t,c) & = &  
- \frac{17 \sqrt{2} \hspace{+.2em} i}{64} \hspace{+.1em} t^{- \frac{5}{2}} + {O}(t^{-4})  .  \\
R_k(t,c) & = & {O}(t^{-4}) \hspace{+1.em} (k \ge 2)   .
\end{eqnarray}
Here the branch of $t^{\frac{1}{2}}$ is chosen so that 
Re \hspace{-.5em} $t^{\frac{1}{2}} > 0$ holds on $\Gamma$.
\end{lemm}
\begin{proof}
It follows from \eqref{eq:lambda_0} that $\lambda_0$ has the following three possible 
asymptotic behaviors when $t \rightarrow \infty$.
\begin{eqnarray*}
\lambda_0 \sim 
\begin{cases}
\displaystyle  + \frac{i}{\sqrt{2}} \hspace{+.1em} t^{\frac{1}{2}} , \\[+1.em]
\displaystyle  - \frac{i}{\sqrt{2}} \hspace{+.1em} t^{\frac{1}{2}} , \\[+1.em]
\displaystyle  - c \hspace{+.1em} t^{-1} .
\end{cases}
\end{eqnarray*}
Especially, the behavior of $\lambda_0$ when $t \rightarrow \infty$ along 
the $P$-Stokes curve $\Gamma$ is given by 
\[
\lambda_0 \sim - \frac{i}{\sqrt{2}} \hspace{+.1em} t^{\frac{1}{2}} .
\]
Thus we have \eqref{eq:behavior of lambda_0}.
By \eqref{eq:behavior of lambda_0} we fined that 
$R_{-1} = \sqrt{\Delta}$ has two possible asymptotic behaviors below. 
\begin{eqnarray*}
R_{-1} \sim 
\begin{cases}
\displaystyle  + \sqrt{2} \hspace{+.1em} i \hspace{+.1em} t^{\frac{1}{2}} , \\[+.5em]
\displaystyle  - \sqrt{2} \hspace{+.1em} i \hspace{+.1em} t^{\frac{1}{2}} .
\end{cases}
\end{eqnarray*}
Because the sign of the $P$-Stokes curve $\Gamma$ is $\ominus$, 
we have
\[
R_{-1} \sim \hspace{+.1em}   - \sqrt{2} \hspace{+.1em} i \hspace{+.1em} t^{\frac{1}{2}}.
\]
Thus we obtain \eqref{eq:behavior of R_-1}. 
The other asymptotic behaviors are obtained from 
\eqref{eq:behavior of lambda_0}, \eqref{eq:behavior of R_-1} and the recursive relations
(cf. \eqref{eq:recurrence relation of lambda(0)}, \eqref{eq:relation between lambda(0)_k and nu(0)_k} 
and \eqref{eq:R_k+1}).
\end{proof}


\section{The Voros coefficient and the parametric Stokes phenomena of ($P_{\rm II}$)}

To formulate the connection formula for the 1-parameter solutions of ($P_{\rm II}$), 
first we introduce two normalizations of 1-parameter solutions. 
The Voros coefficient of ($P_{\rm II}$) is defined as the difference of these two normalizations.


\subsection{The Voros coefficient of ($P_{\rm II})$}

We introduce two normalizations of the integral 
\[
\int^t R_{\rm odd} dt
\]
in \eqref{eq:lambda(1)tilde}. 
Because the coefficients $R_{2k - 1}$ of $\eta^{-(2k - 1)}$ in $R_{\rm odd}$
have a singularity of the form $(t - \tau_1)^{- \frac{l}{4}}$ 
(where $l$ is an odd integer) at a $P$-turning point $t = \tau_1$, 
we can define the integral of $R_{\rm odd}$ 
from $t = \tau_1$ as a contour integral:
\begin{eqnarray}
\tilde{\lambda}_{\tau_1}^{(1)}(t,c,\eta;\alpha) & = & 
 \alpha \frac{1}{\sqrt{R_{\rm{odd}}}} \hspace{+.2em} 
{\rm{exp}} \biggl( \int_{\tau_1}^{t} R_{\rm{odd}} \hspace{+.2em} dt \biggr) , 
\label{eq:normalized at tau1} \\
{\lambda}_{\tau_1}^{(1)}(t,c,\eta) & = & 
\Delta^{-\frac{1}{4}} + \eta^{-1} \Delta^{- \frac{1}{4}} 
\Bigl( \int_{\tau_1}^{t} R_1 \hspace{+.em} dt \Bigr) + \cdots .
\label{eq:normalized at tau1-2}
\end{eqnarray}
Here the integral $\int_{\tau_1}^{t} R_{\rm{odd}} \hspace{+.2em} dt$ in \eqref{eq:normalized at tau1} 
is defined by $\frac{1}{2} \int_{\Gamma_t} R_{\rm{odd}} \hspace{+.2em} dt$, 
where $\Gamma_t$ is a path on the Riemann surface of $\sqrt{\Delta}$ shown 
in Figure \ref{fig:normalization path of lambda(1)_tau_1} and 
  \begin{figure}[h]
  \begin{minipage}{0.32\hsize}
  \begin{center}
  \includegraphics[width=48mm]{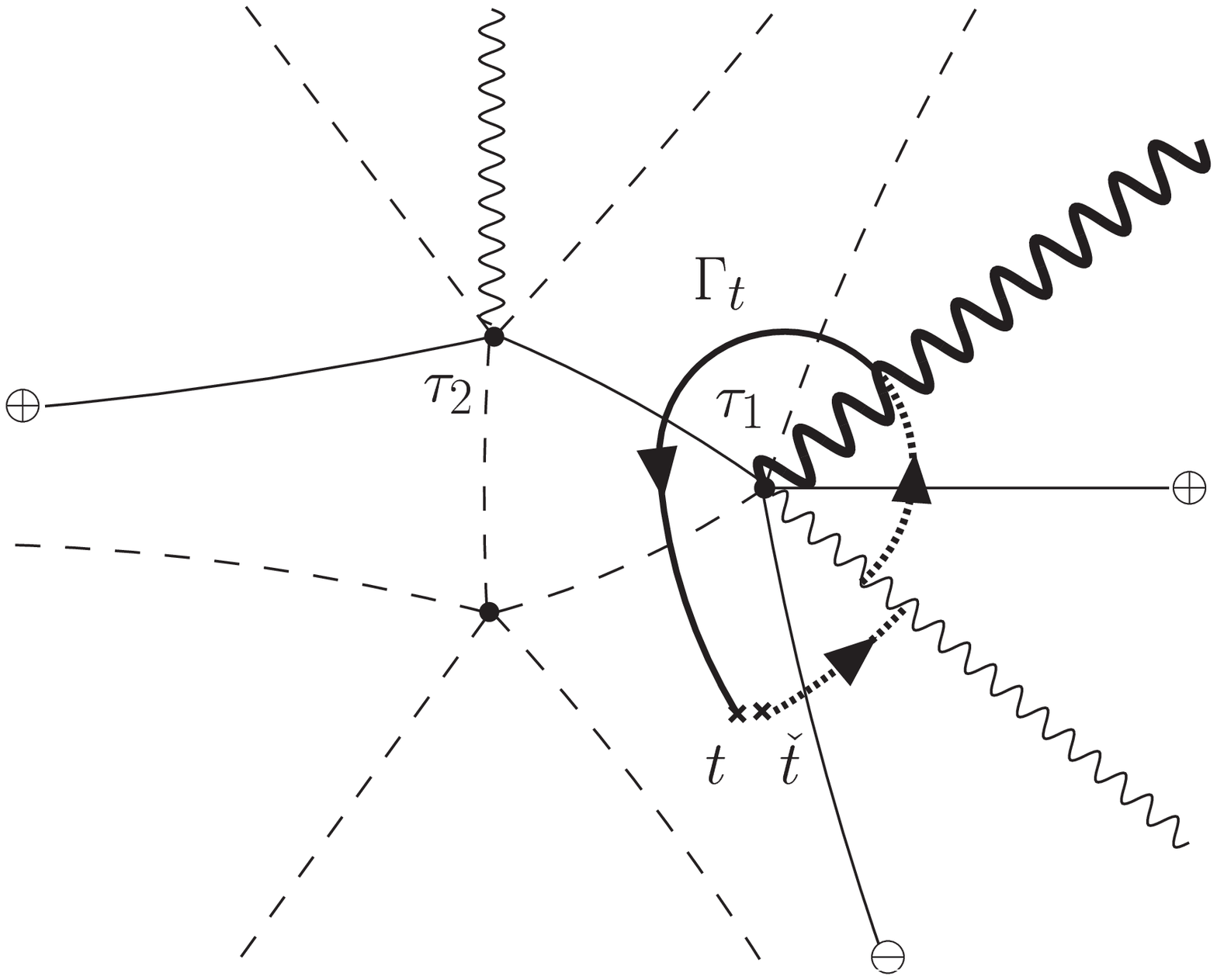}
  \end{center}
  \end{minipage}
  \begin{minipage}{0.32\hsize}
  \begin{center}
  \includegraphics[width=48mm]{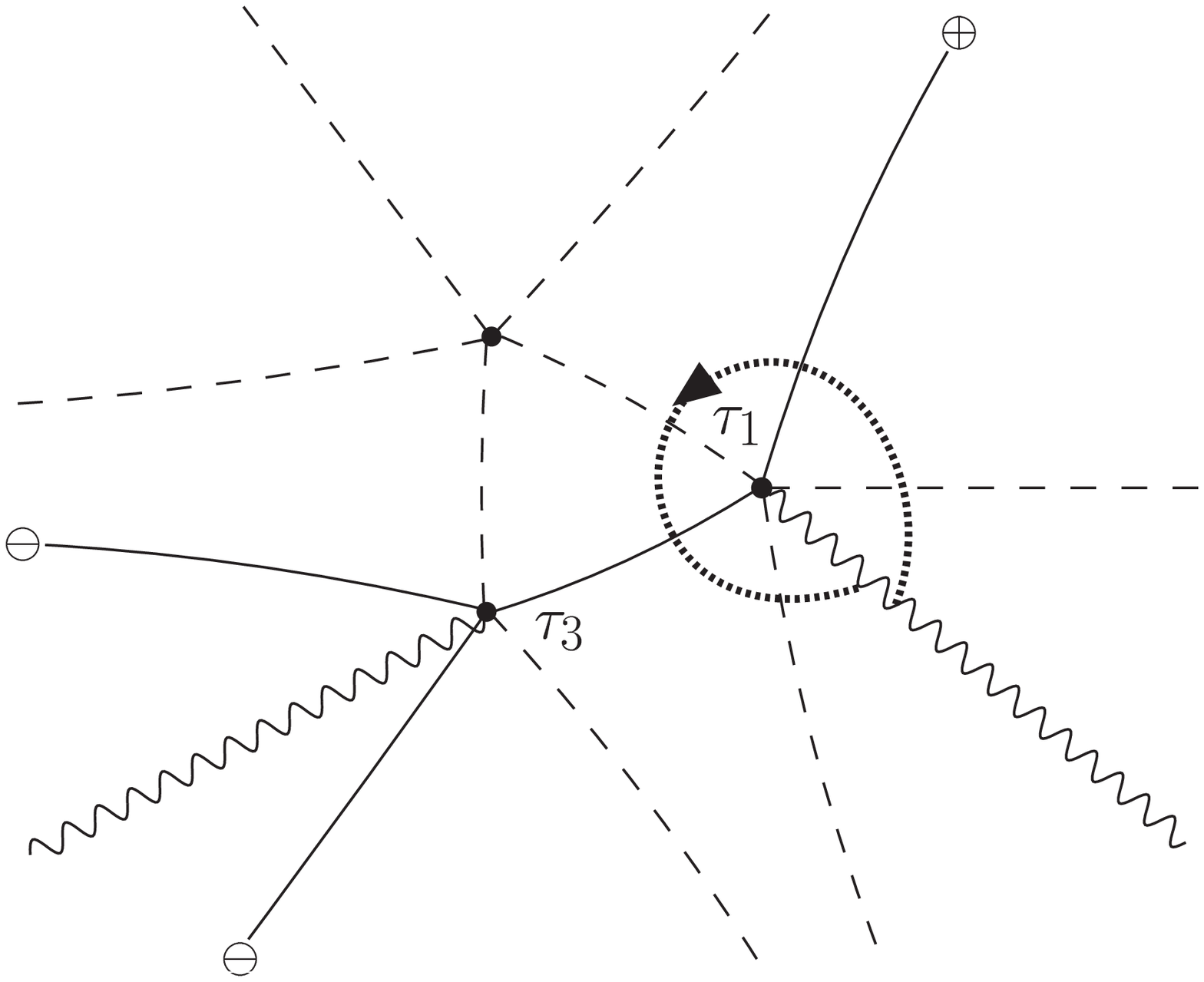}
  \end{center}
  \caption{Normalization path $\Gamma_t$ of $\tilde{\lambda}^{(1)}_{\tau_1}$.}
  \label{fig:normalization path of lambda(1)_tau_1}
  \end{minipage}
  \begin{minipage}{0.32\hsize}
  \begin{center}
  \includegraphics[width=48mm]{P-Stokes-Curve-sheet3-2.eps}
  \end{center}
  \end{minipage}
  \end{figure}
$\check{t}$ represents a point on the Riemann surface of $\sqrt{\Delta}$ 
satisfying that $\lambda_0(\check{t},c) = \lambda_0(t,c)$ and 
$\sqrt{\Delta(\check{t},c)} = - \sqrt{\Delta(t,c)}$. 
(The dotted part of $\Gamma_t$ represents a path on the other sheet of the Riemann surface 
of $\sqrt{\Delta}$.)
On the other hand, since $R_{2k - 1}$ ($k \ge 1$) are integrable at $t = \infty$ 
by Lemma \ref{behavior of coeff of 1-parameter solution},  
the following integral is well-defined: 
\begin{eqnarray}
\tilde{\lambda}_{\infty}^{(1)}(t,c,\eta;\alpha) & = & \alpha \frac{1}{\sqrt{R_{\rm{odd}}}} 
\hspace{+.2em} {\rm{exp}} \biggl( \eta \int_{\tau_1}^{t} R_{-1} \hspace{+.2em} dt  
+ \int_{\infty}^{t} \bigl( R_{\rm{odd}} - \eta R_{-1} \bigr) \hspace{+.2em} dt \biggr) , 
\label{eq:normalized at infinity}  \\
\lambda_{\infty}^{(1)}(t,c,\eta) & = & 
\Delta^{- \frac{1}{4}} + \eta^{-1} \Delta^{- \frac{1}{4}} 
\Bigl( \int_{\infty}^{t} R_1 \hspace{+.2em} dt \Bigr) + \cdots , 
\label{eq:normalized at infinity-2}
\end{eqnarray}
where the path of integral from infinity is taken as in 
Figure \ref{fig:normalization path of lambda(1)_infinity}. 
 \begin{figure}[h]
 \begin{center}
 \includegraphics[width=50mm]{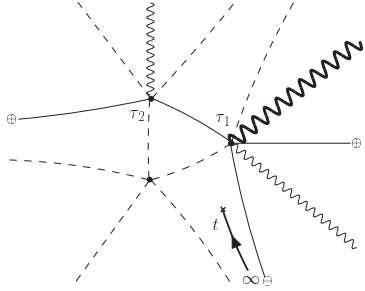}
 \end{center}
 \caption{Normalization path of $\tilde{\lambda}^{(1)}_{\infty}$.}
 \label{fig:normalization path of lambda(1)_infinity}
 \end{figure}
As we noted in Section 2, once the normalization of $\tilde{\lambda}^{(1)}$ is fixed, 
a 1-parameter solution is uniquely determined. We define the 1-parameter solution 
$\lambda_{\tau_1}(t,c,\eta;\alpha)$(resp. $\lambda_{\infty}(t,c,\eta;\alpha)$) 
by using $\tilde{\lambda}_{\tau_1}^{(1)}$(resp. $\tilde{\lambda}_{\infty}^{(1)}$)
for the normalization of $\tilde{\lambda}^{(1)}$.
From now on we consider 1-parameter solutions normalized at 
either $t = \tau_1$ or $t = \infty$. 
Therefore, $\phi_{\rm II}$ is always normalized as
\[
\phi_{\rm II} = \int_{\tau_1}^{t} \sqrt{\Delta} \hspace{+.2em} dt.
\]

Next, we define the Voros coefficient of ($P_{\rm II}$). 
There is a relation between the above two normalizations of $\tilde{\lambda}^{(1)}$: 
\begin{equation}
\tilde{\lambda}_{\tau_1}^{(1)} = {{e}}^W \hspace{+.2em} \tilde{\lambda}_{\infty}^{(1)} , 
\label{eq:relation for two normalizations}
\end{equation}
where
\begin{eqnarray}
W = W(c,\eta) & = & \int_{\tau_1}^{\infty} 
\bigl(R_{\rm{odd}}(t,c,\eta) - \eta \hspace{+.1em} R_{-1}(t,c) \bigr) dt 
\label{eq:PIIVoros}  \\
& = & 
\frac{1}{2} \int_{\Gamma_{\infty}} 
\bigl(R_{\rm{odd}}(t,c,\eta) - \eta \hspace{+.1em} R_{-1}(t,c) \bigr) dt .
\nonumber
\end{eqnarray}
Here $\Gamma_{\infty}$ is a path on the Riemann surface of $\sqrt{\Delta}$ 
shown in Figure \ref{fig:integral path of Gamma_infinity}. 

 \begin{figure}[h]
  \begin{minipage}{0.32\hsize}
  \begin{center}
  \includegraphics[width=48mm]{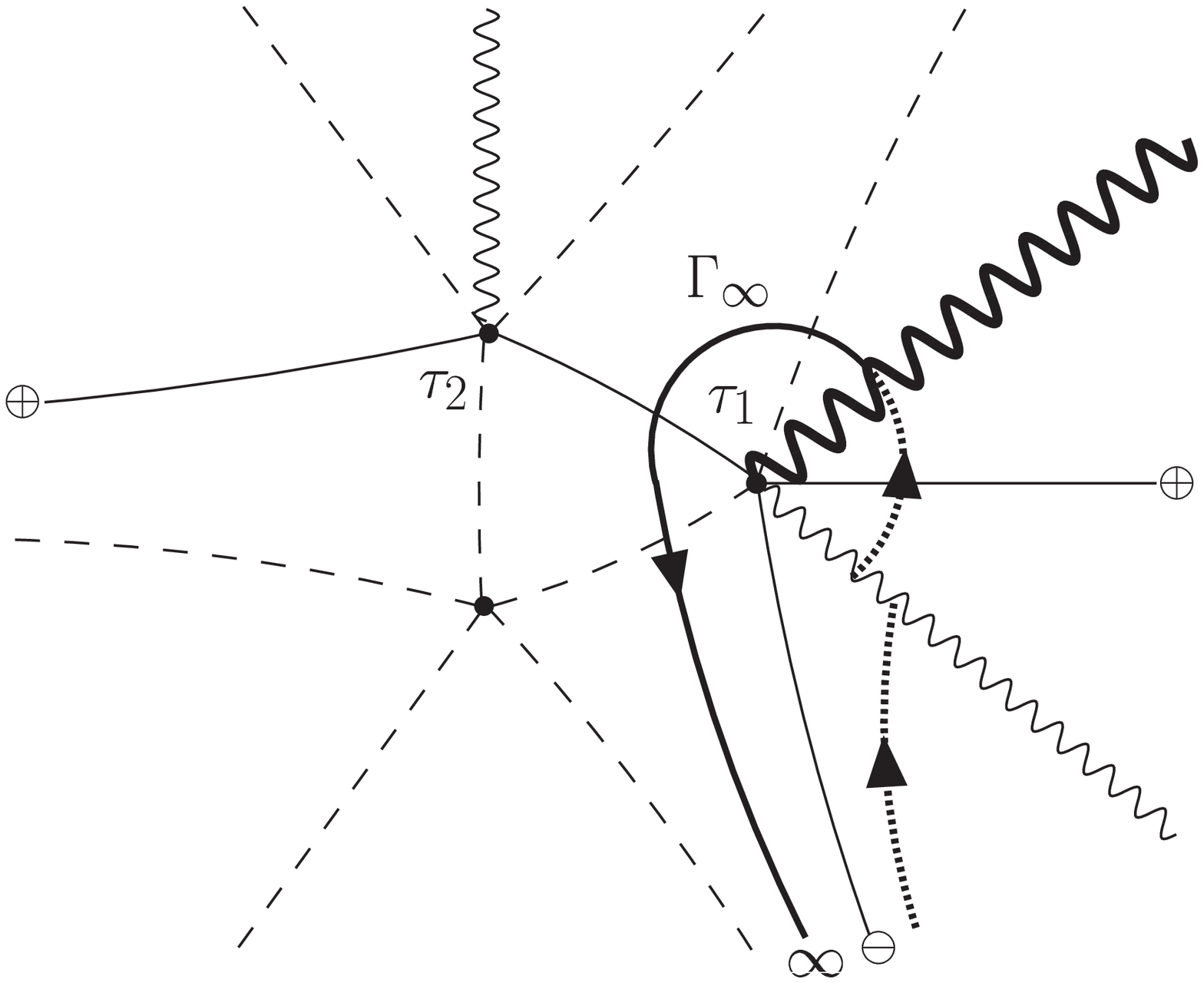}
  \end{center}
  \end{minipage}
  \begin{minipage}{0.32\hsize}
  \begin{center}
  \includegraphics[width=48mm]{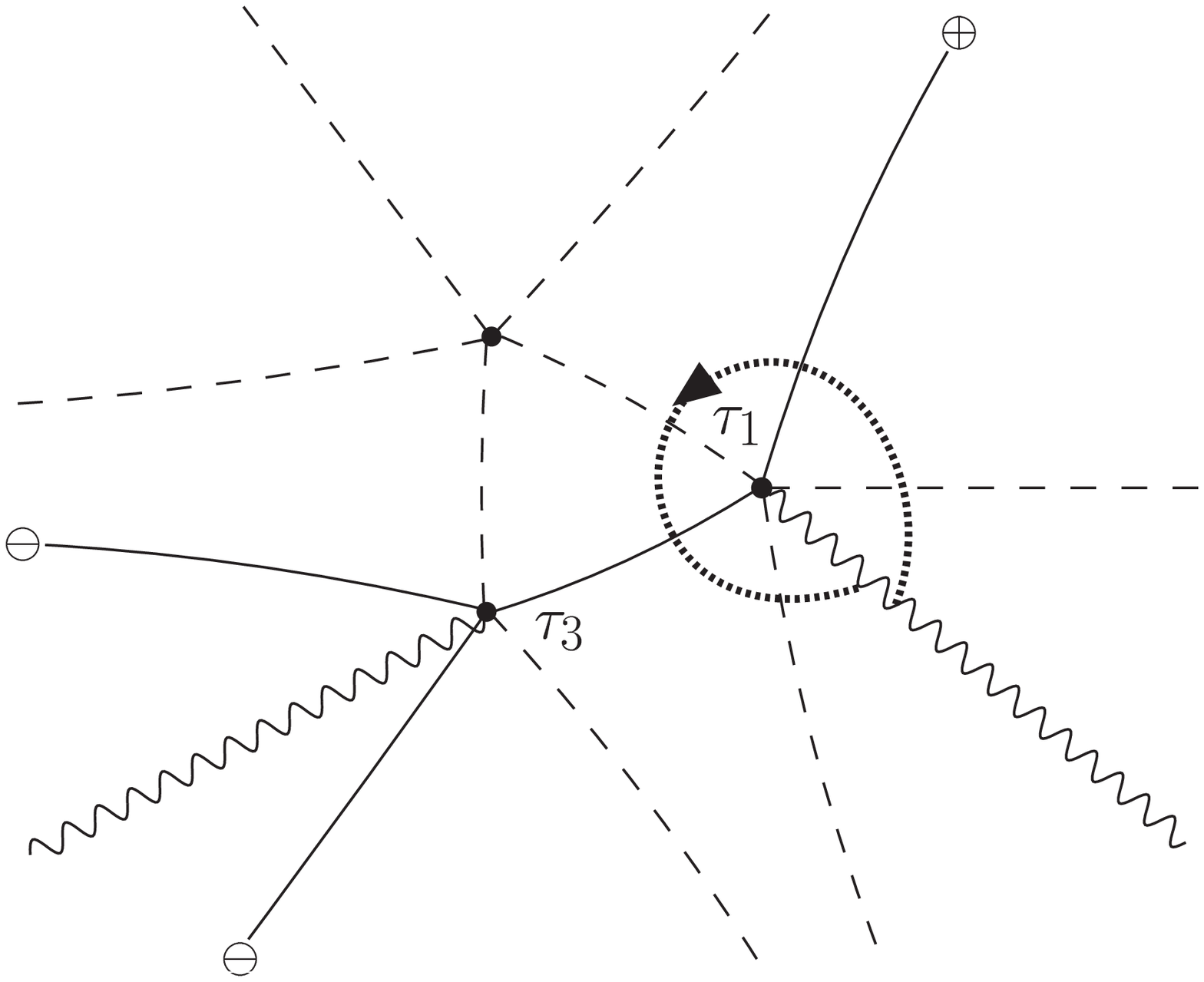}
  \end{center}
  \caption{{Integration path $\Gamma_{\infty}$.}}
  \label{fig:integral path of Gamma_infinity}
  \end{minipage}
  \begin{minipage}{0.32\hsize}
  \begin{center}
  \includegraphics[width=48mm]{P-Stokes-Curve-sheet3-2.eps}
  \end{center}
  \end{minipage}
  \end{figure}

\begin{defi} \label{definition of the P-Voros coefficient}
$W(c,\eta)$ 
{\rm defined by \eqref{eq:PIIVoros} is called} 
the Voros coefficient of ($P_{\rm{II}}$), 
{\rm or the} $P$-Voros coefficient  {\rm for short}.
\end{defi}
\begin{rem} \label{relation between two normalizations}
\normalfont
By the uniqueness of $\lambda^{(k)}$ ($k \ge 2$), 
the relation \eqref{eq:relation for two normalizations} induces the relation 
between the two normalizations of 1-parameter solutions as follows:
\begin{equation}
\lambda_{\tau_1} (t,c,\eta;\alpha) = \lambda_{\infty} (t,c,\eta;\alpha \hspace{+.1em} {{e}}^{W}) .  
\label{eq:relation between two normalizations}
\end{equation}
\end{rem}


\subsection{Determination of the Voros coefficient of ($P_{\rm II}$)}

We have an explicit description of the $P$-Voros coefficient. 
\begin{theo} \label{main theorem 1} 
The $P$-Voros coefficient \eqref{eq:PIIVoros} is represented as follows:
\begin{equation}
W(c,\eta) 
 = - \sum_{n = 1}^{\infty}  
 \frac{2^{1-2n}-1}{2n(2n-1)}B_{2n} (c \hspace{+.1em} \eta)^{1 - 2n} ,
\label{eq:expression of PIIVoros}
\end{equation}
where $B_{2n}$ is the 2n-th Bernoulli number defined by \eqref{eq:Bernoulli number}.
\end{theo}

In what follows we prove this theorem by using the idea of \cite{Takei Sato conjecture}.
\begin{lemm} [{\cite[Lemma 1.2]{Takei Sato conjecture}}] \label{lemma for Voros coeff}
{(i)} The formal power series
\begin{equation}
F(c,\eta) = - \sum_{n = 1}^{\infty}  
 \frac{2^{1-2n}-1}{2n(2n-1)}B_{2n} (c \hspace{+.1em} \eta)^{1 - 2n}  
\label{eq:WeberVoros2}
\end{equation}
satisfies the following difference equation formally:
\begin{equation}
F(c,\eta) - F(c - \eta^{-1},\eta) = 
- 1 + c \hspace{+.1em} \eta \hspace{+.2em}
{\rm{log}} 
\Bigl( 1+\frac{1}{c \hspace{+.1em} \eta-1} \Bigr) - 
{\rm{log}}
\Bigl( 1+\frac{1}{2(c \hspace{+.1em} \eta - 1)} \Bigr). 
\label{eq:difference equation}
\end{equation}
{(ii)} Conversely, if $F(c,\eta) = \sum_{n = 1}^{\infty} F_n (c \hspace{+.1em} \eta)^{-n}$ 
where $F_n \in {\mathbb C}$ is independent of $c$ is a formal solution of \eqref{eq:difference equation}, 
then it coincides with \eqref{eq:WeberVoros2}.
\end{lemm}

\noindent
By \eqref{eq:homogenity of W} in Appendix, 
\begin{equation}
W(r^{-1}c,r \eta) = W(c,\eta)
\end{equation}
holds for all positive real numbers $r$. Thus $W(c,\eta)$ is written in the form 
$\sum_{n = 1}^{\infty} W_n (c \hspace{+.1em} \eta)^{-n}$ and $W_n$ ($n \ge 1$) is independent of $c$ .
Hence it is sufficient to show that $W(c,\eta)$ satisfies the difference equation 
\eqref{eq:difference equation}. 

\begin{lemm}\label{expression of W} 
\[
W(c,\eta) = \frac{1}{2} \int_{\Gamma_{\infty}}
\bigl( R(t,c,\eta) - \eta R_{-1}(t,c) - R_0(t,c) \bigr) dt  .
\]
\end{lemm}
\begin{proof}
We have to prove that the following relation holds:
\[
\int_{\Gamma_{\infty}} R_{\rm{even}} (t,c,\eta) \hspace{+.2em} dt 
 =  \int_{\Gamma_{\infty}} R_0 (t,c) \hspace{+.2em} dt .
\]
By the relation \eqref{eq:R_odd and R_even}, we obtain
\begin{eqnarray*}
R_{\rm{even}} & = & - \frac{1}{2} \frac{d}{dt} \hspace{+.2em} \rm{log} R_{\rm{odd}} \\
& = & - \frac{1}{2 R_{-1}} \hspace{+.1em} \frac{d R_{-1}}{dt} 
 - \frac{1}{2} \hspace{+.1em} \frac{d}{dt} \hspace{+.2em} {\rm{log}}  
\Bigl(1 + \frac{R_{\rm{odd}}-\eta R_{-1}}{\eta R_{-1}} \Bigr) .
\end{eqnarray*}
Since the first term of the right-hand side coincides with $R_0$ 
and the integral of the second term along $\Gamma_{\infty}$ vanishes,  
we obtain the desired relation.
\end{proof}

\noindent
By Lemma \ref{expression of W}, if we define
\begin{eqnarray*}
I(t,c,\eta) & = & \int_{\Gamma_{t}}R(t,c,\eta) \hspace{+.2em} dt 
 - \int_{\Gamma_{t}}R(t,c-\eta^{-1},\eta) \hspace{+.2em} dt  \hspace{+.2em},  \label{eq:I}  \\
I_{j}(t,c,\eta) & = & \int_{\Gamma_{t}}R_{j}(t,c) \hspace{+.2em} dt  
 - \int_{\Gamma_{t}}R_{j}(t,c-\eta^{-1}) \hspace{+.2em} dt   \hspace{+1.em}(j \ge -1)  \hspace{+.2em},
        \label{eq:I_j}
\end{eqnarray*}
then we obtain
\begin{equation}
W(c,\eta) - W(c - \eta^{-1},\eta) 
 = \frac{1}{2} 
  \lim_{t \rightarrow \infty} \bigl( I(t,c,\eta) - \eta \hspace{+.1em} I_{-1}(t,c,\eta) - I_{0}(t,c,\eta) \bigr) , 
\label{eq:difference limit}
\end{equation}
where the limit is taken along the $P$-Stokes curve $\Gamma$  
in Figure \ref{fig:Riemann surface of sqrt Delta}. 
To calculate the right-hand side of \eqref{eq:difference limit}, we employ the so-called 
B$\ddot{\rm a}$cklund transformation which induces the translation of the parameter 
$c \mapsto c - \eta^{-1}$. 

\begin{lemm} [{\cite[pp.423-424]{Jimbo-Miwa}}] \label{Backlund}
Let $(\lambda,\nu)$ be a solution of $(H_{\rm{II}})$.
If we define
\begin{eqnarray}
\left\{
\begin{array}{ll}
\displaystyle \Lambda(\lambda,\nu) = 
 -\lambda + \frac{c - \frac{1}{2}\eta^{-1}}{\nu - \lambda^2 - \frac{1}{2}t} , \\[+1.4em]
\displaystyle {\cal N}(\lambda,\nu) = - \nu + 
 \frac{2(c - \frac{1}{2}\eta^{-1})\lambda}{\nu - \lambda^2 - \frac{1}{2}t} -
 \Bigl( \frac{c - \frac{1}{2}\eta^{-1}}{\nu - \lambda^2 - \frac{1}{2}t} \Bigr)^2 ,
\end{array} \right.  
\label{eq:Backlund}
\end{eqnarray}
then $\bigl( \Lambda, {\cal N} \bigr) = \bigl( \Lambda(\lambda,\nu),{\cal N}(\lambda,\nu) \bigr)$ 
satisfies the following Hamiltonian system: 
\begin{eqnarray}
\left\{
\begin{array}{ll}
\displaystyle 
\frac{d \Lambda}{dt} = \eta \hspace{+.2em} {\cal N} , \\[+1.4em]
\displaystyle 
\frac{d{\cal N}}{dt} = \eta \hspace{+.2em} 
(2 \Lambda^3 + t \Lambda + c - \eta^{-1}) . 
\end{array} \right.  
\label{eq:translated Hamiltonian system}
\end{eqnarray}
Thus $\Lambda = \Lambda(\lambda, \nu)$ satisfies 
\begin{eqnarray}
\frac{d^2 \Lambda}{dt^2} = \eta^2 
(2 \Lambda^3 + t \Lambda + c - \eta^{-1}). \label{eq:translated PII}
\end{eqnarray}
\end{lemm}

\noindent
This lemma is easily confirmed by straightforward computations. 
With the aid of Lemma \ref{Backlund}, we can calculate the difference 
$R(t,c,\eta) - R(t,c - \eta^{-1},\eta)$ in the following manner.

\begin{lemm} \label{application of Backlund transformation}
Let $(\lambda^{(0)}, \nu^{(0)})$ = $(\lambda^{(0)}(t,c,\eta), \nu^{(0)}(t,c,\eta))$
be a 0-parameter solution of $(H_{\rm{II}})$.

\noindent 
{(i)} 
\begin{eqnarray}
\Lambda \bigl( \lambda^{(0)}(t,c,\eta),\nu^{(0)}(t,c,\eta) \bigr)  
& = & \lambda^{(0)}(t,c-\eta^{-1},\eta) , \label{eq:Backlund lambda(0)}  \\
{\cal N} \bigl( \lambda^{(0)}(t,c,\eta),\nu^{(0)}(t,c,\eta) \bigr) 
& = & \nu^{(0)}(t,c-\eta^{-1},\eta) . \label{eq:Backlund nu(0)} 
\end{eqnarray}

\noindent
{(ii)} The formal solution $R(t,c,\eta)$ of \eqref{eq:R} satisfies the following:
\begin{equation}
R(t,c,\eta) - R(t,c-\eta^{-1},\eta) = 
 - \frac{d}{dt} \hspace{+.2em} {\rm{log}}
 \Biggl\{
 1 + \Bigl(c-\frac{1}{2}\eta^{-1} \Bigr) 
 \frac{\eta^{-1}R(t,c,\eta)-2\lambda^{(0)}(t,c,\eta)}
{(\nu^{(0)}(t,c,\eta) - \lambda^{(0)}(t,c,\eta)^2 - \frac{1}{2} t)^2}  
 \Biggr\} .
\label{eq:difference equation for R}
\end{equation}
\end{lemm}
\begin{proof}
(i) By Lemma \ref{Backlund},  
$\bigl( \Lambda \bigl( \lambda^{(0)},\nu^{(0)} \bigr), {\cal N} \bigl( \lambda^{(0)}, \nu^{(0)} \bigr) \bigr)$ 
is a formal power series solution of \eqref{eq:translated Hamiltonian system}, 
so it is a 0-parameter solution of \eqref{eq:translated Hamiltonian system}. 
On the other hand, $\bigl( \lambda^{(0)}(t,c - \eta^{-1},\eta)$, $\nu^{(0)}(t,c - \eta^{-1},\eta) \bigr)$ 
is also a 0-parameter solution of \eqref{eq:translated Hamiltonian system}.
Because it follows from \eqref{eq:lambda_0} and \eqref{eq:nu_0} that 
the leading terms of $\Lambda \bigl( \lambda^{(0)},\nu^{(0)} \bigr)$ 
and $\lambda^{(0)}(t,c - \eta^{-1},\eta)$ both coincide with $\lambda_0(t,c)$ ,  
we obtain the relations \eqref{eq:Backlund lambda(0)} and \eqref{eq:Backlund nu(0)} 
due to the uniqueness of the coefficients of $\eta^{-k}$ ($k \ge 1$) 
of 0-parameter solutions of \eqref{eq:translated Hamiltonian system}.

\noindent
(ii) We apply the B$\ddot{\rm a}$cklund transformation \eqref{eq:Backlund} 
to a 1-parameter solution ($\lambda(t,c,\eta;\alpha)$, $\nu(t,c,\eta;\alpha)$) of $(H_{\rm II})$. 
Since $\Lambda(\lambda,\nu)$ is expanded as
\[
\Lambda(\lambda,\nu) =  
\Lambda(\lambda^{(0)},\nu^{(0)})
+ 
\alpha \eta^{-\frac{1}{2}}
\Biggl\{ 
-\lambda^{(1)} -(c-\frac{1}{2}\eta^{-1}) 
\frac{\nu^{(1)} - 2 \lambda^{(0)} \lambda^{(1)}}
{(\nu^{(0)} - {\lambda^{(0)}}^2 - \frac{1}{2}t)^2}
\Biggr\} {{e}}^{\eta \phi_{\rm{II}}} + \cdots ,
\]
($\Lambda(\lambda,\nu)$, ${\cal N}(\lambda,\nu) $) is a 1-parameter solution 
of \eqref{eq:translated Hamiltonian system}.  
Hence
\[
\varphi^{(1)} = 
\alpha \eta^{-\frac{1}{2}}
\Biggl\{ 
-\lambda^{(1)} -(c-\frac{1}{2}\eta^{-1}) 
\frac{\nu^{(1)} - 2 \lambda^{(0)} \lambda^{(1)}}
{(\nu^{(0)} - {\lambda^{(0)}}^2 - \frac{1}{2}t)^2}
\Biggr\} {{e}}^{\eta \phi_{\rm{II}}}
\]
is a WKB solution of the linear differential equation obtained as the
Fr$\acute{\rm e}$chet derivative of \eqref{eq:translated PII}  
at $\Lambda = \Lambda(\lambda^{(0)}, \nu^{(0)}) = \lambda^{(0)}(t,c-\eta^{-1},\eta)$;
\[
\biggl( 
\frac{d^2}{d t^2} - \eta^2 \bigl(6 {\lambda^{(0)}(t,c-\eta^{-1},\eta)}^2 + t \bigr)  
\biggr) \varphi^{(1)} = 0 .
\]
This implies that $\varphi^{(1)}$ can be written as  
$\alpha \eta^{-\frac{1}{2}} C(\eta) \hspace{+.2em} {\rm{exp}}  
\Bigl( \int^{t} R(t,c-\eta^{-1},\eta) dt  \Bigr)$  
for some formal power series $C(\eta)$ of $\eta^{-1}$ whose coefficients
are independent of $t$. 
On the other hand, by \eqref{eq:nu(1)}, $\varphi^{(1)}$ is also expressed as 
\[
\alpha \eta^{-\frac{1}{2}} \tilde{C}(\eta) \hspace{+.2em} 
\Biggl\{ 
1 + \bigl( c-\frac{1}{2}\eta^{-1} \bigr) 
\frac{\eta^{-1} R(t,c,\eta) - 2 \lambda^{(0)}(t,c,\eta)}
{(\nu^{(0)}(t,c,\eta) - {\lambda^{(0)}(t,c,\eta)}^2 - \frac{1}{2}t)^2}
\Biggr\} {\rm{exp}}\Bigl( \int^{t} R(t,c,\eta) dt  \Bigr)   
\]
for some formal series $\tilde{C}(\eta)$ of $\eta^{-1}$ whose coefficients are independent of $t$. 
Therefore, by taking the logarithmic derivative with respect to $t$ 
of these two different expressions of $\varphi^{(1)}$, 
we obtain the relation \eqref{eq:difference equation for R}.
\end{proof}

\noindent
By \eqref{eq:difference equation for R}, $I(t,c,\eta)$ is represented as
\begin{eqnarray*}
 I(t,c,\eta) & = &  
 - {\rm{log}}
\Biggl\{
 1 + \Bigl(c-\frac{1}{2}\eta^{-1} \Bigr) 
 \frac{\eta^{-1}R(t,c,\eta)-2\lambda^{(0)}(t,c,\eta)}
{(\nu^{(0)}(t,c,\eta) - \lambda^{(0)}(t,c,\eta)^2 - \frac{1}{2} t)^2}
\Biggr\} \nonumber \\
&  & +
{\rm{log}}
\Biggl\{
 1 + \Bigl(c-\frac{1}{2}\eta^{-1} \Bigr) 
 \frac{\eta^{-1}R(\check{t},c,\eta)-2\lambda^{(0)}(t,c,\eta)}
{(\nu^{(0)}(t,c,\eta) - \lambda^{(0)}(t,c,\eta)^2 - \frac{1}{2} t)^2}
\Biggr\} .
\end{eqnarray*}
Using Lemma \ref{behavior of coeff of 1-parameter solution} in Section 2 
and $R_k(\check{t},c) = (-1)^k R_k(t,c)$ ($k \ge -1$), 
we then obtain the following asymptotic behavior of $I(t,c,\eta)$ when 
taking the limit $t \rightarrow \infty$ along the 
$P$-Stokes curve $\Gamma$ in Figure \ref{fig:Riemann surface of sqrt Delta}:
\begin{equation}
I(t,c,\eta) =
- {\rm{log}} \Bigl(-\frac{(1 - 2 c\eta)^2}{128} \eta^{-2} t^{-3} \Bigl) + {O}(t^{-1}) .
\label{eq:behavior of I}
\end{equation}
Because $R_{-1}$ can be integrated explicitly as
\[
\int_{\Gamma_{t}} R_{-1}(t,c) dt = \frac{4}{3} t \sqrt{\Delta(t,c)}  
- 2 c \hspace{+.2em} {\rm{log}} 
\Biggl\{
 \frac{2 \lambda_0(t,c) - \sqrt{\Delta(t,c)}}{2 \lambda_0(t,c) + \sqrt{\Delta(t,c)}} 
\Biggr\} ,
\]
we also have the following asymptotic behavior of $I_{-1}(t,c,\eta)$:
\begin{equation}
\eta \hspace{+.1em} I_{-1} (t,c,\eta) = 2  
+ 2 c \eta \hspace{+.2em} {\rm{log}}
\Bigl(
\frac{c \eta - 1}{c \eta}
\Bigr)
 + 
\hspace{+.em} {\rm{log}}
\Bigl(
- \frac{32}{(c \eta - 1)^2} \eta^2 t^3
\Bigr)
+
{O}(t^{-1}) .
\label{eq:behavior of I_-1}
\end{equation}
Furthermore, since  $R_{0} = - \frac{1}{2} \frac{1}{R_{-1}} \frac{d R_{-1}}{dt}$ 
has a singularity of the form $R_0 \sim - \frac{1}{8} \hspace{+.1em} (t - \tau_1)^{-1} 
$ at $t = \tau_1$, we obtain
\[
\int_{\Gamma_{t}} R_0 \hspace{+.2em} dt  = 
4 \pi i \hspace{+.2em} {\rm Res}_{t = \tau_1} R_0  =  - \frac{\pi i}{2} .  
\]
This implies that
\begin{equation}
I_0(t,c,\eta) = 0 .
\label{eq:behavior of I_0}
\end{equation}
Making use of \eqref{eq:difference limit}, \eqref{eq:behavior of I}, 
\eqref{eq:behavior of I_-1} and \eqref{eq:behavior of I_0} 
and taking the limit $t \rightarrow \infty$, 
we thus obtain the following. 

\begin{prop} \label{proposition for Voros coeff}
The $P$-Voros coefficient $W(c,\eta)$ is a formal solution of the 
following difference equation:
\[
W(c,\eta) - W(c - \eta^{-1},\eta) = - 1 + 
c\eta \hspace{+.2em} \rm{log}
\Bigl( 1 + \frac{1}{c \eta - 1} \Bigr)
- {\rm{log}}
\Bigl( 1 + \frac{1}{2(c\eta - 1)} \Bigr) .
\]
\end{prop}

\noindent
By Lemma \ref{lemma for Voros coeff} and Proposition \ref{proposition for Voros coeff}, 
the proof of Theorem \ref{main theorem 1} is completed.  $\Box$
\\[-.5em]

Using Theorem \ref{main theorem 1}, we can explicitly analyze the parametric Stokes phenomenon 
for the $P$-Voros coefficient  which occurs when arg \hspace{-.5em} $c$ varies near $\frac{\pi}{2}$.
\begin{cor} \label{connection formula for PII Voros}
By denoting the Borel resummation operator by ${\cal S}$,    
we obtain the following:
\begin{equation}
{\cal S}\bigl[ {{e}}^{W(c,\eta)}\bigl|_{{\rm{arg}}c = \frac{\pi}{2} - \varepsilon} \bigr] = 
(1 + {{e}}^{2 \pi i c \eta}) \hspace{+.2em}
{\cal S}\bigl[ {{e}}^{W(c,\eta)}\bigl|_{{\rm{arg}}c = \frac{\pi}{2} + \varepsilon} \bigr] ,
\label{eq:connection formula for PII Voros}
\end{equation}
where $\varepsilon$ is a sufficiently small positive number.
\end{cor}
\begin{proof}
By \cite[Theorem 1.1]{Takei Sato conjecture}, 
the Voros coefficient $V_{\rm Weber}(E,\eta)$ of the Weber equation \eqref{eq:Weber equation1}
is represented as 
\begin{equation}
V_{\rm{Weber}}(E,\eta) = \frac{1}{2} \sum_{n = 1}^{\infty} 
\frac{2^{1-2n}-1}{2n (2n-1)} B_{2n} (i E \eta)^{1-2n} .
\label{eq:Voros coefficient of Weber equation}
\end{equation}
Because the Voros coefficient of the Weber equation $V_{\rm Weber}(E,\eta)$ 
and the $P$-Voros coefficient $W(c,\eta)$ are related as
\[
W(c,\eta) = 
- 2 \hspace{+.2em} V_{\rm{Weber}}(- i c,\eta) ,
\]
the relation \eqref{eq:connection formula for PII Voros} 
follows from \cite[Theorem 2.1]{Takei Sato conjecture} immediately. 
\end{proof}


\subsection{Derivation of the connection formula for the parametric Stokes phenomena 
through the Voros coefficient of ($P_{\rm II}$)}

We determine the connection formula for the parametric Stokes phenomenon 
which occurs when arg \hspace{-.5em} $c$ varies near $\frac{\pi}{2}$ 
for the 1-parameter solutions $\lambda_{\infty}$ and $\lambda_{\tau_1}$ of ($P_{\rm II}$).  

We know the following result about the Borel summability of WKB solutions 
of second order linear differential equations.
\begin{theo} [{\cite[Theorem 1.2.2]{Pham}}] \label{theorem for fixed singularity}
Consider a second order linear differential equation of the form
\[
\Bigl( \frac{d^2}{d x^2} - \eta^2 Q(x) \Bigr) \psi = 0 ,
\]
where $Q(x)$ is a polynomial of $x$.
Let $\psi_{\pm}$ be a WKB solution of the above equation normalized at infinity 
(like \eqref{eq:normalized at infinity}), 
where the path of integration from $x = \infty$ is assumed to touch with 
no turning points and no Stokes curves. 
Then, $\psi_{\pm}$ is Borel summable even under the situation where the 
degeneration of Stokes curves occurs.
\end{theo}

This result suggests that 
no parametric Stokes phenomenon occurs 
(when arg \hspace{-.4em} $c$ varies near $\frac{\pi}{2}$) 
for $\tilde{\lambda}_{\infty}^{(1)}$ given by \eqref{eq:normalized at infinity}, 
which is a WKB solution of \eqref{eq:lambda(1)} 
normalized at infinity along the path in Figure \ref{fig:normalization path of lambda(1)_infinity}. 
That is, the following holds:
\[
{\cal S} \bigl[ \tilde{\lambda}_{\infty}^{(1)}(t,c,\eta;\alpha) 
\big|_{{\rm arg}c = \frac{\pi}{2} - \varepsilon} \bigr]
 = 
{\cal S} \bigl[ \tilde{\lambda}_{\infty}^{(1)}(t,c,\eta;\tilde{\alpha}) 
\big|_{{\rm arg}c = \frac{\pi}{2} + \varepsilon} \bigr]
\]
\begin{equation}
\Rightarrow \tilde{\alpha} = {\alpha} .
\label{eq:connection formula for lambda(1)_infinity}
\end{equation}
On the other hand, combining \eqref{eq:relation for two normalizations}, 
\eqref{eq:connection formula for lambda(1)_infinity} 
and Corollary \ref{connection formula for PII Voros},
we obtain the following connection formula  
for $\tilde{\lambda}^{(1)}_{\tau_1}$ given by \eqref{eq:normalized at tau1}: 
\[
{\cal S} \bigl[ \tilde{\lambda}_{\tau_1}^{(1)}(t,c,\eta;\alpha) 
\big|_{{\rm arg}c = \frac{\pi}{2} - \varepsilon} \bigr]
 = 
{\cal S} \bigl[ \tilde{\lambda}_{\tau_1}^{(1)}(t,c,\eta;\tilde{\alpha}) 
\big|_{{\rm arg}c = \frac{\pi}{2} + \varepsilon} \bigr]
\]
\[
\Rightarrow 
{\alpha} \hspace{+.2em} {\cal S} \bigl[ e^W \big|_{{\rm arg}c = \frac{\pi}{2} - \varepsilon} \bigr]
 = 
\tilde{\alpha} \hspace{+.2em} {\cal S} \bigl[ e^W \big|_{{\rm arg}c = \frac{\pi}{2} + \varepsilon} \bigr]
\]
\begin{equation}
\Rightarrow \tilde{\alpha} = (1 + e^{2 \pi i c \eta}) \hspace{+.2em} {\alpha} .
\end{equation}
As noted in Section 2, 1-parameter solutions are determined uniquely 
once the normalizations of $\tilde{\lambda}^{(1)}$ are fixed. 
These observations suggest that the parametric Stokes phenomena 
for the 1-parameter solutions of ($P_{\rm II}$) can be described as follows: \\[-.3em]

\noindent
\textbf{Connection formula for the 1-parameter solutions of ($P_{\rm II}$).} \hspace{+.1em}
\textit{Let $\varepsilon$ be a sufficiently small positive number. }

\noindent
\textit{ (i)  If the true solutions represented by 
$\lambda_{\infty}(t,c,\eta;\alpha)$ for arg c $ = \frac{\pi}{2} - \varepsilon$ and those by 
$\lambda_{\infty}(t,c,\eta;\tilde{\alpha})$ for arg c $ = \frac{\pi}{2} + \varepsilon$ coincide, 
then the followig holds: }
\begin{equation}
\tilde{\alpha} = \alpha . 
\label{connection formula for 1-parameter solution normalized at infinity}
\end{equation}

\noindent
\textit{ (ii) If the true solutions represented by 
$\lambda_{\tau_1}(t,c,\eta;\alpha)$ for arg c $ = \frac{\pi}{2} - \varepsilon$ and those by 
$\lambda_{\tau_1}(t,c,\eta;\tilde{\alpha})$ for arg c $ = \frac{\pi}{2} + \varepsilon$ coincide,
then the following holds: }
\begin{equation}
\tilde{\alpha} = (1 + e^{2 \pi i c \eta}) \hspace{+.2em} \alpha . 
\label{connection formula for 1-parameter solution normalized at tau_1}
\end{equation}

In the subsequent sections, 
employing the exact WKB analysis for the Sch$\ddot{\rm o}$dinger equation 
($SL_{\rm II}$) associated with ($P_{\rm II}$) through the isomonodromic deformation,  
we will rederive this connection formula in a completely different manner.


\section{WKB solutions and the Stokes geometry of ($SL_{\rm II}$) and ($D_{\rm II}$)}

($P_{\rm II}$) represents the condition for isomonodromic deformation of the following 
associated Schr$\ddot{\rm o}$dinger equation ($SL_{\rm II}$): 
\[
(SL_{\rm{II}}) : \hspace{+.2em}
\Bigl( \frac{\partial^2}{\partial x^2} - \eta^2 Q_{\rm{II}} \Bigr) \hspace{+.2em} \psi = 0 ,
\]
where
\[
Q_{\rm{II}} = x^4 + t x^2 +2 c x + 2 K_{\rm{II}} 
 - \eta^{-1} \frac{\nu}{x-\lambda} + \eta^{-2} \frac{3}{4(x-\lambda)^2} ,
\]
\[
K_{\rm{II}} = \frac{1}{2} 
 [ \nu^2 - (\lambda^4 + t \lambda^2 + 2 c \lambda) ] . 
\]
Here the regular singular point $x = \lambda$ of ($SL_{\rm II}$) 
is an apparent singular point because $K_{\rm II}$ has the above form.
We will construct WKB solutions of ($SL_{\rm II}$) 
satisfying its deformation equation ($D_{\rm II}$) in this section.
Here the deformation equation ($D_{\rm II}$) of ($SL_{\rm II}$) is given by the following:
\[
(D_{\rm{II}}) : \hspace{+.2em} \frac{\partial \psi}{\partial t}  = 
A_{\rm{II}} \frac{\partial \psi}{\partial x} 
- \frac{1}{2} \frac{\partial A_{\rm{II}}}{\partial x} \psi , 
\]
\[
A_{\rm{II}} = \frac{1}{2(x - \lambda)} .
\]
Note that the system ($SL_{\rm II}$) and ($D_{\rm II}$) is obtained from  
the isomonodromic deformation equation of 
Jimbo-Miwa's form \cite[Appendix C]{Jimbo-Miwa}.

Before discussing the construction of WKB solution satisfying 
both ($SL_{\rm II}$) and ($D_{\rm II}$), 
we first establish the relation between degeneration of 
the $P$-Stokes curves of ($P_{\rm II}$) and that of the Stokes curves of ($SL_{\rm II}$).


\subsection{Stokes geometry of ($SL_{\rm II}$)}

We investigate the Stokes geometry of ($SL_{\rm II}$) in this subsection. 
By substituting a 1-parameter solution ($\lambda(t,c,\eta;\alpha), \nu(t,c,\eta;\alpha)$) 
of ($H_{\rm II}$) into the coefficients of ($SL_{\rm II}$) and ($D_{\rm II}$), 
we find that $Q_{\rm II}$ and $A_{\rm II}$ are expanded as follows:
\begin{equation}
Q_{\rm{II}}(x,t,c,\eta;\alpha) = 
 Q_{\rm{II}}^{(0)}(x,t,c,\eta) 
 + \alpha \eta^{-\frac{1}{2}} Q_{\rm{II}}^{(1)}(x,t,c,\eta) 
 {{e}}^{\eta \phi_{\rm{II}}}
 + (\alpha \eta^{-\frac{1}{2}})^2 Q_{\rm{II}}^{(2)}(x,t,c,\eta) 
 {{e}}^{2 \eta \phi_{\rm{II}}}
 + \cdots ,  
    \label{eq:Q_II}
\end{equation}
\begin{equation}
A_{\rm{II}}(x,t,c,\eta;\alpha) = 
 A_{\rm{II}}^{(0)}(x,t,c,\eta) 
 + \alpha \eta^{-\frac{1}{2}} A_{\rm{II}}^{(1)}(x,t,c,\eta) 
 {{e}}^{\eta \phi_{\rm{II}}}
 + (\alpha \eta^{-\frac{1}{2}})^2 A_{\rm{II}}^{(2)}(x,t,c,\eta) 
 {{e}}^{2 \eta \phi_{\rm{II}}}
 + \cdots  ,  
      \label{eq:A_II}
\end{equation}
\[
Q_{\rm{II}}^{(k)}(x,t,c,\eta) = 
Q_0^{(k)}(x,t,c) + \eta^{-1} Q_1^{(k)}(x,t,c) + \eta^{-2} Q_2^{(k)}(x,t,c) + \cdots , 
\]
\[
A_{\rm{II}}^{(k)}(x,t,c,\eta) = 
A_0^{(k)}(x,t,c) + \eta^{-1} A_1^{(k)}(x,t,c) + \eta^{-2} A_2^{(k)}(x,t,c) + \cdots .
\] 
Especially, $Q_{\rm II}^{(0)}$ and $A_{\rm II}^{(0)}$ 
and their leading terms are given by the following:
\begin{eqnarray}
Q_{\rm{II}}^{(0)}(x,t,c,\eta) & = & 
 x^4 + t x^2 +2 c x + 2 K_{\rm{II}} 
 - \eta^{-1} \frac{\nu^{(0)}}{x-\lambda^{(0)}} + \eta^{-2} \frac{3}{4(x-\lambda^{(0)})^2} , 
      \label{eq:Q(0)} \\
A_{\rm{II}}^{(0)}(x,t,c,\eta) & = & \frac{1}{2(x - \lambda^{(0)})} ,  
      \label{eq:A(0)}
\end{eqnarray}
\begin{eqnarray} 
Q_0^{(0)}(x,t,c) & = & x^4 + t x^2 +2 c x 
 - (\lambda_0^4 + t \lambda_0^2 + 2 c \lambda_0) \nonumber \\
& = & (x -\lambda_0)^2 (x^2 + 2 \lambda_0 x + 3 \lambda_0^2 + t)  ,  
    \label{eq:Q_0} \\
A^{(0)}_0(x,t,c) & = & \frac{1}{2(x - \lambda_0)} .  
      \label{eq:A_0}
\end{eqnarray}
Here ($\lambda^{(0)}, \nu^{(0)}$) is the 0-parameter solution 
which is the principal part of the 1-parameter solution 
substituted. 
In what follows we abbreviate $Q_0^{(0)}(x,t,c)$ and $A_0^{(0)}(x,t,c)$ 
to  $Q_0(x,t,c)$ and $A_0(x,t,c)$, respectively.
\begin{defi} [{\cite[Definition 2.4, Definition 2.6]{KT iwanami}}] \normalfont
(i) A point $x = a$ is called a
\textit{turning point of ($SL_{\rm II}$)} 
if $x$ satisfies $Q_0(a,t,c) = 0$. 

\noindent
(ii) For a turning point $x = a$,  a real one-dimensional curve defined by 
\[
{\rm Im} \int_{a}^x \sqrt{Q_{0}(x,t,c)} \hspace{+.2em} dx = 0
\]
is called a \textit{Stokes curve of $(SL_{\rm II})$} .
\end{defi}
In view of \eqref{eq:Q_0} we know that ($SL_{\rm II}$) has a double turning point 
at $x = \lambda_0$ and two simple turning points which denoted by $x = a_1, a_2$ 
in what follows, where $a_1$ and $a_2$ are two roots of $x^2 + 2 \lambda_0 x + 3 \lambda_0^2 + t = 0$. 
Figure \ref{fig:SL_II,argc=0.5Pi-epsilon} $\sim$ \ref{fig:SL_II,argc=0.5Pi+epsilon}  
describe the Stokes curves of ($SL_{\rm II}$) 
with $t$ being a fixed point in the region in Figure \ref{fig:Riemann surface of sqrt Delta} 
and arg \hspace{-.5em} $c$ varying near $\frac{\pi}{2}$. 
(Here $\varepsilon$ is a sufficiently small positive number.)
  \begin{figure}[h]
  \begin{minipage}{0.32\hsize}
  \begin{center}
  \includegraphics[width=50mm]{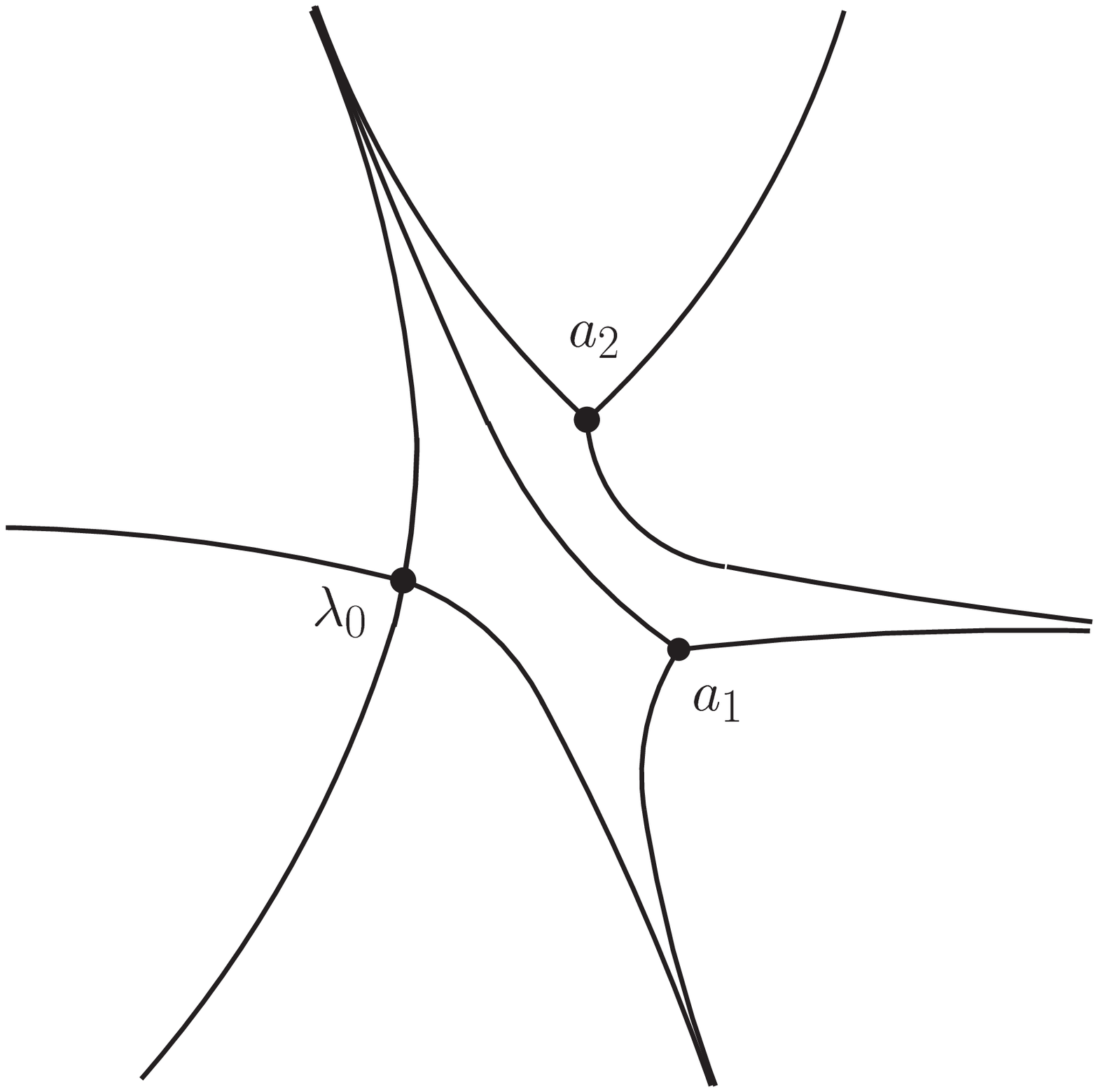}
  \end{center}
  \caption{\small{Stokes curve of ($SL_{\rm II}$) when arg \hspace{-.5em} $c = \frac{\pi}{2}-\varepsilon$.}}
  \label{fig:SL_II,argc=0.5Pi-epsilon}
  \end{minipage} \hspace{+.31em}
  \begin{minipage}{0.32\hsize}
  \begin{center}
  \includegraphics[width=50mm]{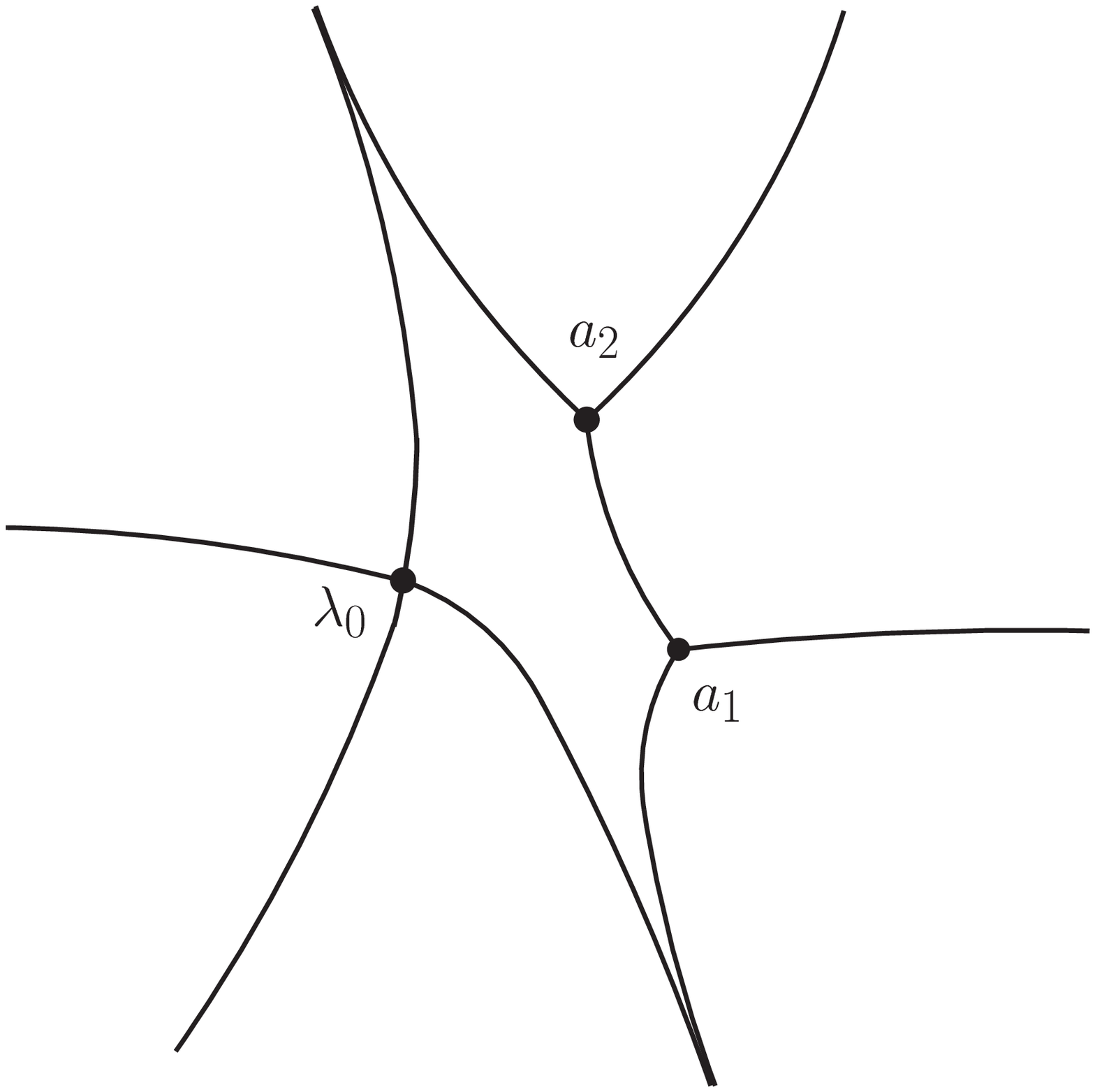}
  \end{center}
  \caption{\small{Stokes curve of ($SL_{\rm II}$) when arg \hspace{-.5em} $c = \frac{\pi}{2}$.}}
  \label{fig:SL_II,argc=0.5Pi}
  \end{minipage} \hspace{+.31em}
  \begin{minipage}{0.32\hsize}
  \begin{center}
  \includegraphics[width=50mm]{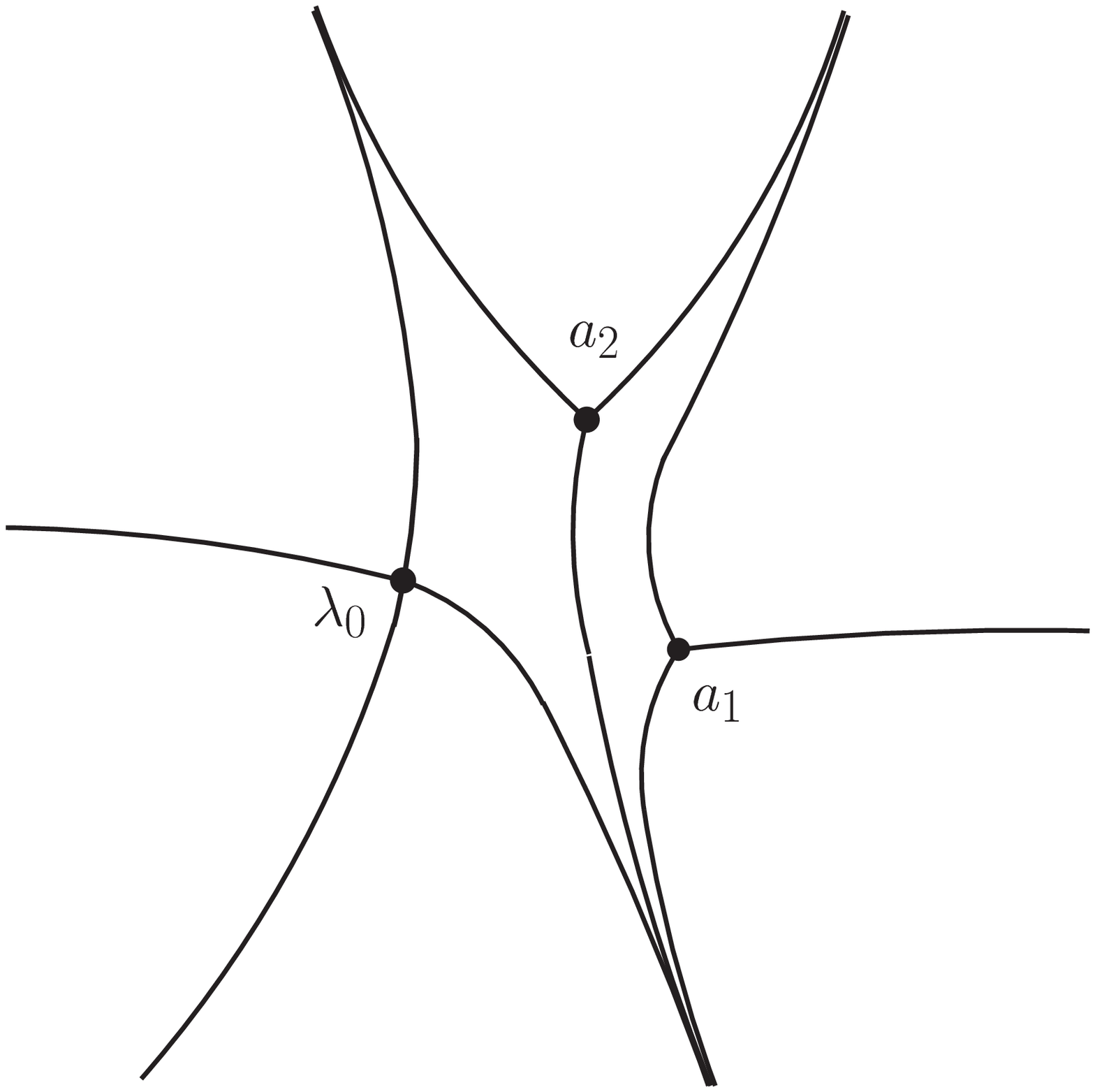}
  \end{center}
  \caption{\small{Stokes curve of ($SL_{\rm II}$) when arg \hspace{-.5em} $c = \frac{\pi}{2}+\varepsilon$.}}
  \label{fig:SL_II,argc=0.5Pi+epsilon}
  \end{minipage}
  \end{figure}
It is observed that a degeneration of the $P$-Stokes geometry  
of ($P_{\rm II}$) and that of the Stokes geometry of ($SL_{\rm II}$) occurs at
arg \hspace{-.3em} $c = \frac{\pi}{2}$ simultaneously.
This intriguing observation can be confirmed analytically by the following proposition. 
\begin{prop} \label{integral of top terms}
\begin{equation}
\int_{a_1}^{a_2} \sqrt{Q_0(x,t,c)} \hspace{+.2em} dx 
 = - \frac{1}{2} \int_{\tau_1}^{\tau_2} \sqrt{\Delta(t,c)} \hspace{+.2em} dt 
 = \pi i c .  \label{eq:integral of Q_0 and Delta}
\end{equation}
Here the integral $\int_{a_1}^{a_2} \sqrt{Q_0(x,t,c)} \hspace{+.2em} dx $ 
is defined by $\frac{1}{2} \int_{\gamma} \sqrt{Q_0(x,t,c)} \hspace{+.2em} dx$,  
where $\gamma$ designates the closed curve in the cut plane shown in 
Figure \ref{fig:integral path gamma},
and the path of the integral $\int_{\tau_1}^{\tau_2} \sqrt{\Delta(t,c)} \hspace{+.2em} dt $  
is taken to be along the $P$-Stokes curve which connects two 
$P$-turning points $t = \tau_1$ and $\tau_2$ in Figure \ref{fig:Riemann surface of sqrt Delta}. 
(The Wiggly line in Figure \ref{fig:integral path gamma} is a cut to define the Riemann surface   
of $\sqrt{Q_0}$. We adopt the branch of $\sqrt{Q_0}$ such that
$\sqrt{Q_0} \sim x^2$ as $x \rightarrow \infty$ in this cut plane.)
 \begin{figure}[h]
 \begin{center}
 \includegraphics[width=50mm]{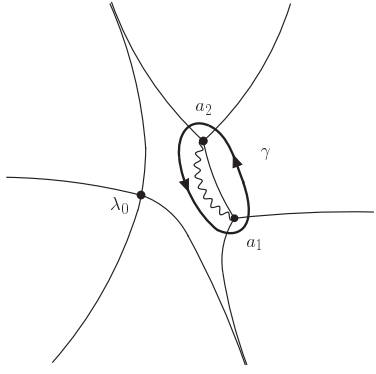}
 \end{center}
 \caption{Integration path $\gamma$.}
 \label{fig:integral path gamma}
 \end{figure}
\end{prop}
\begin{proof}
Firstly, we compute the integral $\int_{a_1}^{a_2} \sqrt{Q_0(x,t,c)} \hspace{+.2em} dx $.
Because $\sqrt{Q_0}$ has no singularities besides at $x = a_1, a_2$ in ${\mathbb C}$, 
we have 
\begin{equation}
\int_{\gamma} \sqrt{Q_0(x,t,c)} dx = 
- 2 \pi i \hspace{+.3em} {\rm{Res}}_{x = \infty} 
\bigl( \sqrt{Q_0( x,t,c)} \hspace{+.2em} dx \bigr) .
\end{equation}
We also obtain 
\[
{\rm{Res}}_{x = \infty} \bigl( \sqrt{Q_0(x,t,c)} \hspace{+.2em} dx \bigr) = - c
\]
by \eqref{eq:behavior of S(0)_-1} below. Hence we find
\[
\int_{a_1}^{a_2} \sqrt{Q_0(x,t,c)} \hspace{+.2em} dx  =  \pi i c .
\]

Next, we prove the first equality of \eqref{eq:integral of Q_0 and Delta}. 
It follows from \cite[Theorem 4.9]{KT iwanami} that, 
when $t$ tends to a simple $P$-turning point, 
a simple turning point of ($SL_{\rm II}$) merges with the double turning point $x = \lambda_0$.  
Under the current situation we can check that 
$x = a_j$ merges with the double turning point $x = \lambda_0$ 
at $t = \tau_j$ for $j = 1, 2$ respectively by numerical computation. 
Thus, by the same theorem in \cite{KT iwanami}, we also have
\begin{equation}
\int_{a_j}^{\lambda_0} \sqrt{Q_0(x,t,c)} \hspace{+.2em} dx = 
\pm \frac{1}{2} \int_{\tau_j}^{t} \sqrt{\Delta(t,c)}  \hspace{+.2em} dt \hspace{+1.em} (j=1, 2) . 
\label{eq:integral relation of top terms}
\end{equation}
The sign of the right-hand side of \eqref{eq:integral relation of top terms} 
depends on the choice of the branch of $\sqrt{\Delta}$. 
Thus we obtain
\begin{equation}
\pi i c = 
\int_{a_1}^{a_2} \sqrt{Q_0(x,t,c)} \hspace{+.2em} dx  = 
\mp \frac{1}{2} \int_{\tau_1}^{\tau_2} \sqrt{\Delta(t,c)} \hspace{+.2em} dt .
\label{eq:integral relation2}
\end{equation}
Since the branch of $\sqrt{\Delta}$ was determined such that 
the real part of $\int_{\tau_1}^{\tau_2} \sqrt{\Delta} dt$ is positive 
when arg$c = \frac{\pi}{2}$ 
(see Figure \ref{fig:Riemann surface of sqrt Delta}),  
the sign of the right-hand side of \eqref{eq:integral relation2} is $-$ 
(hence the sign of \eqref{eq:integral relation of top terms} is $+$), 
which completes the proof of Proposition \ref{integral of top terms}.
\end{proof}


\subsection{WKB solutions of ($SL_{\rm II}$) and ($D_{\rm II}$)}

We construct WKB solutions of ($SL_{\rm II}$) satisfying ($D_{\rm II}$) simultaneously 
in this subsection. Although $Q_{\rm II}$ is expanded as \eqref{eq:Q_II}, 
we can construct WKB solutions in the following form, 
similarly to \cite[$\S2$]{KT iwanami}
(the well-definedness of the integral will be confirmed 
in Lemma \ref{lemma for behaviors of S} below):
\begin{equation}
\psi_{\pm,\infty} = \frac{1}{\sqrt{S_{\rm{odd}}}} {\rm{exp}} \pm
\biggl\{
\eta \int_{a_1}^{x} S_{-1} dx  + \int_{\infty}^{x} \bigl(S_{\rm{odd}} - \eta S_{-1} \bigr) dx
\biggr\} ,  \label{eq:psi_pm infinity}
\end{equation}
where 
\begin{equation}
S_{\rm{odd}}(x,t,c,\eta;\alpha) = S_{\rm{odd}}^{(0)}(x,t,c,\eta) + 
\alpha \eta^{-\frac{1}{2}} S_{\rm{odd}}^{(1)}(x,t,c,\eta) 
{{e}}^{\eta \phi_{\rm{II}}} +
(\alpha \eta^{-\frac{1}{2}})^2 S_{\rm{odd}}^{(2)}(x,t,c,\eta) 
{{e}}^{2 \eta \phi_{\rm{II}}} + \cdots  \hspace{+.2em} 
\label{eq:expansion of S_odd}
\end{equation}
is the odd part (in the sense of Remark \ref{odd part}) of a formal solution
\begin{equation}
S(x,t,c,\eta;\alpha) =
S^{(0)}(x,t,c,\eta) + 
\alpha \eta^{-\frac{1}{2}} S^{(1)}(x,t,c,\eta) {{e}}^{\eta \phi_{\rm{II}}} + 
(\alpha \eta^{-\frac{1}{2}})^2 S^{(2)}(x,t,c,\eta) {{e}}^{2 \eta \phi_{\rm{II}}} + \cdots
\label{eq:expansion of S}
\end{equation}
of the associated Riccati equation of ($SL_{\rm II}$)
\begin{equation}
S^2 + \frac{\partial S}{\partial x} = \eta^2 Q_{\rm{II}}(x,t,c,\eta;\alpha) ,
\label{eq:SL_II-Riccati}
\end{equation}
$S_{\rm{odd}}^{(k)}$ and $S^{(k)}$ are formal power series of $\eta^{-1}$ of the form 
\[
S_{\rm{odd}}^{(k)}(x,t,c,\eta) = \eta S_{{\rm odd},-1}^{(k)}(x,t,c) +
 S_{\rm{odd},0}^{(k)}(x,t,c) + \eta^{-1}S_{\rm{odd},1}^{(k)}(x,t,c) + 
 \cdots \hspace{+.2em} (k \ge 0)  \hspace{+.2em} ,
\]
\[
S^{(k)}(x,t,c,\eta) = \eta S_{-1}^{(k)}(x,t,c) +
 S_{0}^{(k)}(x,t,c) + \eta^{-1}S_{1}^{(k)}(x,t,c) + 
 \cdots \hspace{+.2em} (k \ge 0)  \hspace{+.2em} ,
\]
and
\begin{eqnarray}
S_{-1}(x,t,c)  & = & S^{(0)}_{-1}(x,t,c) = \sqrt{Q_0(x,t,c)} \nonumber \\
& = & (x - \lambda_0) \sqrt{x^2 + 2 \lambda_0 x + 3 \lambda_0^2 + t} \hspace{+.2em} . \label{eq:S_-1}
\end{eqnarray} 
The formal series $S^{(k)}$ ($k \ge 0$) satisfy the following differential equations: 
\begin{equation}
{S^{(0)}}^2 + \frac{\partial S^{(0)}}{\partial x} = \eta^{2} Q^{(0)}. 
\label{eq:SL_II-Riccati(0)}
\end{equation}
\begin{equation}
2 S^{(0)} S^{(k)} + \sum_{k_1 + k_2 = k \hspace{+.3em} ,k_j < k} S^{(k_1)} S^{(k_2)} 
+ \frac{\partial S^{(k)}}{\partial x} = \eta^2 Q^{(k)}_{\rm{II}} \hspace{+1.em} (k \ge 1) . 
\label{eq:S(k)}
\end{equation}
Once we fix the branch of $S_{-1}$, then $S^{(k)}_{\ell}$ is determined uniquely 
by the following recursive relations: 
\begin{equation}
2 S_{-1} S_{\ell + 1}^{(k)} 
+ \sum_{
\tiny
\begin{array}{ll}
k_1+k_2 = k  \\
\ell_1 + \ell_2 =\ell \\
0 \le k_j < k ,
\ell_j \ge 0
\end{array} 
\normalsize
}
S_{\ell_1}^{(k_1)} S_{\ell_2}^{(k_2)} 
+ \frac{\partial S_{\ell}^{(k)}}{\partial x} = Q_{\ell + 2}^{(k)}  
\hspace{+.8em} \hspace{+.2em} (k \ge 0 , \ell \ge -2)
\label{eq:S^(k)_ell}
\end{equation} 
In what follows we adopt the branch of $S_{-1}$ such that 
$S_{-1} \sim x^2$ as $x \rightarrow \infty$ on the cut plane 
shown in Figure \ref{fig:integral path gamma}.
\begin{prop} \label{vanishing of S(k)_-1 = 0}
\begin{equation}
S_{{\rm odd},-1}^{(k)} = S_{-1}^{(k)} = 0 \hspace{+1.em} (k \ge 1) . \label{eq:S(k)_-1 = 0} 
\end{equation}
\end{prop}
\begin{proof}
The following relation is shown in 
\cite[Proposition 2.1]{AKT Painleve WKB}:  
\begin{equation}
\frac{\partial}{\partial t} S = 
\frac{\partial}{\partial x} 
\Bigl(
A_{\rm{II}} S - \frac{1}{2} \frac{\partial A_{\rm{II}}}{\partial x}
\Bigr). \label{eq:conpatibility of SL_II and D_II}
\end{equation}
Comparing the coefficients of $e^{k \eta \phi_{\rm II}}$ of both sides of the equation 
\eqref{eq:conpatibility of SL_II and D_II}, we obtain 
\[
k \eta \frac{d \phi_{\rm{II}}}{d t} S^{(k)} + \frac{\partial}{\partial t} S^{(k)} =  
\frac{\partial}{\partial x} 
\Bigl(
A^{(k)}_{\rm{II}} S^{(0)} + \sum_{j = 1}^{k} A^{(k - j)} S^{(j)}
 - \frac{1}{2} \frac{\partial A^{(k)}_{\rm{II}}}{\partial x}
\Bigr) .
\]
By taking the coefficients of $\eta^2$ of both sides of this equation,
we have $k \frac{d \phi_{\rm{II}}}{d t} S_{-1}^{(k)} = 0$.  
Thus we obtain \eqref{eq:S(k)_-1 = 0}.  
It is obvious that $S_{{\rm odd},-1}^{(k)} = S_{-1}^{(k)}$ for all $k \ge 0$ by the definition.
\end{proof}

\noindent
We can check the following facts easily by straightforward computations.
\begin{lemm} \label{lemma for behaviors of S}
The following asymptotic behaviors hold as $x \rightarrow \infty$:
\begin{eqnarray}
S_{-1}(x,t,c) & = &  
x^2 + \frac{t}{2} + \frac{c}{x} + {O}(x^{-2}) ,
\label{eq:behavior of S(0)_-1}  \\
S^{(0)}_0(x,t,c) & = & 
- x^{-1} +  {O}(x^{-2}) , \\
S^{(0)}_{\ell}(x,t,c) & = &  {O}(x^{-2}) \hspace{+.3em} (\ell \ge 1) , 
\label{eq:behavior of S(0)_ell}  \\
S^{(k)}_{\ell}(x,t,c) & = &  {O}(x^{-2}) \hspace{+.3em}
 (k \ge 1 , \ell \ge 0) .
\label{eq:behavior of S(k)_ell} 
\end{eqnarray}
\end{lemm}

\noindent
Therefore, the integral in \eqref{eq:psi_pm infinity} is well-defined. 
We also note that $S_{\rm even}$, the even part of $S$ 
(in the sense of Remark \ref{odd part}), satisfies that 
\begin{equation}
S_{\rm{even}} = - \frac{1}{2} \frac{\partial}{\partial x} {\rm{log}} S_{\rm{odd}} .
\label{eq:S_odd and S_even}
\end{equation}

The WKB solutions $\psi_{\pm, \infty}$ do not satisfy ($D_{\rm II}$). 
As a matter of fact, we can verify the following proposition.
\begin{prop} \label{t-derivation of psi_infinity}
\begin{equation}
\frac{\partial}{\partial t} \psi_{\pm,\infty} = 
A_{\rm{II}} \frac{\partial \psi_{\pm,\infty}}{\partial x}  
- \frac{1}{2} \frac{\partial A_{\rm{}II}}{\partial x} \psi_{\pm,\infty}
\mp \frac{1}{2} \eta (\lambda - \lambda_0) \psi_{\pm,\infty}   .
\label{eq:t-derivation of psi_infinity}
\end{equation}
\end{prop}
\begin{proof}
Using \eqref{eq:conpatibility of SL_II and D_II}, we have
\begin{eqnarray}
\frac{\partial}{\partial t} S_{\rm{odd}} & = & 
\frac{\partial}{\partial x} 
\bigl(
A_{\rm{II}} S_{\rm{odd}} 
\bigr) ,
\label{eq:conpatibility of SL_II and D_II 2} \\
\frac{\partial}{\partial t} S_{-1} & = & 
\frac{\partial}{\partial x} 
\bigl(
A_0 S_{-1} 
\bigr) .
\label{eq:conpatibility of SL_II and D_II 3}
\end{eqnarray}
Taking these relations into account and differentiating $\psi_{\pm,\infty}$ with respect to $t$, 
we obtain the following:
\begin{eqnarray}
\frac{\partial}{\partial t} \psi_{\pm,\infty} 
& = &  
- \frac{1}{2} \frac{1}{S_{\rm{odd}}} \frac{\partial S_{\rm{odd}}}{\partial t} \psi_{\pm,\infty} 
\pm \biggl\{
\eta \int_{a_1}^{x} \frac{\partial S_{-1}}{\partial t} dx  
+ \int_{\infty}^{x} \frac{\partial}{\partial t} \bigl(S_{\rm{odd}} - \eta S_{-1}  \bigr)
\biggl\}  \psi_{\pm,\infty}  \nonumber  \\[+.4em]
& = &
- \frac{1}{2} \frac{1}{S_{\rm{odd}}} \frac{\partial}{\partial x} 
\bigl(
A_{\rm{II}} S_{\rm{odd}} 
\bigr)   \psi_{\pm,\infty}  \nonumber \\
&   &    
\pm \biggl\{
\eta \int_{a_1}^{x} \frac{\partial}{\partial x}
\bigl(
A_0 S_{-1} 
\bigr)  dx  
+ \int_{\infty}^{x} \frac{\partial}{\partial x} 
\bigl(A_{\rm{II}}S_{\rm{odd}} - \eta A_0 S_{-1}  \bigr)
\biggl\}  \psi_{\pm,\infty}   \nonumber  \\[+.5em]
& = &
- \frac{1}{2} 
\Bigl(
\frac{\partial A_{\rm{II}}}{\partial x}  +
A_{\rm{II}} \frac{1}{S_{\rm{odd}}} \frac{\partial S_{\rm{odd}}}{\partial x} 
\Bigr)  \psi_{\pm,\infty}  \nonumber \\
&   &  
\pm A_{\rm{II}} S_{\rm{odd}} \psi_{\pm,\infty}  
\mp  \bigl[A_{\rm{II}}S_{\rm{odd}} - \eta A_0 S_{-1} \bigr] \hspace{-.1em} \Bigl|_{x = \infty} 
 \hspace{-.7em} \psi_{\pm,\infty} .
\label{eq:t-derivation of psi_infinity0}
\end{eqnarray}
Since $A_{\rm II}$ behaves as 
\[
A_{\rm{II}} = 
\frac{1}{2(x-\lambda_0)} + \frac{\lambda - \lambda_0}{2 (x - \lambda_0)^2} + {O}(x^{-3}) 
\]
when $x$ tends to $\infty$, we have 
\begin{equation}
\bigl[A_{\rm{II}}S_{\rm{odd}} - \eta A_0 S_{-1} \bigr] \Bigl|_{x = \infty} = 
\frac{1}{2} \eta (\lambda - \lambda_0) 
\label{eq:value of A_II S_odd - eta A_0 S_-1 at x = infinity}
\end{equation}
in view of \eqref{eq:behavior of S(0)_-1} $\sim$ \eqref{eq:behavior of S(k)_ell}.
Furthermore, it follows from the definition \eqref{eq:psi_pm infinity} of $\psi_{\pm,\infty}$ that 
\begin{equation}
\frac{\partial \psi_{\pm,\infty}}{\partial x} = 
\Bigl(
\pm S_{\rm odd} 
- \frac{1}{2}\frac{1}{S_{\rm{odd}}} \frac{\partial S_{\rm{odd}}}{\partial x} 
\Bigr)  \psi_{\pm,\infty} .
\label{eq:x-derivation of psi_infinity}
\end{equation}
Making use of \eqref{eq:t-derivation of psi_infinity0}, 
\eqref{eq:value of A_II S_odd - eta A_0 S_-1 at x = infinity} 
and \eqref{eq:x-derivation of psi_infinity}, we have \eqref{eq:t-derivation of psi_infinity}.
\end{proof}

Proposition \ref{t-derivation of psi_infinity} implies that the WKB solutions
\[
\psi_{\pm,\rm IM} = e^{\pm \frac{1}{2} U} \psi_{\pm,\infty}
\]
satisfy both ($SL_{\rm II}$) and ($D_{\rm II}$), 
where 
\begin{equation}
U = U(t,c,\eta;\alpha) = \eta \int_{\infty}^{t}
\bigl( \lambda(t,c,\eta;\alpha) - \lambda_0(t,c) \bigr) dt . 
\label{eq:U}
\end{equation}
(IM stands for Iso-Monodromic.)
The Stokes multipliers of ($SL_{\rm II}$) around $x = \infty$ will be computed 
in Section 6.3 by using $\psi_{\pm,\rm IM}$. 

\begin{rem} \label{remark for definition of U} \normalfont
The integral \eqref{eq:U} is defined in the following sense:  
\begin{equation}
U(t,c,\eta;\alpha) = U^{(0)}(t,c,\eta) + 
 \alpha \eta^{-\frac{1}{2}} U^{(1)}(t,c,\eta) {{e}}^{\eta \phi_{\rm{II}}} + 
 (\alpha \eta^{-\frac{1}{2}})^2 U^{(2)}(t,c,\eta) {{e}}^{2 \eta \phi_{\rm{II}}} + \cdots ,
\label{eq:expansion of U}
\end{equation}
\[
U^{(k)}(t,c,\eta) = U_0^{(k)}(t,c) + \eta^{-1} U_1^{(k)}(t,c) + \eta^{-2} U_2^{(k)}(t,c) + \cdots 
\hspace{+1.em} (k \ge 0) , 
\]
where $U^{(0)}$ is defined by the following integral  
whose integration path is the same as 
the normalization path of $\tilde{\lambda}_{\infty}^{(1)}$ 
(see Figure \ref{fig:normalization path of lambda(1)_infinity}): 
\begin{equation}
U^{(0)}(t,c,\eta) = \eta\int_{\infty}^{t} (\lambda^{(0)}(t,c,\eta) - \lambda_0(t,c)) dt ,  
\label{eq:U(0)}
\end{equation}
and $U^{(k)}$ ($k \ge 1$) is the unique formal series satisfying 
\[ 
k \eta \hspace{+.1em} \frac{d \phi_{\rm{II}}}{d t} \hspace{+.1em} U^{(k)} 
+ \frac{\partial}{\partial t} \hspace{+.1em} U^{(k)} 
 = \frac{1}{2} \hspace{+.1em} \eta \hspace{+.1em} \lambda^{(k)} .
\]
This is not an integral in the usual sense, but $U$ satisfies
\begin{equation}
\frac{\partial}{\partial t} U = \eta \hspace{+.2em} (\lambda - \lambda_0) . 
\label{eq:t-derivation of U}
\end{equation}
The end point $t = \infty$ of the integral in \eqref{eq:U} is chosen so as to 
obtain an explicit representation of the ``Voros coefficient" of ($SL_{\rm II}$), 
which will be calculated in the next section. 
We also note that, making use of Lemma \ref{lemma for behaviors of S} and 
\eqref{eq:behavior of integrated S_-1} below, we have the following asymptotic behavior  
of $\psi_{\pm, {\rm IM}}$ as $x$ tends to $\infty$: 
\begin{equation}
\psi_{\pm, {\rm IM}} = \eta^{-\frac{1}{2}} x^{-1 \pm c \hspace{+.1em} \eta} \hspace{+.1em} 
{\rm exp}  \hspace{+.1em} \pm \Big\{ \eta \hspace{+.1em} \Bigl( \frac{1}{3} \hspace{+.1em} x^3 
+ \frac{1}{2} \hspace{+.1em} t \hspace{+.1em} x \Bigr) 
- \frac{1}{2} \hspace{+.1em} \Big(\frac{4}{3} \hspace{+.1em} \eta \hspace{+.1em} \lambda_0^3  
+ \hspace{+.1em} c \hspace{+.1em} \eta \hspace{+.1em} {\rm log}  
\Bigl( - \frac{2 \lambda_0^2 + t}{4}  \Bigl) - U
\Bigr) \Bigr\} \bigl( 1 + O(x^{-1}) \bigr).
\end{equation}
If we define $u$ by 
\begin{equation}
{\rm log} \hspace{+.1em} u = \frac{4}{3} \hspace{+.1em} \eta \hspace{+.1em} \lambda_0^3  
+ \hspace{+.1em} c \hspace{+.1em} \eta \hspace{+.1em} {\rm log}  
\Bigl( - \frac{2 \lambda_0^2 + t}{4}  \Bigl) - U, 
\end{equation}
we can verify that 
\[
\frac{d}{dt} \hspace{+.1em} {\rm log} \hspace{+.1em} u = - \eta \hspace{+.1em} \lambda.
\]
Hence, the quantity $u$ defined above is nothing but $u$ which appeared in 
\cite[(C.10)-(C.13) in Appendix]{Jimbo-Miwa}.
\end{rem}

\begin{rem} \label{remark for expansion of psi_IM} \normalfont
Since $S_{\rm odd}$ and $U$ are expanded as 
\eqref{eq:expansion of S_odd} and \eqref{eq:expansion of U} respectively, 
$\psi_{\pm,\rm IM}$ is expanded as follows 
(see also Proposition \ref{vanishing of S(k)_-1 = 0}): 
\begin{equation}
\hspace{-.5em}
\psi_{\pm,\rm IM}(x,t,c,\eta;\alpha) = 
\psi_{\pm}^{(0)}(x,t,c,\eta) + 
\alpha \eta^{-\frac{1}{2}} \psi_{\pm}^{(1)}(x,t,c,\eta) {e}^{\eta \phi_{\rm II}} +
(\alpha \eta^{-\frac{1}{2}})^2 \psi_{\pm}^{(2)}(x,t,c,\eta) {e}^{2 \eta \phi_{\rm II}}
 + \cdots  \hspace{+.2em} ,
\label{eq:expansion of psi_pm,IM}
\end{equation}
\[
\hspace{-.5em}
\psi_{\pm}^{(k)}(x,t,c,\eta) = \eta^{-\frac{1}{2}} \bigl\{ 
\psi^{(k)}_{\pm,0}(x,t,c) + \eta^{-1} \psi^{(k)}_{\pm,1}(x,t,c) + \eta^{-2} \psi^{(k)}_{\pm,2}(x,t,c) + \cdots
\bigr\} {\rm exp} \pm \eta \Bigl( \int_{a_1}^{x} S_{-1}(x,t,c)dx  \Bigr) .
\]
Especially, $\psi_{\pm}^{(0)}$ is given by the following:
\begin{equation}
\hspace{-.5em}
\psi_{\pm}^{(0)}(x,t,c,\eta) =  \frac{1}{\sqrt{S^{(0)}_{\rm{odd}}}} {\rm{exp}} \pm
\biggl\{ 
\frac{1}{2} U^{(0)} + 
\int_{\infty}^{x} \bigl(S^{(0)}_{\rm{odd}} - \eta S_{-1} \bigr) dx
\biggr\} \hspace{+.2em} {\rm exp}\pm \eta \Bigl( \int_{a_1}^{x} S_{-1} dx  \Bigr) .
\label{eq:psi(0)_pm,IM}
\end{equation}
\end{rem}

\begin{rem} \normalfont
Because ${\rm{Res}}_{x = \infty} \bigl\{ \bigl( S_{\rm odd} - \eta S_{-1} \bigr) dx \bigr\} = 0$ 
by \eqref{eq:behavior of S(0)_ell} and \eqref{eq:behavior of S(k)_ell}, we have 
\begin{equation}
\int_{a_1}^{a_2} S_{\rm odd} \hspace{+.2em} dx 
= \pi i c \eta . \label{eq:integral of S_odd}
\end{equation}
(The left-hand side is defined by the contour integral 
$\frac{1}{2} \int_{\gamma} S_{\rm odd} \hspace{+.2em} dx$, 
where $\gamma$ is a closed path in Figure \ref{fig:integral path gamma}.)
This relation will be used in the computations of the Stokes multipliers of ($SL_{\rm II}$).
\end{rem}


\section{The Voros coefficient of ($SL_{\rm II}$)}

In this section, as a preparation of the computation of the Stokes multipliers of ($SL_{\rm II}$),
we investigate the Voros coefficient of ($SL_{\rm II}$). 


\subsection{The Voros coefficient of ($SL_{\rm II}$)}

First we define the Voros coefficient of ($SL_{\rm II}$).
\begin{defi} \normalfont
\textit{The Voros coefficient of $(SL_{\rm II})$} is given by
\begin{eqnarray}
V(t,c,\eta;\alpha) =  \int_{a_1}^{\infty} \bigl( S_{\rm odd}(x,t,c,\eta;\alpha) - \eta S_{-1}(x,t,c) \bigr) dx . 
 \label{eq:SL_II Voros coeff}
\end{eqnarray}
Here the right-hand side of \eqref{eq:SL_II Voros coeff} is defined by
\[
\frac{1}{2} \int_{\gamma_{\infty}} \bigl( S_{\rm odd}(x,t,c,\eta;\alpha) - \eta S_{-1}(x,t,c) \bigr) dx ,
\]
where $\gamma_{\infty}$ is a path on the 
Riemann surface of $\sqrt{Q_0}$ shown in Figure \ref{fig:integral path gamma_infinity}.
(The dotted part of $\gamma_{\infty}$ represents a path on the other sheet 
of the Riemann surface of $\sqrt{Q_0}$.)
The right-hand side of \eqref{eq:SL_II Voros coeff} 
is well-defined by Lemma \ref{lemma for behaviors of S}.
 \begin{figure}[h]
 \begin{center}
 \includegraphics[width=55mm]{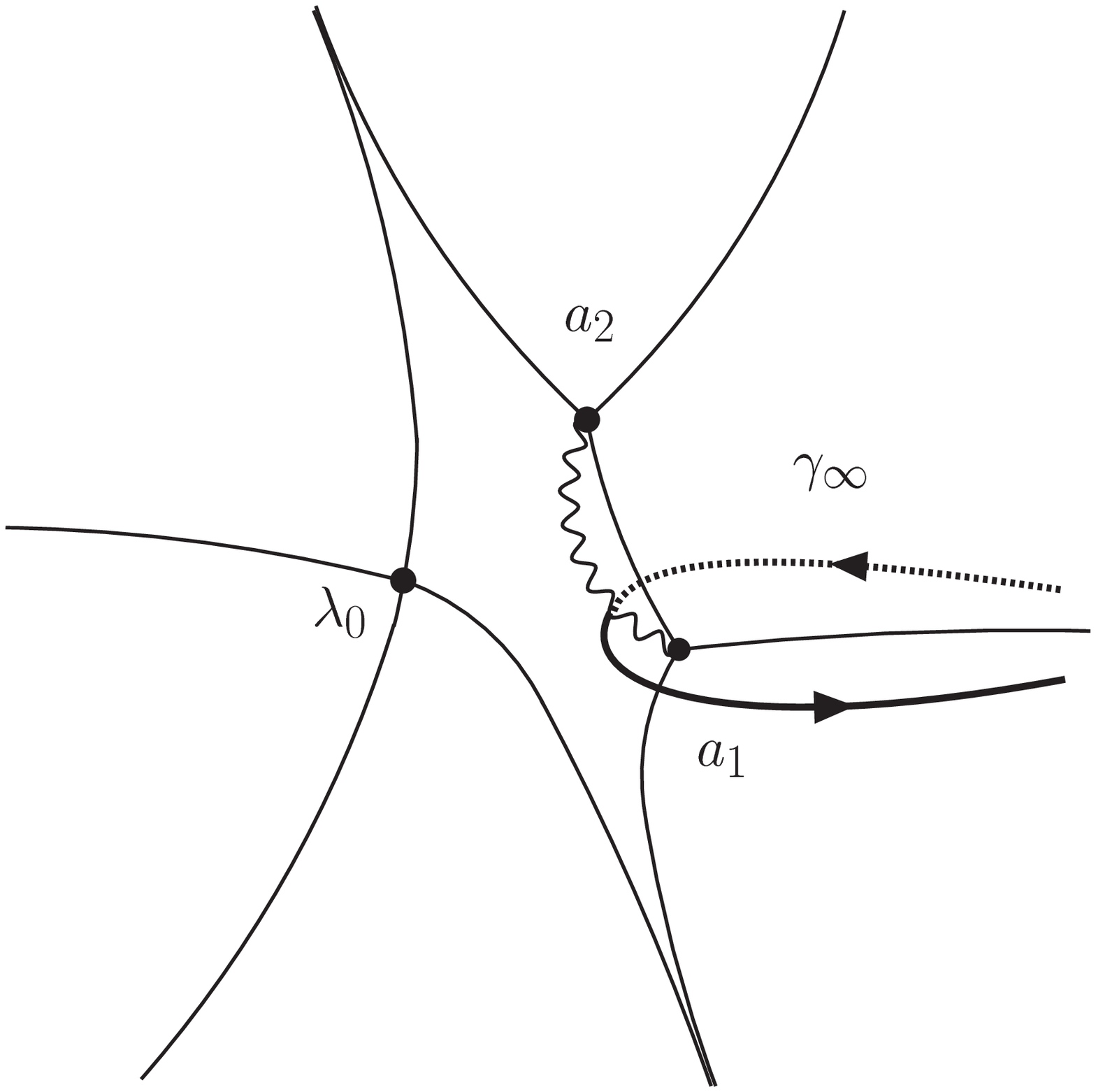}
 \end{center}
 \caption{Integration path $\gamma_{\infty}$.}
 \label{fig:integral path gamma_infinity}
 \end{figure}
\end{defi}
Because $S_{\rm odd}$ is expanded as \eqref{eq:expansion of S}, 
$V$ is also expanded as follows:
\begin{equation}
V(t,c,\eta;\alpha) = V^{(0)}(t,c,\eta) + 
\alpha \eta^{-\frac{1}{2}} V^{(1)}(t,c,\eta) {{e}}^{\eta \phi_{\rm{II}}} +
(\alpha \eta^{-\frac{1}{2}})^2 V^{(2)}(t,c,\eta) {{e}}^{2 \eta \phi_{\rm{II}}}  
  + \cdots \hspace{+.2em},  
\end{equation}
\begin{equation}
V^{(0)}(t,c,\eta)  =  \int_{a_1}^{\infty} \bigl(S^{(0)}_{\rm odd}(x,t,c,\eta) - \eta S_{-1}(x,t,c) \bigr) dx 
= V_0^{(0)}(t,c) + \eta^{-1}V_1^{(0)}(t,c) + \cdots  \hspace{+.2em}, \\
\label{eq:V(0)}
\end{equation}
\[
V^{(k)}(t,c,\eta) = \int_{a_1}^{\infty} S^{(k)}_{\rm odd}(x,t,c,\eta) dx = 
 V_0^{(k)}(t,c) + \eta^{-1}V_1^{(k)}(t,c) + \cdots \hspace{+1.em} (k \ge 1) . \hspace{+3.4em}
\]
\begin{prop} \label{prop for t-derivation of V}
\begin{equation}
\frac{\partial}{\partial t} V = 
 \frac{1}{2} \eta \hspace{+.2em} \bigl(\lambda - \lambda_0 \bigr)  .
\label{eq:t-derivation of V}
\end{equation}
\end{prop}
\begin{proof}
Using \eqref{eq:conpatibility of SL_II and D_II 2},
\eqref{eq:conpatibility of SL_II and D_II 3} and 
\eqref{eq:value of A_II S_odd - eta A_0 S_-1 at x = infinity}, 
we can compute $\frac{\partial}{\partial t} V$ as follows:
we have
\begin{eqnarray*}
\frac{\partial}{\partial t} V & = &  
 \frac{1}{2} \int_{\gamma_{\infty}} 
   \frac{\partial}{\partial t} (S_{\rm odd} - \eta S_{-1}) dx \nonumber \\
& = &
\frac{1}{2} \int_{\gamma_{\infty}} 
   \frac{\partial}{\partial x} (A_{\rm{II}} S_{\rm odd} - \eta A_0 S_{-1}) dx \nonumber \\
& = &
\frac{1}{2} \eta (\lambda - \lambda_0  ). \hspace{+.6em} 
\end{eqnarray*}
\end{proof}

Proposition \ref{prop for t-derivation of V} together with \eqref{eq:t-derivation of U} 
shows that the quantity $2V - U$ is independent of $t$.
Since $2V - U$ is expanded as
\[
2V - U = (2V^{(0)} - U^{(0)}) + 
\alpha \eta^{-\frac{1}{2}} (2V^{(1)} - U^{(1)}) {{e}}^{\eta \phi_{\rm{II}}} +
(\alpha \eta^{-\frac{1}{2}})^2 (2V^{(2)} - U^{(2)} ) {{e}}^{2 \eta \phi_{\rm{II}}} + \cdots ,
\]
this independence of $t$ implies that 
all the coefficients of $e^{k \eta \phi_{\rm II}}$ ($k \ge 1$) must vanish. 
So we obtain the following:
\begin{prop} \label{proposition for 2V - U}
\begin{equation}
2V(t,c,\alpha,\eta) - U(t,c,\alpha,\eta)  = 
2V^{(0)}(t,c,\eta) - U^{(0)}(t,c,\eta) 
\label{eq:2V - U}
\end{equation}
and this is independent of $t$.
\end{prop}

In view of \eqref{eq:Q(0)}, \eqref{eq:SL_II-Riccati(0)}, and \eqref{eq:V(0)} 
$V^{(0)}$ coincides with the Voros coefficient of ($SL_{\rm II}$) 
with a 0-parameter solution of ($H_{\rm II}$) substituted into the coefficients of $Q_{\rm II}$. 
Hence, to compute $2V - U$, it suffices to consider ($SL_{\rm II}$) 
with a 0-parameter solution substituted . 


\subsection{Determination of the Voros coefficient of ($SL_{\rm II}$)}

We have an explicit description of the quantity $2V - U$.
\begin{theo} \label{main theorem 2}
$2V - U$ is represented explicitly as follows: 
\begin{equation}
2V(t,c,\eta) - U(t,c,\eta) = 
- \sum_{n = 1}^{\infty} \frac{2^{1-2n} - 1}{2n(2n - 1)} B_{2n} (c \hspace{+.1em} \eta)^{1-2n} , 
\label{eq:2V(0) - U(0)}
\end{equation}
where $B_{2n}$ is the 2n-th Bernoulli number defined by \eqref{eq:Bernoulli number}.
\end{theo}

\noindent
We will prove Theorem \ref{main theorem 2} in the remainder of this section. 
The idea of proof is similar to that of Theorem \ref{main theorem 1},   
that is, we derive the difference equation for $V^{(0)}$, 
which is the Voros coefficient of ($SL_{\rm II}$) with a 0-parameter solution substituted. 
\begin{lemm} \label{lemma for SL_II Voros coeff}
\begin{equation}
V^{(0)}(t,c,\eta) = 
\frac{1}{2} \int_{\gamma_{\infty}}
\bigl( S^{(0)}(x,t,c,\eta) - \eta S_{-1}(x,t,c) - S^{(0)}_0(x,t,c) \bigr) dx  . 
\end{equation}
\end{lemm}

\noindent
We omit the proof of Lemma \ref{lemma for SL_II Voros coeff} 
because we can show it similarly to Lemma \ref{expression of W}. 
By Lemma \ref{lemma for SL_II Voros coeff}, if we define
\begin{eqnarray*}
{J}(x,t,c,\eta) & = & 
\int_{\gamma_{x}} S^{(0)}(x,t,c,\eta)dx - 
\int_{\gamma_{x}} S^{(0)}(x,t,c-\eta^{-1},\eta) \hspace{+.2em} dx , \\
{J}_{j}(x,t,c,\eta) & = & 
\int_{\gamma_{x}} S_j^{(0)}(x,t,c)dx - 
\int_{\gamma_{x}} S_j^{(0)}(x,t,c-\eta^{-1}) \hspace{+.2em} dx \hspace{+1.em} (j \ge -1), 
\end{eqnarray*}
then we obtain 
\begin{equation}
2V^{(0)}(t,c,\eta) - 2V^{(0)}(t,c-\eta^{-1},\eta) = 
\lim_{x \rightarrow \infty} \bigl(
{J}(x,t,c,\eta) - \eta \hspace{+.2em} {J}_{-1}(x,t,c,\eta) - {J}_{0}(x,t,c,\eta)
\bigr) . 
\label{eq:difference of V(0) in limit form}
\end{equation}
Here $\gamma_x$ is a path on the Riemann surface of $\sqrt{Q_0}$ 
shown in Figure \ref{fig:integration path gamma_x}, 
where $\check{x}$ represents a point on the Riemann surface of $\sqrt{Q_0}$ 
satisfying $Q_0(x,t,c) = Q_0(\check{x},t,c)$ and 
$\sqrt{Q_0(x,t,c)} = - \sqrt{Q_0(\check{x},t,c)}$. 
(The dotted part of $\gamma_x$ represents a path on the other sheet 
of the Riemann surface of $\sqrt{Q_0}$.)
 \begin{figure}[h]
 \begin{center}
 \includegraphics[width=50mm]{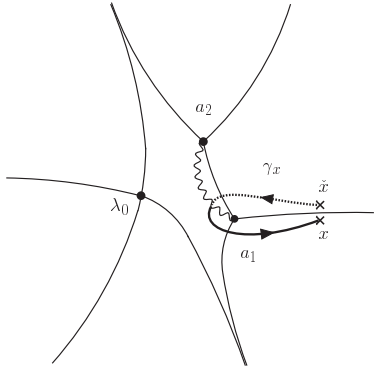}
 \end{center}
 \caption{Integration path $\gamma_{x}$.}
 \label{fig:integration path gamma_x}
 \end{figure}
To calculate the right-hand side of \eqref{eq:difference of V(0) in limit form}, 
we employ the so-called Schlesinger transformation of ($SL_{\rm II}$),  
which induces the translation of the parameter $c \mapsto c - \eta^{-1}$,
as the counterpart of B$\ddot{\rm a}$cklund transformation of ($P_{\rm II}$) 
in Lemma \ref{Backlund}.
\begin{lemm}  [{\cite[pp.423-424]{Jimbo-Miwa}}] \label{Schlesinger}
Let $\psi = \psi(x,t,c,\eta)$ be a solution of ($SL_{\rm II}$) with a solution 
($\lambda, \nu$) of ($H_{\rm II}$) substituted into $Q_{\rm II}$. 
If we define $\varphi = \varphi(x,t,c,\eta)$ by 
\begin{equation}
\varphi(x,t,c,\eta) = f(x,t,c,\eta) \psi(x,t,c,\eta) + g(x,t,c,\eta) \frac{\partial \psi}{\partial x}(x,t,c,\eta) 
\end{equation} 
with
\begin{eqnarray*}
f & = & \bigl( x - \Lambda(\lambda,\nu) \bigr)^{-\frac{1}{2}} (x - \lambda)^{-\frac{1}{2}}
\Bigl\{
x^2 - {\lambda}^2 + \nu 
- \eta^{-1} \frac{1}{2 (x-\lambda)}
\Bigr\} ,  \\
g & = & -\eta^{-1} \bigl( x - \Lambda(\lambda,\nu) \bigr)^{-\frac{1}{2}} (x - \lambda)^{-\frac{1}{2}} , 
\end{eqnarray*}
then $\varphi$ satisfies the following: 
\begin{equation}
\Bigl( 
\frac{\partial^2}{\partial x^2} - \eta^2 \hat{Q}_{\rm II} 
\Bigr) \varphi = 0, 
\label{eq:translated SLII}
\end{equation}
where
\[
\hat{Q}_{\rm II} = x^4 + t x^2+2 (c - \eta^{-1}) x + 2 \hat{K}_{\rm{II}} 
 - \eta^{-1} \frac{{\cal N}(\lambda, \nu)}{x-\Lambda(\lambda, \nu)}  
 + \eta^{-2} \frac{3}{4(x - \Lambda(\lambda, \nu))^2} ,
\]
\[
\hat{K}_{\rm{II}} = \frac{1}{2} 
 \Bigl[ {\cal N}(\lambda, \nu)^2 - \bigl( \Lambda(\lambda, \nu)^4 
 + t \hspace{+.2em} \Lambda(\lambda, \nu)^2 + 2 (c - \eta^{-1}) \Lambda(\lambda, \nu) \bigr) \Bigr] .
\]
(See Lemma \ref{Backlund} for the definition of ($\Lambda(\lambda,\nu), {\cal N}(\lambda, \nu)$).)
\end{lemm}

\noindent
Lemma \ref{Schlesinger} can be shown by straightforward computations. 
With the aid of this Schlesinger transformation, we can derive the difference equation 
for $S^{(0)}$.
\begin{lemm} \label{lemma for SL_II Voros coeff 2}
Let $S^{(0)}(x,t,c,\eta)$ be a formal power series solution of the Riccati equation \eqref{eq:SL_II-Riccati} 
associated with ($SL_{\rm II}$) with a 0-parameter solution ($\lambda^{(0)}(t,c,\eta)$, $\nu^{(0)}(t,c,\eta)$) 
of $(H_{\rm II})$ substituted into the coefficients of $Q_{\rm II}$.  
Then $S^{(0)}(x,t,c,\eta)$ satisfies the following: 
\begin{equation}
S^{(0)}(x,t,c,\eta) - S^{(0)}(x,t,c-\eta^{-1},\eta) = 
 - \frac{\partial}{\partial x} {\rm{log}} \Bigl(f^{(0)}(x,t,c,\eta) + g^{(0)}(x,t,c,\eta) S^{(0)}(x,t,c,\eta) \Bigr) ,
       \label{eq:difference equation for S(0)}
\end{equation}
where $f^{(0)}(x,t,c,\eta)$ and $g^{(0)}(x,t,c,\eta)$ are the formal power series of $\eta^{-1}$ 
obtained by substituting ($\lambda^{(0)}$, $\nu^{(0)}$) 
into the expressions of $f$ and $g$ in Lemma \ref{Schlesinger}, respectively.
\end{lemm}
\begin{proof}
We apply the Schlesinger transformation of Lemma \ref{Schlesinger} to a WKB solution
\[
\psi(x,t,c,\eta) = {\rm{exp}} \Bigl(\int^x S^{(0)}(x,t,c,\eta) \hspace{+.2em} dx \Bigr)
\]
of ($SL_{\rm II}$) with ($\lambda^{(0)}, \nu^{(0)}$) substituted into its coefficients.
Then, by Lemma \ref{Schlesinger}, 
\begin{eqnarray}
\varphi(x,t,c,\eta)
 & = & 
f^{(0)}(x,t,c,\eta) \psi(x,t,c,\eta) + g^{(0)}(x,t,c,\eta) \frac{\partial \psi}{\partial x}(x,t,c,\eta) \nonumber \\
 & = & 
\Bigl(f^{(0)}(x,t,c,\eta) + g^{(0)}(x,t,c,\eta) S^{(0)}(x,t,c,\eta) \Bigr) \psi(x,t,c,\eta) 
\label{eq:Schlesinger transformed WKB solution}
\end{eqnarray}
satisfies \eqref{eq:translated SLII} with 
($\lambda^{(0)}, \nu^{(0)}$) substituted into the coefficients of $\hat{Q}_{\rm II}$. 
On the other hand, since
\[
(\Lambda(\lambda^{(0)}, \nu^{(0)}), {\cal N}(\lambda^{(0)}, \nu^{(0)})) =
(\lambda^{(0)}(t,c-\eta^{-1},\eta), \nu^{(0)}(t,c-\eta^{-1},\eta)) 
\] 
holds as noted in Lemma \ref{application of Backlund transformation} (i), 
the potential of the differential equation satisfied by 
$\varphi$ of \eqref{eq:Schlesinger transformed WKB solution} is 
given by $Q^{(0)}_{\rm II}(x,t,c-\eta^{-1},\eta)$.
Therefore, $\varphi$ is also represented as 
\begin{equation}
\varphi = C(\eta) \hspace{+.2em} {\rm{exp}} 
\Bigl( \int^x S^{(0)}(x,t,c-\eta^{-1},\eta) \hspace{+.2em} dx \Bigr) 
\label{eq:another expression of varphi}
\end{equation}
for some formal power series $C(\eta)$ of $\eta^{-1}$ 
whose coefficients are independent of $x$. 
Thus, taking the logarithmic derivatives with respect to $x$ 
of \eqref{eq:Schlesinger transformed WKB solution} 
and \eqref{eq:another expression of varphi}, 
we have the required relation \eqref{eq:difference equation for S(0)}.
\end{proof}

By Lemma \ref{lemma for SL_II Voros coeff 2}, $J$ is represented as 
\begin{equation}
{J}(x,t,c,\eta) = - {\rm{log}} \hspace{+.1em}
\biggl\{
\frac{f^{(0)}(x,t,c,\eta) + g^{(0)}(x,t,c,\eta) S^{(0)}(x,t,c,\eta)}
 {f^{(0)}(x,t,c,\eta) + g^{(0)}(x,t,c,\eta) S^{(0)}(\check{x},t,c,\eta)} 
\biggr\} . \label{eq:cal I}
\end{equation}
In order to know the asymptotic behavior of $J$ when $x$ tends to $\infty$, 
we have to investigate the asymptotic behaviors of $f^{(0)}$, $g^{(0)}$ and $S^{(0)}$.   
The asymptotic behaviors of $f^{(0)}$ and $g^{(0)}$ are given by 
\begin{eqnarray}
f^{(0)}(x,t,c,\eta) & = &
 x + \frac{1}{2} \bigl( \lambda^{(0)} + \Lambda(\lambda^{(0)},\nu^{(0)}) \bigr) 
\nonumber \\
&  & \hspace{+.8em}
+ \hspace{+.2em} \bigl( \nu^{(0)} - \frac{5}{8}{\lambda^{(0)}}^2 
+ \frac{1}{4} \lambda^{(0)} \Lambda(\lambda^{(0)},\nu^{(0)})
+ \frac{3}{8}{\Lambda(\lambda^{(0)},\nu^{(0)})}^2 \bigr) x^{-1} \nonumber \\
&  & \hspace{+.8em}
+ \hspace{+.2em} {O}(x^{-2}) ,  \\
g^{(0)}(x,t,c,\eta) & = & - \eta^{-1} 
\Bigl\{
x^{-1} + \frac{1}{2} \bigl( \lambda^{(0)} + \Lambda(\lambda^{(0)},\nu^{(0)}) \bigr) x^{-2}  
\nonumber \\
&  & \hspace{+.8em}
 + \hspace{+.2em} \bigl( \frac{3}{8} {\lambda^{(0)}}^2 
 + \frac{1}{4}\lambda^{(0)} \Lambda(\lambda^{(0)},\nu^{(0)})
 + \frac{3}{8}\Lambda(\lambda^{(0)},\nu^{(0)})^2 \bigr)  x^{-3} \Bigr\} \nonumber \\
&  & \hspace{+.8em}
 + \hspace{+.2em} {O}(x^{-4}) .  
\end{eqnarray}
(Here we take the branch $(x - \Lambda(\lambda, \nu))^{-\frac{1}{2}}
(x - \lambda)^{-\frac{1}{2}}$ $\sim$ $x^{-1}$ as $x$ tends to $\infty$. 
We note that the choice of the 
branch of  $(x - \Lambda(\lambda, \nu))^{-\frac{1}{2}} (x - \lambda)^{-\frac{1}{2}}$ 
does not affect to \eqref{eq:cal I}.)
The behavior of $S^{(0)}$ follows from Lemma \ref{lemma for behaviors of S}: 
\begin{eqnarray}
S^{(0)}(x,t,c,\eta) & = & 
 + \Bigl( \eta x^2 + \frac{1}{2} t \eta + (c \eta - 1) x^{-1} + {\cal O}(x^{-2}) \Bigr)  , \\
S^{(0)}(\check{x},t,c,\eta) & = & 
 - \Bigl( \eta x^2 + \frac{1}{2} t \eta + (c \eta + 1) x^{-1} + {\cal O}(x^{-2}) \Bigr) .
\end{eqnarray}
Hence, we have the following asymptotic behavior of $J$:
\begin{equation}
{J} = - {\rm{log}} \Bigl\{ \frac{1}{2}(\nu^{(0)} - {\lambda^{(0)}}^2 - \frac{1}{2}t) x^{-2} \Bigr\}
 + {\cal O}(x^{-1}) . \label{eq:behavior of J}
\end{equation}
Because we can compute the integral of $S_{-1}$ explicitly as
\begin{eqnarray}
\int_{\gamma_x} S_{-1}(x,t,c) dx & = &
\frac{2}{3}(x^2 + 2 \lambda_0 x + 3 \lambda_0^2 + t)^{\frac{3}{2}} - 
2 \lambda_0 (x + \lambda_0)(x^2 + 2 \lambda_0 x + 3 \lambda_0^2 + t)^{\frac{1}{2}}  
\nonumber \\
&  &
+ \hspace{+.2em} c \hspace{+.2em} 
{\rm{log}}
\biggl\{
\frac
{2(x + \lambda_0) + 2 (x^2 + 2 \lambda_0 x + 3 \lambda_0^2 + t)^{\frac{1}{2}}}
{2(x + \lambda_0) - 2 (x^2 + 2 \lambda_0 x + 3 \lambda_0^2 + t)^{\frac{1}{2}}}
\biggr\}  \nonumber \\
& = & \frac{2}{3} x^3 + t x + 2 \hspace{+.1em} c \hspace{+.2em} {\rm log} \hspace{+.1em} x \hspace{+.1em}
- \frac{4}{3} \hspace{+.1em} \lambda_0^3 - c \hspace{+.1em} {\rm log} \hspace{+.1em} 
\Bigl(- \frac{2\lambda_0^2 + t}{4} \Bigr) + O(x^{-1}), 
\label{eq:behavior of integrated S_-1}
\end{eqnarray}
we also have the following asymptotic behavior of $J_{-1}$: 
 \begin{eqnarray}
\eta \hspace{+.1em} {J}_{-1} & = &  
- \frac{4}{3} \eta \hspace{+.1em} (\lambda_0(t,c)^3 - \lambda_0(t,c-\eta^{-1})^3) \nonumber \\
& &
+ \hspace{+.2em} c \hspace{+.2em} 
{\rm{log}}
\biggl\{
\frac
{2 \lambda_0(t,c-\eta^{-1})^2 + t}
{2 \lambda_0(t,c)^2 + t}
\biggr\} \nonumber \\
& &
+ \hspace{+.2em}  
{\rm{log}}
\biggl\{ - 
\frac
{4 x^2}
{2 \lambda_0(t,c-\eta^{-1})^2 + t}
\biggr\} + {O}(x^{-1}) . \label{eq:behavior of J_-1}
\end{eqnarray}
Furthermore, since
$S^{(0)}_0 = - \frac{1}{2} \frac{1}{S_{-1}} \frac{\partial S_{-1}}{\partial x}$ is single valued at $x = a_1$
and has a singularity of the form 
$S^{(0)}_0 \sim - \frac{1}{4} (x - a_1)^{-1}$ as $x$ tends to $a_1$, we have 
\[
\int_{\gamma_{x}} S^{(0)}_0 dx = 
2 \pi i \hspace{+.2em} {\rm Res}_{x = a_1} S^{(0)}_0  
 = - \frac{\pi i}{2} .  
\]
This implies that 
\begin{equation}
{J}_0 = 0. \label{eq:behavior of J_0}
\end{equation}
Making use of \eqref{eq:behavior of J}, \eqref{eq:behavior of J_-1} 
and \eqref{eq:behavior of J_0}, we have the following asymptotic behavior:
\begin{eqnarray}
{J} - \eta \hspace{+.1em}  {J}_{-1} - {J}_{0} & = &
\frac{4}{3} \eta \hspace{+.1em} (\lambda_0(t,c)^3 - \lambda_0(t,c-\eta^{-1})^3) \nonumber \\
& &
- \hspace{+.1em} c \hspace{+.1em} \eta \hspace{+.2em} 
{\rm{log}}
\biggl\{
\frac
{2 \lambda_0(t,c-\eta^{-1})^2 + t}
{2 \lambda_0(t,c)^2 + t}
\biggr\} \nonumber \\
& &
- \hspace{+.2em} {\rm{log}} \Bigl\{ 
\frac
{2{\lambda^{(0)}}(t,c,\eta)^2 + t - 2\nu^{(0)}(t,c,\eta)}{2 \lambda_0(t,c-\eta^{-1})^2 + t}  \Bigr\}
 + {O}(x^{-1}) . 
\label{eq:behavior of J - eta J_-1 - J_0}
\end{eqnarray}
Therefore, by taking the limit $x \rightarrow \infty$ 
in both sides of \eqref{eq:difference of V(0) in limit form},  
we obtain the difference equation for $V^{(0)}$ as follows:  
\begin{prop}
The Voros coefficient $V^{(0)}(t,c,\eta)$ satisfies the following difference equation formally:
\begin{eqnarray}
\hspace{-1.5em}
2V^{(0)}(t,c,\eta) - 2V^{(0)}(t,c - \eta^{-1},\eta) & = &
\frac{4}{3} \eta \hspace{+.1em} 
\bigl( \lambda_0(t,c)^3 - \lambda_0(t,c-\eta^{-1})^3 \bigr) \nonumber \\
& &
- \hspace{+.1em} c \hspace{+.1em} \eta \hspace{+.1em} 
{\rm{log}}
\biggl\{
\frac
{2 \lambda_0(t,c-\eta^{-1})^2 + t}
{2 \lambda_0(t,c)^2 + t}
\biggr\} \nonumber \\
& &
- {\rm{log}} \Bigl\{ 
\frac
{2{\lambda^{(0)}}(t,c,\eta)^2 + t - 2\nu^{(0)}(t,c,\eta)}{2 \lambda_0(t,c-\eta^{-1})^2 + t}  \Bigr\} .
\label{eq:difference equation2}
\end{eqnarray}
\end{prop}

\noindent
Although it is difficult to solve \eqref{eq:difference equation2} explicitly, 
we can obtain the following fact from this relation.
\begin{prop} \label{computation of V(0)(infinity)}
The limit $V^{(0)}(\infty,c,\eta)$ of $V^{(0)}(t,c,\eta)$  
as $t$ tends to $\infty$ along the $P$-Stokes curve 
$\Gamma$ in Figure \ref{fig:Riemann surface of sqrt Delta} 
is represented explicitly as follows:
\begin{equation}
V^{(0)}(\infty,c,\eta) = - \frac{1}{2} \sum_{n = 1}^{\infty} 
 \frac{2^{1-2n}-1}{2n(2n-1)}B_{2n}(c \hspace{+.1em} \eta)^{1-2n} .
\label{eq:V(0)(infinity)}
\end{equation}
\end{prop}
\begin{proof}
Since $2V^{(0)} - U^{(0)}$ is independent of $t$ and $U^{(0)} \rightarrow 0$ as $t \rightarrow \infty$, 
the limit 
\[
V^{(0)}(\infty,c,\eta) = V_0^{(0)}(\infty,c) + \eta^{-1} V_1^{(1)}(\infty,c)  
+ \eta^{-2} V_2^{(2)}(\infty,c) + \cdots
\]
is well-defined. 
Moreover, because it follows from \eqref{eq:homogenity of V} in Appendix that 
$V^{(0)}$ is invariant under the scaling $(c,\eta) \mapsto (r^{-1}c,r\eta)$, 
$V^{(0)}(\infty,c,\eta)$ is written in the form 
$\sum_{n=1}^{\infty} v_n (c \hspace{+.1em} \eta)^{-n}$ where $v_n \in {\mathbb C}$ is independent of $c$ . 
Hence it suffices to show that $2V^{(0)}(\infty,c,\eta)$ satisfies the 
difference equation \eqref{eq:difference equation} by Lemma \ref{lemma for Voros coeff}. 
Because the following asymptotic behaviors for $t \rightarrow \infty$
\begin{eqnarray}
\lambda_0(t,c)^3 - \lambda_0(t,c-\eta^{-1})^3 & = &  
- \frac{3}{4} \eta^{-1} + {O}(t^{-\frac{3}{2}}) ,  \\
2 \lambda_0(t,c)^2 + t & = &  
- \sqrt{2} \hspace{+.1em} i c t^{-\frac{1}{2}} + {O}(t^{-2})  ,  \\
2 \lambda^{(0)}(t,c,\eta)^2 + t - 2\nu^{(0)}(t,c,\eta) & = &  
\frac{i}{\sqrt{2}} (1 - 2 c \eta) \eta^{-1} t^{-\frac{1}{2}} + {O}(t^{-2})  , 
\end{eqnarray}
hold by Lemma \ref{behavior of coeff of 1-parameter solution}, 
by taking the limit $t \rightarrow \infty$ in both sides of \eqref{eq:difference equation2}, 
we have
\[
2 V^{(0)}(\infty,c,\eta) - 2 V^{(0)}(\infty,c - \eta^{-1},\eta) = 
-1 + \hspace{+.1em} c \hspace{+.1em} \eta \hspace{+.1em} {\rm{log}} 
\Bigl(1+ \frac{1}{c \hspace{+.1em} \eta - 1} \Bigr) 
-  {\rm{log}} 
\Bigl( 1 + \frac{1}{2(c \hspace{+.1em} \eta - 1)} \Bigr) .
\]
Thus the proof of Proposition \ref{computation of V(0)(infinity)} is completed.
\end{proof}

Since $2V - U = 2V^{(0)} - U^{(0)}$ is independent on $t$, we obtain 
\[
2V^{(0)}(t,c,\eta) - U^{(0)}(t,c,\eta) = 2 V^{(0)}(\infty,c,\eta)
\]
by taking the limit $t \rightarrow \infty$. 
(Note that $U^{(0)} \rightarrow 0$ as $t \rightarrow \infty$.)
Hence Theorem \ref{main theorem 2} is derived from Proposition \ref{proposition for 2V - U}
and Proposition \ref{computation of V(0)(infinity)} immediately. 
\begin{cor} \label{connection formula for 2V - U} 
\begin{equation}
{\cal S}\bigl[ {{e}}^{2V - U}\bigl|_{\rm{arg}c = \frac{\pi}{2} - \varepsilon} \bigr] = 
(1 + {{e}}^{2 \pi i c \eta}) \hspace{+.2em} 
{\cal S}\bigl[ {{e}}^{2V - U}\bigl|_{{\rm{arg}}c = \frac{\pi}{2} + \varepsilon} \bigr] ,
\end{equation}
where $\varepsilon$ is a sufficiently small positive number.
\end{cor}
\begin{proof}
Because $2V - U = W$ holds, 
Corollary \ref{connection formula for 2V - U} follows from 
Corollary \ref{connection formula for PII Voros} immediately.
\end{proof}

As we will see in Section 6, the $t$-independent quantity $2V - U$ appears 
in the expression of the Stokes multipliers of ($SL_{\rm II}$). 
Corollary \ref{connection formula for 2V - U} will be used 
in the derivation of the connection formulas for 
the parametric Stokes phenomenon of 1-parameter solutions 
in Section 6.4. 


\section{The Stokes multipliers of ($SL_{\rm II}$) around $x = \infty$ and 
the parametric Stokes phenomena of ($P_{\rm II}$)}

To seek the connection formulas for $\psi_{\pm, {\rm IM}}$ on Stokes curves of ($SL_{\rm II}$) 
emanating from the double turning point $x = \lambda_0$, 
we make use of the local transformation theory 
established in \cite{AKT Painleve WKB}
(cf.\ \cite{KT WKB Painleve I}, \cite{KT WKB Painleve III} and \cite{Takei Painleve} also),
which reduces ($SL_{\rm II}$) to the following canonical equation ($Can$) 
in a neighborhood of $x = \lambda_0$:  
\[
({{Can}}) : \hspace{+.2em}
\Bigl( \frac{\partial^2}{\partial {\tilde{x}}^2} - \eta^2 Q_{\rm{can}} \Bigr) \tilde{\psi} = 0  , 
\]
where
\begin{eqnarray*}
Q_{\rm{can}} = Q_{\rm{can}}(\tilde{x},\tilde{E},\tilde{\sigma},\tilde{\rho},\eta) & = & 
4 \tilde{x}^2 + \eta^{-1} \tilde{E}  
 + \frac{\eta^{-\frac{3}{2}} \tilde{\rho}}{\tilde{x} - \eta^{-\frac{1}{2}} \tilde{\sigma}}  
+ \eta^{-2} \frac{3}{4(\tilde{x} - \eta^{-\frac{1}{2}} \tilde{\sigma})^2}  ,  \\
\tilde{E} & = & {\tilde{\rho}}^2 - 4 {\tilde{\sigma}}^2 . 
\end{eqnarray*}
Note that one peculiar feature of ($SL_{\rm II}$) is that 
the location of the double turning point $x = \lambda_0$ 
coincides with the leading term of 
the apparent singular point $x = \lambda$ of ($SL_{\rm II}$). 
The canonical equation ($Can$) shares the same structure.

Later we need to consider the deformation equation ($D_{\rm can}$) of ($Can$) 
which is given by the following: 
\begin{eqnarray*}
(D_{\rm{can}}) : \frac{\partial \tilde{\psi}}{\partial \tilde{t}}  = 
A_{\rm{can}} \frac{\partial \tilde{\psi}}{\partial \tilde{x}}  
- \frac{1}{2} \frac{\partial A_{\rm{can}}}{\partial \tilde{x}} \tilde{\psi} ,
\end{eqnarray*}
\begin{eqnarray*}
A_{\rm{can}} = \frac{1}{2(\tilde{x} - \eta^{-\frac{1}{2}} \tilde{\sigma})} . 
\end{eqnarray*}
The compatibility condition of the system ($Can$) and ($D_{\rm can}$) is 
represented as the following Hamiltonian system ($H_{\rm can}$):
\begin{eqnarray*}
(H_{\rm{can}}) \hspace{+.2em} 
 \left\{
\begin{array}{ll}
\displaystyle \frac{d \tilde{\sigma}}{d \tilde{t}} = - \eta \tilde{\rho} , \\[+1.em]
\displaystyle \frac{d \tilde{\rho}}{d \tilde{t}} = - 4 \eta \tilde{\sigma} .
\end{array} \right. 
\end{eqnarray*}
A general solution of ($H_{\rm can}$) is given by 
\begin{eqnarray}
 \left\{
\begin{array}{ll}
{\tilde{\sigma}} (\tilde{t},A,B,\eta) =  
A \hspace{+.2em} {{e}}^{2 \eta \tilde{t}} + B \hspace{+.2em} {{e}}^{- 2 \eta \tilde{t}} , \\[+.5em]
{\tilde{\rho}} (\tilde{t},A,B,\eta) =  
-2 A \hspace{+.2em} {{e}}^{2 \eta \tilde{t}} + 2 B \hspace{+.2em} {{e}}^{- 2 \eta \tilde{t}}  ,
\end{array} \right.  \label{eq:canonical Painleve solution} 
\end{eqnarray}
where $A$ and $B$ are free parameters. 
Taking appropriate parameters $A$ and $B$, we will obtain a full order correspondence 
between the WKB solutions $\psi_{\pm, {\rm IM}}$ of ($SL_{\rm II}$) and ($D_{\rm II}$) 
and some WKB solutions of ($Can$) and ($D_{\rm can}$) through the transformation theory.


\subsection{Transformation of ($SL_{\rm II}$) near the double turning point $x = \lambda_0$}

First we review the transformation theory from ($SL_{\rm II}$) to ($Can$). 
\begin{theo} [{\cite[Theorem3.1]{AKT Painleve WKB}, \cite[Theorem2.1]{KT WKB Painleve III}}]
\label{transformation theory wrt x} 
There exist a neighborhood ${\cal U}$ of $x = \lambda_0$ and a formal series 
\begin{eqnarray}
x_{\rm{II}} & = & x_{\rm{II}}(x,t,c,\eta;\alpha) \nonumber \\ 
& = &  
x_{\rm{II}}^{(0)}(x,t,c,\eta)   
+ \alpha \eta^{-\frac{1}{2}} x_{\rm{II}}^{(1)}(x,t,c,\eta)  
{e}^{\eta \phi_{\rm II}}  
+ (\alpha \eta^{-\frac{1}{2}})^2 x_{\rm{II}}^{(2)}(x,t,c,\eta) 
{e}^{2 \eta \phi_{\rm II}} + \cdots  ,  \hspace{+1.7em}
\label{eq:x_II}  
\end{eqnarray}
\[
x_{\rm{II}}^{(k)}(x,t,c,\eta)  = 
x_0^{(k)}(x,t,c) + \eta^{-1}x_1^{(k)}(x,t,c) + \eta^{-2}x_2^{(k)}(x,t,c) + \cdots , 
\]
satisfying (i) $\sim$ (iv) below. \\[-.5em]
 
\noindent   
(i)  $x_\ell^{(k)}(x,t,c)$ ($k \ge 0 , \ell \ge 0$) are holomorphic in $x \in {\cal U}$ and also in $t$.

\noindent 
(ii)  $x_0^{(0)}(x,t,c)$ satisfies that 
\begin{eqnarray}
x^{(0)}_0(\lambda_0,t,c) & = & 0  \label{eq:vanishing of x_0 at lambda_0} ,  \\
\frac{\partial x^{(0)}_0}{\partial x}(\lambda_0,t,c) & \ne & 0 ,
\end{eqnarray}

\noindent  
(iii) \begin{equation}
x_0^{(k)}(x,t,c)  =  0 \hspace{+.5em} (k \ge 1) .
\label{eq:x(k)_0 = 0}
\end{equation}

\noindent  
(iv)  \begin{eqnarray}
\hspace{-.5em}
Q_{\rm II}(x,t,c,\eta;\alpha) & = &  
\Bigl( \frac{\partial x_{\rm II}}{\partial x} \Bigr)^2   
Q_{\rm can}(x_{\rm II}, E, \sigma, \rho, \eta) 
 -\frac{1}{2} \eta^{-2} \bigl\{ x_{\rm II} ; x \bigr\} ,
\label{eq:transformation relation for potential}
\end{eqnarray}
where
\begin{eqnarray}
\displaystyle \sigma = \sigma(t,c,\eta;\alpha) & =  &  
\eta^{\frac{1}{2}} x_{\rm II}\bigl( \lambda(t,c,\eta;\alpha),t,c,\eta;\alpha \bigr), \label{eq:sigma} \\[+.5em]
\rho = \rho(t,c,\eta;\alpha) & = &   
 - \eta^{\frac{1}{2}} \frac{\nu(t,c,\eta;\alpha)}
{\frac{\partial x_{\rm II}}{\partial x} \bigl( \lambda(t,c,\eta;\alpha),t,c,\eta;\alpha \bigr)} 
\hspace{+1.em} \nonumber \\
&  & \hspace{+1.em}
- \frac{3}{4} \eta^{-\frac{1}{2}} 
\frac{\frac{\partial^2 x_{\rm II}}{\partial x^2}(\lambda(t,c,\eta;\alpha),t,c,\eta;\alpha)}
{\Bigl( \frac{\partial x_{\rm II}}{\partial x}(\lambda(t,c,\eta;\alpha),t,c,\eta;\alpha) \Bigr)^2}, 
\label{eq:rho} \\[+.5em]
E = E(t,c,\eta;\alpha) & = &  \rho(t,c,\eta;\alpha)^2 - 4 \sigma(t,c,\eta;\alpha)^2 ,  \label{eq:E}
\end{eqnarray}
and $\{ x_{\rm II} ; x \}$ denotes the Schwarzian derivative, i.e.,
\[ 
\{ x_{\rm II} ; x \} = 
\frac{\partial^3 x_{\rm II}}{\partial x^3} \bigg/
{\frac{\partial x_{\rm II}}{\partial x}}   
- \frac{3}{2}  
\biggl(
{\frac{\partial^2 x_{\rm II}}{\partial x^2}} \bigg/
{\frac{\partial x_{\rm II}}{\partial x}}
\biggr)^2 .
\]
\end{theo}

\begin{rem} \normalfont \label{remark for the construction of transformation}
We note that {\cite[Theorem3.1]{AKT Painleve WKB} and 
\cite[Theorem2.1]{KT WKB Painleve III}} deal with the case 
where a 2-parameter solution of ($P_{\rm II}$) substituted into the coefficients of ($SL_{\rm II}$). 
Although 1-parameter solutions are not discussed in those papers, the proof 
of Theorem \ref{transformation theory wrt x} can be done in a similar manner.  
\end{rem}

\begin{rem} \normalfont
In what follows we abbreviate $x_0^{(0)}$ to $x_0$ for simplicity. 
The coefficients $x^{(k)}_{\ell}(x,t,c)$ are uniquely determined once the branch of 
\begin{equation}
x_0(x,t,c) = \biggl[ \int_{\lambda_0}^{x} \sqrt{Q_0(x,t,c)} dx \biggr]^{\frac{1}{2}}  \label{eq:x_0}
\end{equation}
is fixed. We adopt the branch in such a way that
\begin{equation}
x_0(x,t,c) \sim \frac{1}{\sqrt{2}} \hspace{+.2em} \Delta^{\frac{1}{4}} (x - \lambda_0) 
\label{eq:branch of x_0}
\end{equation}
holds as $x$ tends to $\lambda_0$, where the branch of $\Delta^{\frac{1}{4}}$ is taken to be 
the same as in the expressions \eqref{eq:normalized at tau1-2} and \eqref{eq:normalized at infinity-2} 
of $\lambda^{(1)}$. 
\end{rem}

We write down some properties of the formal series $E$, $\sigma$ and $\rho$.  
These properties are used in the construction of $t_{\rm II}$ 
(in Proposition \ref{transformation theory wrt t}) and 
the computations of the Stokes multipliers of ($SL_{\rm II}$). 
\begin{prop} \label{E = 0}
\begin{equation}
E(t,c,\eta;\alpha) = 0. 
\label{eq:E = 0}
\end{equation}
\end{prop}
\begin{proof}
It is shown in \cite[(3.33)]{AKT Painleve WKB} that the following holds: 
\begin{eqnarray}
\frac{E}{4} = {\rm{Res}}_{x = \lambda_0} S_{\rm odd} . 
\label{eq:E = 4 Res S_odd}
\end{eqnarray}
Because ${\rm{Res}}_{x = \lambda_0} S_{\rm odd}$ is independent of $t$ 
by \eqref{eq:conpatibility of SL_II and D_II 2}, 
the coefficients of $e^{k \eta \phi_{\rm II}}$ in 
\[
{\rm{Res}}_{x = \lambda_0} S_{\rm odd} = 
{\rm{Res}}_{x = \lambda_0} S^{(0)}_{\rm odd} + 
\alpha \eta^{-\frac{1}{2}} {\rm{Res}}_{x = \lambda_0} S^{(1)}_{\rm odd}   
{e}^{\eta \phi_{\rm II}} + 
(\alpha \eta^{-\frac{1}{2}})^2 {\rm{Res}}_{x = \lambda_0} S^{(2)}_{\rm odd}   
{e}^{2 \eta \phi_{\rm II}} + \cdots
\]
must vanish for all $k \ge 1$. 
Furthermore, ${\rm{Res}}_{x = \lambda_0} S^{(0)}_{\rm odd} = 0$ follows from 
(ii) of Lemma \ref{lemma for S(0)_odd} below. Thus we have \eqref{eq:E = 0}.
\end{proof}

The formal series $\sigma$ and $\rho$ of \eqref{eq:sigma} and \eqref{eq:rho}
are expanded as follows: 
\begin{eqnarray}
\sigma(t,c,\eta;\alpha)  =  \eta^{\frac{1}{2}} \bigl\{
\sigma^{(0)}(t,c,\eta) + \alpha \eta^{-\frac{1}{2}} \sigma^{(1)}(t,c,\eta)  
{e}^{\eta \phi_{\rm II}}  
+ (\alpha \eta^{-\frac{1}{2}})^2 \sigma^{(2)}(t,c,\eta) 
{e}^{2 \eta \phi_{\rm II}} + \cdots  \bigr\}, \hspace{+1.em} 
\label{eq:expansion of sigma}   
\end{eqnarray}
\\[-2.8em]
\begin{eqnarray}
\rho(t,c,\eta;\alpha)  =  \eta^{\frac{1}{2}} \bigl\{
\hspace{+.1em} \rho^{(0)}(t,c,\eta)   
+ \alpha \eta^{-\frac{1}{2}} \hspace{+.1em} \rho^{(1)}(t,c,\eta)  
{e}^{\eta \phi_{\rm II}}  
+ (\alpha \eta^{-\frac{1}{2}})^2 \hspace{+.1em} \rho^{(2)}(t,c,\eta) 
{e}^{2 \eta \phi_{\rm II}} + \cdots   \bigr\},  \hspace{+1.em}
\label{eq:expansion of rho} 
\end{eqnarray}
\begin{eqnarray*}
\sigma^{(k)}(t,c,\eta) & = & \sigma^{(k)}_0(t,c) + \eta^{-1} \sigma^{(k)}_1(t,c) 
+ \eta^{-2} \sigma^{(k)}_2(t,c) + \cdots ,  \\
\rho^{(k)}(t,c,\eta) & = & \rho^{(k)}_0(t,c) + \eta^{-1} \rho^{(k)}_1(t,c) 
+ \eta^{-2} \rho^{(k)}_2(t,c) + \cdots .
\end{eqnarray*}
We know that 
\begin{equation}
\rho^{(0)}(t,c,\eta)^2 - 4 \hspace{+.2em} \sigma^{(0)}(t,c,\eta)^2 = 0 
\label{eq:E(0)=0}
\end{equation}
as a consequence of Proposition \ref{E = 0}.
\begin{prop}\label{sigma(0) = rho(0) = 0}
\begin{equation}
\sigma^{(0)}(t,c,\eta) = \rho^{(0)}(t,c,\eta) = 0  \label{eq:sigma(0) = rho(0) = 0} .
\end{equation}
\end{prop}

To show Proposition \ref{sigma(0) = rho(0) = 0}, we recall the following lemma: 
\begin{lemm} [{\cite[Theorem 4.4]{KT iwanami}, 
\cite[Theorem 1.1, Theorem 1.2, Proposition 1.4]{KT WKB Painleve I}}] 
\label{lemma for S(0)_odd} 
(i) There exists a formal power series 
$z(x,t,c,\eta) = z_0(x,t,c) + \eta^{-1} z_1(x,t,c) + \eta^{-2} z_2(x,t,c) + \cdots$ 
whose coefficients are holomorphic at $x = \lambda_0$ satisfying 
\begin{equation}
Q^{(0)}_{\rm II}(x,t,c,\eta) =  
\Bigl( \frac{\partial z}{\partial x}(x,t,c,\eta) \Bigr)^2 
\Bigl\{ 4 z(x,t,c,\eta)^2 + \eta^{-2} \frac{3}{4 z(x,t,c,\eta)^2} \Bigr\}
-\frac{1}{2} \eta^{-2} \bigl\{ z(x,t,c,\eta) ; x \bigr\} .
\label{eq:transformation relation for z}
\end{equation}
The coefficients $z_{\ell}(x,t,c)$ are uniquely determined once we fix the branch of 
\begin{equation}
z_0(x,t,c) = \biggl[ \int_{\lambda_0}^{x} \sqrt{Q_0(x,t,c)} dx \biggr]^{\frac{1}{2}}.  \label{eq:z_0}
\end{equation}
Moreover, 
\begin{equation}
z_1(x,t,c) = 0.   \label{eq:z_1}
\end{equation}

\noindent
(ii) All the coefficients of the formal power series 
$S_{\rm odd}^{(0)}$ and $S^{(0)}_{\rm odd} \big/ (x - \lambda^{(0)})$ 
are holomorphic at $x = \lambda_0$. 

\noindent
(iii) The formal series $z(x,t,c,\eta)$ in (i) satisfies the following:
\begin{equation}
S^{(0)}_{\rm odd}(x,t,c,\eta) = 2 \eta \hspace{+.2em} z(x,t,c,\eta) \frac{\partial z}{\partial x}(x,t,c,\eta) . 
\label{eq:transformation relation of Riccati top series} 
\end{equation}

\end{lemm}
\noindent
Since  
\[
S^{(0)}_{\rm odd}(\lambda^{(0)}(t,c,\eta),t,c,\eta) = 0 
\]
follows from (ii) of Lemma \ref{lemma for S(0)_odd}, we obtain 
\begin{equation}
z(\lambda^{(0)}(t,c,\eta),t,c,\eta) = 0    
\label{eq:z(lambda(0)) = 0}
\end{equation}
from \eqref{eq:transformation relation of Riccati top series}
and the invertibility of the formal power series $\frac{\partial z}{\partial x}(\lambda^{(0)})$.
Now we show that the following lemma 
which is a generalization of Proposition \ref{sigma(0) = rho(0) = 0}:
\begin{lemm} \label{lemma for x(0) and z}
If the branch of $z_0$ is taken as $z_0 = x_0$, then we have 
\begin{equation}
x^{(0)}_{\ell}(x,t,c) = z_{\ell}(x,t,c) ,  \label{eq:x(0)_n = z_n} 
\end{equation}
\begin{equation}
\sigma^{(0)}_{\ell}(t,c) = \rho^{(0)}_{\ell}(t,c) = 0 ,   \label{eq:sigma(0)_n = rho(0)_n = 0}
\end{equation}
for all $\ell \ge 0$.
\end{lemm}
\begin{proof}[Proof of Lemma \ref{lemma for x(0) and z}]
We prove Lemma \ref{lemma for x(0) and z} by induction.
Because $\sigma^{(0)}_0 = x_0(\lambda_0) = 0$ and 
${\rho^{(0)}_0}^2 - 4 \hspace{+.2em} {\sigma^{(0)}_0}^2 = 0$ 
follow from \eqref{eq:x_0} and \eqref{eq:E(0)=0}, the claim for $\ell = 0$ holds. 
Next we assume that the claims are true for $\ell = 0, \cdots , L-1$ 
for a positive integer $L \ge 1$.
Since the formal power series $x_{\rm II}^{(0)}(x,t,c,\eta)$ and $z(x,t,c,\eta)$ satisfy 
\begin{eqnarray}
Q^{(0)}_{\rm II}(x,t,c,\eta) & = &  
\Bigl( \frac{\partial x^{(0)}_{\rm II}}{\partial x}(x,t,c,\eta) \Bigr)^2 
Q_{\rm can}
\bigl(x^{(0)}_{\rm II}(x,t,c,\eta),
E^{(0)}(t,c,\eta),\sigma^{(0)}(t,c,\eta),\rho^{(0)}(t,c,\eta),\eta \bigr) 
 \nonumber  \\
&   &  
 - \frac{1}{2} \eta^{-2} \bigl\{ x^{(0)}_{\rm II}(x,t,c,\eta) ; x \bigr\} ,    
\label{eq:transformation relation for x(0)}
\end{eqnarray}
and \eqref{eq:transformation relation for z} respectively, 
$x_L^{(0)}$ and $z_L$ satisfy the following differential equations respectively:  
\begin{eqnarray}
8 x_0 \frac{\partial x_0}{\partial x} 
\bigl( x_0 \frac{\partial x^{(0)}_L}{\partial x} + x^{(0)}_L \frac{\partial x_0}{\partial x} \bigr)
& = &  r_L(x,t,c),  \label{eq:construction of transformation function1}
\\
8 z_0 \frac{\partial z_0}{\partial x} 
\bigl( z_0 \frac{\partial z_L}{\partial x} + z_L \frac{\partial z_0}{\partial x} \bigr)
& = & \hat{r}_L(x,t,c).  \label{eq:construction of transformation function2}
\end{eqnarray}
Here $r_L$ (resp. $\hat{r}_L$) is written by 
$x_0, \cdots ,x_{L-1}$, 
$\sigma^{(0)}_0, \cdots , \sigma^{(0)}_{L-1}$ and 
$\rho^{(0)}_0. \cdots , \rho^{(0)}_{L-1}$
(resp. $z_0, \cdots ,z_{L-1}$, 
$\sigma^{(0)}_0, \cdots , \sigma^{(0)}_{L-1}$ and 
$\rho^{(0)}_0. \cdots , \rho^{(0)}_{L-1}$). 
Therefore $r_L = \hat{r}_L$ holds under the assumption of the induction.
Hence it follows from \eqref{eq:construction of transformation function1} 
and  \eqref{eq:construction of transformation function2} that 
$x_0 (x^{(0)}_L - z_L)$ equals some constant which is independent of $x$.  
The holomorphy of $x^{(0)}_L$ and $z_L$ at $x = \lambda_0$ implies 
that the constant must be 0. Thus we obtain $x^{(0)}_L = z_L$, and 
$\sigma_L^{(0)} = \rho^{(0)}_L = 0$ follows from 
\eqref{eq:z(lambda(0)) = 0} and \eqref{eq:E(0)=0}. 
\end{proof}

We note that  
\begin{equation}
x^{(0)}_1(x,t,c) = 0 ,   
\label{eq:x(0)_1} 
\end{equation}
\begin{equation}
S^{(0)}_{\rm odd}(x,t,c,\eta) = 
2 \eta \hspace{+.2em} x^{(0)}_{\rm II}(x,t,c,\eta) 
\frac{\partial x^{(0)}_{\rm II}}{\partial x}(x,t,c,\eta) ,
\label{eq:transformation relation for Riccati(0)} 
\end{equation}
follow from Lemma \ref{lemma for S(0)_odd} and Lemma \ref{lemma for x(0) and z}. 

Unfortunately, Theorem \ref{transformation theory wrt x} is not sufficient to 
determine the connection formula for $\psi_{\pm, {\rm IM}}$ on Stokes curves 
emanating from $x = \lambda_0$ in all orders. 
For that purpose we consider an extended transformation of the 
simultaneous equations ($SL_{\rm II}$) and ($D_{\rm II}$), 
which is constructed in \cite{KT WKB Painleve III}. 
(See Proposition \ref{transformation between holonomic systems} below.) 
We introduce the formal series $t_{\rm II}$ 
which plays a role of the transformation of the independent variable of ($H_{\rm can}$).
\begin{prop} [{\cite[Lemma 3.3]{KT WKB Painleve III}}] \label{transformation theory wrt t}
There exists a formal series 
\begin{eqnarray}
t_{\rm II} & = & t_{\rm II}(t,c,\eta;\alpha) \nonumber \\
& = & 
t_{\rm{II}}^{(0)}(t,c,\eta)   
+ \alpha \eta^{-\frac{1}{2}} t_{\rm{II}}^{(1)}(t,c,\eta)  
{e}^{\eta \phi_{\rm II}}  
+ (\alpha \eta^{-\frac{1}{2}})^2 t_{\rm{II}}^{(2)}(t,c,\eta) 
{e}^{2 \eta \phi_{\rm II}} + \cdots ,   
\label{eq:t_II}  
\end{eqnarray}
\begin{equation}
t^{(k)}_{\rm II}(t,c,\eta) = t^{(k)}_0(t,c) + \eta^{-1} t^{(k)}_1(t,c) + 
\eta^{-2} t^{(k)}_2(t,c) + \cdots ,  
\end{equation}
satisfying (i) and (ii) below. 

\noindent
(i) 
\begin{eqnarray}
t^{(0)}_0(t,c) & = & \frac{1}{2} \phi_{\rm II}(t,c) . \label{eq:t_0}  \\
t^{(0)}_1(t,c) & = & 0 \label{eq:t(0)_1} . \label{eq:t(0)_1} \\
t_0^{(k)}(t,c) & = & 0 \hspace{+.5em} (k \ge 1) .
\label{eq:t(k)_0 = 0}
\end{eqnarray}

\noindent
(ii) 
\begin{eqnarray}
\sigma(t,c,\eta;\alpha) & = & \frac{1}{\sqrt{2}}\alpha \hspace{+.2em} 
{e}^{2 \eta \hspace{+.2em} t_{\rm II}(t,c,\eta;\alpha)} .
\label{eq:transformation of sigma}   \\
\rho(t,c,\eta;\alpha) & = & - \sqrt{2} \alpha \hspace{+.2em} 
{e}^{2 \eta \hspace{+.1em} t_{\rm II}(t,c,\eta;\alpha) } .
\label{eq:transformation of rho}   
\end{eqnarray}
\end{prop}
\begin{rem} \normalfont 
The right-hand sides of \eqref{eq:transformation of sigma} and \eqref{eq:transformation of rho}
are the formal series obtained from the solution \eqref{eq:canonical Painleve solution} of ($H_{\rm can}$) 
through the substitution $\tilde{t} \mapsto t_{\rm II}$ and $(A,B) \mapsto (\frac{1}{\sqrt{2}} \alpha, 0)$.
The choice of the parameters originate from Proposition \ref{sigma(0) = rho(0) = 0}
and the following relations: 
\begin{eqnarray*}
\sigma^{(1)}_0(t,c) & = & 
x^{(1)}_0(\lambda_0,t,c) +
\frac{\partial x^{(0)}_0}{\partial x}(\lambda_0,t,c) \hspace{+.2em} \lambda^{(1)}_0(t,c) 
\hspace{+.2em} = \hspace{+.2em} \frac{1}{\sqrt{2}} , \\
\rho^{(1)}_0(t,c) & = & 
\frac{\nu_0(t,c) \Bigl(
\frac{\partial x^{(1)}_0}{\partial x} (\lambda_0,t,c) + 
\frac{\partial^2 x_0}{\partial x^2} (\lambda_0,t)  
 \hspace{+.2em} \lambda^{(1)}_0(t,c)
 \Bigr) }
{\Bigl( \frac{\partial x_0}{\partial x}
 (\lambda_0,t,c) \Bigr)^2} 
- \hspace{+.2em} \frac{\nu^{(1)}_0(t,c)}
{\frac{\partial x_0}{\partial x} (\lambda_0,t,c)}
 \hspace{+.2em} = \hspace{+.2em}  - \sqrt{2} .
\end{eqnarray*}
\end{rem}

In what follows we abbreviate $t^{(0)}_0(t,c)$ to $t_0(t,c)$. 
The formal power series $t^{(0)}_{\rm II}$ is determined by the following relation:
\begin{equation}
\hspace{+.2em} \sigma^{(1)}(t,c,\eta)  = 
\frac{1}{\sqrt{2}} \hspace{+.2em}
{\rm exp} \Bigl( 2 \eta \bigl( t_{\rm II}^{(0)}(t,c,\eta) - t_0(t,c)  \bigr)  \Bigr). 
\label{eq:definition of t(0)} 
\end{equation}
(We note that $\sigma^{(0)}(t,c,\eta) = 0$ by Proposition \ref{sigma(0) = rho(0) = 0}.)
The relation 
\begin{equation}
1 + \alpha \eta^{-\frac{1}{2}} \frac{\sigma^{(2)}(t,c,\eta)}{\sigma^{(1)}(t,c,\eta)}  
{e}^{\eta \phi_{\rm II}} + 
(\alpha \eta^{-\frac{1}{2}})^2 \frac{\sigma^{(3)}(t,c,\eta)}{\sigma^{(1)}(t,c,\eta)}  
{e}^{2 \eta \phi_{\rm II}} + \cdots
 = 
{\rm exp} \Bigl( 2 \eta \bigl( t_{\rm II}(t,c,\eta) - 
 t_{\rm II}^{(0)}(t,c,\eta)  \bigr)  \Bigr)  
\label{eq:definition of t_II(k)}
\end{equation}
also determines $t^{(k)}_{\rm II}$ ($k \ge 1$) uniquely.

Making use of the above transformation theory, 
we obtain a correspondence between WKB solutions  
of ($SL_{\rm II}$) and ($Can$) satisfying their deformation equations. 
\begin{prop} [{\cite[Proposition3.1]{KT WKB Painleve III}}] 
\label{transformation between holonomic systems}
Let $\tilde{\psi}(\tilde{x},\tilde{t},\eta;\alpha)$ be a WKB solution of ($Can$) and ($D_{\rm can}$) 
with the solution 
\begin{eqnarray}
 \left\{
\begin{array}{ll}
{\tilde{\sigma}} (\tilde{t},\eta;\alpha) =  
\frac{1}{\sqrt{2}} \alpha \hspace{+.2em} {{e}}^{2 \eta \tilde{t}}  ,  \\
{\tilde{\rho}} (\tilde{t},\eta;\alpha) =  
-\sqrt{2} \alpha \hspace{+.2em} {{e}}^{2 \eta \tilde{t}}  , 
\end{array} \right.  \label{eq:canonical Painleve 1-parameter solution} 
\end{eqnarray}
of ($H_{\rm can}$) substituted into their coefficients. 
If we define 
\begin{equation}
\psi(x,t,c,\eta;\alpha) = 
\Bigl(\frac{\partial x_{\rm II}}{\partial x}(x,t,c,\eta;\alpha) \Bigr)^{-\frac{1}{2}}
\tilde{\psi} \bigl(
x_{\rm II}(x,t,c,\eta;\alpha),t_{\rm II}(t,c,\eta;\alpha),\eta;\alpha \bigr),
\end{equation}
then $\psi$ satisfies both ($SL_{\rm II}$) and ($D_{\rm II}$) 
in a neighborhood of $x = \lambda_0$.
\end{prop}

Proposition \ref{transformation between holonomic systems} can be verified easily 
by using the relation 
\begin{equation}
\frac{1}{2(x - \lambda)}
 \frac{\partial x_{\rm II}}{\partial x} 
- \frac{\partial x_{\rm II}}{\partial t} 
-  \frac{1}{2(x_{\rm II} - \eta^{- \frac{1}{2}}\sigma)}
 \frac{\partial t_{\rm II}}{\partial t}  = 0 ,
\label{eq:relation for transformation series}
\end{equation} 
which is proved in \cite[(3.52)]{KT WKB Painleve III}. 


\subsection{WKB solutions of ($Can$) and ($D_{\rm can}$)}

In order to use Proposition \ref{transformation between holonomic systems}, 
we construct WKB solutions which satisfy both ($Can$) and ($D_{\rm can}$) 
with \eqref{eq:canonical Painleve 1-parameter solution} are substituted into the coefficients. 
In what follows we designate 
\begin{eqnarray*}
Q_{\rm can}(\tilde{x},\tilde{t},\eta;\alpha) & = &
Q_{\rm can}(\tilde{x},\tilde{E}(\tilde{t},\eta;\alpha),
\tilde{\sigma}(\tilde{t},\eta;\alpha),\tilde{\rho}(\tilde{t},\eta;\alpha),\eta)  \nonumber \\
& = &
4 {\tilde{x}}^2 
- \eta^{-1} \frac{\sqrt{2} \alpha \eta^{-\frac{1}{2}} {e}^{2 \eta \tilde{t}}}
{\tilde{x} - \frac{1}{\sqrt{2}} \alpha \eta^{-\frac{1}{2}} {e}^{2 \eta \tilde{t}}} 
+ \eta^{-2} 
\frac{3}{4 \bigl(\tilde{x} - \frac{1}{\sqrt{2}} \alpha \eta^{-\frac{1}{2}} 
{e}^{2 \eta \tilde{t}} \bigr)^2} , 
\\
A_{\rm can}(\tilde{x},\tilde{t},\eta;\alpha) & = & 
\frac{1}{2(\tilde{x} - \frac{1}{\sqrt{2}} \alpha \eta^{-\frac{1}{2}} 
{e}^{2 \eta \tilde{t}})} .
\end{eqnarray*}
(Note that $\tilde{E}(\tilde{t},\eta;\alpha) = \tilde{\rho}(\tilde{t},\eta;\alpha)^2 -  
4 \tilde{\sigma}(\tilde{t},\eta;\alpha)^2 = 0$.)
$Q_{\rm can}$ is expanded as follows:
\begin{equation}
Q_{\rm can}(\tilde{x},\tilde{t},\eta;\alpha) = Q^{(0)}_{\rm can}(\tilde{x},\eta)   
+ \alpha \eta^{-\frac{1}{2}} Q^{(1)}_{\rm can}(\tilde{x},\eta) {e}^{2 \eta \tilde{t}}  
+(\alpha \eta^{-\frac{1}{2}})^2 Q^{(2)}_{\rm can}(\tilde{x},\eta) {e}^{4 \eta \tilde{t}} 
+ \cdots  ,  \label{eq:expansion of Q_can} \\[-1.em]
\end{equation}
\begin{eqnarray}
Q^{(0)}_{\rm can}(\tilde{x},\eta) & = & 4 {\tilde{x}}^2 + \eta^{-2} \frac{3}{4 {\tilde{x}}^2} \hspace{+.2em} ,  
\label{eq:Q(0)_can}  \\
Q^{(k)}_{\rm can}(\tilde{x},\eta) & = & - \eta^{-1} 
\Bigl( \frac{1}{\sqrt{2}} \Bigr)^{k-2} \frac{1}{{\tilde{x}}^k}   
+  
\eta^{-2} \Bigl( \frac{1}{\sqrt{2}} \Bigr)^k \hspace{+.2em}
\frac{3 \hspace{+.1em} (k + 1)}{4 {\tilde{x}}^{k+2}}  \hspace{+1.em}  (k \ge 1) \hspace{+.2em} .
\label{eq:Q(k)_can}
\end{eqnarray}
As in Section 4, we can construct WKB solutions of ($Can$) 
in the following form: 
\begin{equation}
\tilde{\psi}_{\pm} (\tilde{x},\tilde{t},\eta;\alpha) = 
\frac{1}{\sqrt{\tilde{S}_{\rm odd}}} \hspace{+.2em} {\rm exp} \pm
\Bigl\{
\eta \int_0^{\tilde{x}} \tilde{S}_{-1} d \tilde{x} + \int_{\infty}^{\tilde{x}} \bigl( 
\tilde{S}_{\rm odd} - \eta \tilde{S}_{-1}
\bigr) d \tilde{x}
\Bigr\} ,
\label{eq:WKB solution of Can and D_can}
\end{equation}
where 
\begin{equation}
\tilde{S}_{\rm odd} = \tilde{S}_{\rm odd}(\tilde{x},\tilde{t},\eta;\alpha) = 
\tilde{S}_{\rm odd}^{(0)}(\tilde{x},\eta)   
+ \alpha \eta^{-\frac{1}{2}} \tilde{S}_{\rm odd}^{(1)}(\tilde{x},\eta) {e}^{2 \eta \tilde{t}}  
+(\alpha \eta^{-\frac{1}{2}})^2 \tilde{S}_{\rm odd}^{(2)}(\tilde{x},\eta) {e}^{4 \eta \tilde{t}}
+ \cdots 
\label{eq:expansion of tilde S_odd}
\end{equation}
is the odd part (in the sense of Remark \ref{odd part}) of a formal solution 
\begin{equation}
\tilde{S} = \tilde{S}(\tilde{x},\tilde{t},\eta;\alpha) = 
\tilde{S}^{(0)}(\tilde{x},\eta)   
+ \alpha \eta^{-\frac{1}{2}} \tilde{S}^{(1)}(\tilde{x},\eta) {e}^{2 \eta \tilde{t}}  
+(\alpha \eta^{-\frac{1}{2}})^2 \tilde{S}^{(2)}(\tilde{x},\eta) {e}^{4 \eta \tilde{t}}
+ \cdots 
\label{eq:expansion of tildeS}
\end{equation}
of the associated Riccati equation of ($Can$)  
\[
{\tilde{S}}^2 + \frac{\partial \tilde{S}}{\partial \tilde{x}} = 
\eta^2 Q_{\rm can}(\tilde{x},\tilde{t},\eta;\alpha) ,
\]
$\tilde{S}_{\rm odd}^{(k)}$ and $\tilde{S}^{(k)}$ are formal power series of $\eta^{-1}$ 
of the form 
\[
\tilde{S}_{\rm odd}^{(k)}(\tilde{x},\eta) = 
\eta \tilde{S}_{\rm odd,-1}^{(k)}(\tilde{x}) + \tilde{S}_{\rm odd,0}^{(k)}(\tilde{x}) 
+ \eta^{-1} \tilde{S}_{\rm odd,1}^{(k)}(\tilde{x})
+  \cdots   \hspace{+.2em}  (k \ge 0) , 
\]
\[
\tilde{S}^{(k)}(\tilde{x},\eta) = 
\eta \tilde{S}^{(k)}_{-1}(\tilde{x}) + \tilde{S}^{(k)}_{0}(\tilde{x}) + \eta^{-1} \tilde{S}^{(k)}_{1}(\tilde{x})
+  \cdots  \hspace{+.2em}  (k \ge 0) ,
\]
and 
\begin{equation}
\tilde{S}_{-1}(\tilde{x}) = \tilde{S}^{(0)}_{-1}(\tilde{x}) = 2 \tilde{x} .
  \label{eq:tildeS_-1 = 2 tildex}
\end{equation}
It is easy to check the integral in \eqref{eq:WKB solution of Can and D_can} is well-defined.
Since $\tilde{S}^{(k)}$ ($k \ge 0$) satisfy the differential equations 
\begin{equation}
{\tilde{S}}^{(0)2} + \frac{\partial \tilde{S}^{(0)}}{\partial \tilde{x}} =    
\eta^2 Q^{(0)}_{\rm can}(\tilde{x},\eta) ,  \label{eq:tildeS(0)}
\end{equation}
\begin{equation}
2 {\tilde{S}}^{(0)} {\tilde{S}}^{(k)} + \sum_{k_1+k_2=k, \hspace{+.1em} k_j < k} {\tilde{S}}^{(k_1)} {\tilde{S}}^{(k_2)}  
+ \frac{\partial \tilde{S}^{(k)}}{\partial \tilde{x}} = 
\eta^{2} Q^{(k)}_{\rm can}(\tilde{x},\eta)  \hspace{+1.em} (k \ge 1)  ,  \label{eq:tildeS(k)}
\end{equation}
we obtain 
\begin{eqnarray}
{\tilde{S}}^{(0)} (\tilde{x},\eta) & = & 2 \eta \tilde{x} - \frac{1}{2 \tilde{x}} ,  \label{eq:tildeS(0) full} \\
{\tilde{S}}_{\rm odd}^{(0)} (\tilde{x},\eta) & = & 2 \eta \tilde{x} ,  \label{eq:tildeS(0)_odd full} \\[+.2em]
{\tilde{S}}^{(k)}_{-1}(\tilde{x}) & = & 0 \hspace{+1.em} (k \ge 1) ,  \label{eq:tildeS(k)_-1 = 0}
\end{eqnarray}
from \eqref{eq:Q(0)_can} and \eqref{eq:Q(k)_can}.
\begin{lemm} [{\cite[Lemma2]{Takei Painleve}}] \label{WKB solution of Can and D_can}
The formal series ${e}^{\pm \eta \tilde{t}} \tilde{\psi}_{\pm}$
satisfy both ($Can$) and ($D_{\rm can}$).
\end{lemm}

\begin{rem} \normalfont
In view of  \eqref{eq:expansion of tilde S_odd}, 
\eqref{eq:tildeS_-1 = 2 tildex} and \eqref{eq:tildeS(k)_-1 = 0}, 
${e}^{\pm \eta \tilde{t}} \tilde{\psi}_{\pm}$ are expanded as follows:
\begin{equation}
{e}^{\pm \eta \tilde{t}}  \tilde{\psi}_{\pm}(\tilde{x},\tilde{t},\eta) = 
{\tilde{\psi}}_{\pm}^{(0)}(\tilde{x},\tilde{t},\eta) + 
\alpha \eta^{-\frac{1}{2}}{\tilde{\psi}}_{\pm}^{(1)}(\tilde{x},\tilde{t},\eta) {e}^{2 \eta \tilde{t}} + 
(\alpha \eta^{-\frac{1}{2}})^2 {\tilde{\psi}}_{\pm}^{(2)}(\tilde{x},\tilde{t},\eta) 
{e}^{4 \eta \tilde{t}} + \cdots ,  
\label{eq:expansion of tildepsi_pm}  
\end{equation}
\begin{equation}
{\tilde{\psi}}_{\pm}^{(k)}(\tilde{x},\tilde{t},\eta) = 
\eta^{-\frac{1}{2}} \bigl\{
{\tilde{\psi}}_{\pm,0}^{(k)}(\tilde{x}) + \eta^{-1}{\tilde{\psi}}_{\pm,1}^{(k)}(\tilde{x})   
+ \eta^{-2}{\tilde{\psi}}_{\pm,2}^{(k)}(\tilde{x}) + \cdots  \bigr\}
 \hspace{+.2em} {\rm exp}\pm \eta \bigl( \tilde{t} + {\tilde{x}}^2 \bigr) .
\label{eq:tildepsi(k)_pm}
\end{equation}
Especially, 
\begin{equation}
{\tilde{\psi}}_{\pm}^{(0)}(\tilde{x},\tilde{t},\eta) = \frac{1}{\sqrt{\tilde{S}^{(0)}_{\rm odd}}}
\hspace{+.2em} {\rm exp} \pm
\Bigl(
\int_{\infty}^{\tilde{x}} \bigl( 
\tilde{S}^{(0)}_{\rm odd} - \eta \tilde{S}_{-1}
\bigr) d \tilde{x}
\Bigr)
 \hspace{+.2em} {\rm exp}\pm \eta \bigl( \tilde{t} + {\tilde{x}}^2 \bigr) .   
\label{eq:tildepsi(0)}
\end{equation}
\end{rem}

Making use of Proposition \ref{transformation between holonomic systems} and 
Lemma \ref{WKB solution of Can and D_can}, we know that 
\[
 {e}^{\pm \eta \hspace{+.1em} t_{\rm II}(t,c,\eta;\alpha)} 
\Bigl(\frac{\partial x_{\rm II}}{\partial x}(x,t,c,\eta;\alpha) \Bigr)^{-\frac{1}{2}}
\tilde{\psi}_{\pm} \bigl(
x_{\rm II}(x,t,c,\eta;\alpha),t_{\rm II}(t,c,\eta;\alpha),\eta;\alpha \bigr)
\]
satisfy both ($SL_{\rm II}$) and ($D_{\rm II}$) near $x = \lambda_0$.
Therefore, there exist a formal power series $C_{\pm} = C_{\pm}(c,\eta)$ of $\eta^{-1}$ whose 
coefficients are independent of $x$ and $t$ such that 
\begin{eqnarray} 
\psi_{\pm,\rm IM}  = 
C_{\pm} \hspace{+.2em} {e}^{\pm \eta \hspace{+.1em} t_{\rm II}} 
\Bigl(\frac{\partial x_{\rm II}}{\partial x} \Bigr)^{-\frac{1}{2}}
\tilde{\psi}_{\pm} \bigl(
x_{\rm II}, t_{\rm II}, \eta; \alpha \bigr) 
\label{eq:C_pm}
\end{eqnarray}
hold. 
\begin{lemm} \label{lemma for C_pm} 
\begin{equation}
C_{\pm} = {\rm exp} \pm
\Bigl\{ \frac{1}{2}U^{(0)} + \int_{\infty}^{x}\bigl( S^{(0)}_{\rm odd} - \eta S_{-1} \bigr) dx 
- \eta \bigl( t^{(0)}_{\rm I\hspace{-.1em}I} - t_0 \bigr) 
- \eta \bigl( {x^{(0)}_{\rm I\hspace{-.1em}I}}^2 - {x_0}^2 \bigr)
\Bigr\} .
\label{eq:lemma for C_pm}
\end{equation}
\end{lemm}
\begin{rem} \label{remark for normalization} \normalfont
Precisely speaking, we need to specify the normalization of $\psi_{\pm, {\rm IM}}$
(i.e., \hspace{-.2em}the choices of the path of integration of $\int_{a_1}^x S_{-1} \hspace{+.1em} dx$  
and $\int_{\infty}^x ( S_{\rm odd} - \eta S_{-1} ) \hspace{+.1em} dx$ in \eqref{eq:psi_pm infinity}). 
\eqref{eq:lemma for C_pm} is the result when the normalization of $\psi_{\pm, {\rm IM}}$ 
is taken as in Figure \ref{fig:normalization of psi_pm-0}, where 
the red and blue curves designate the path of integration 
of $\int_{a_1}^x S_{-1} dx$ and $\int_{\infty}^x (S_{\rm odd} - \eta S_{-1}) dx$ respectively.
 \begin{figure}[h]
 \begin{center}
 \includegraphics[width=65mm]{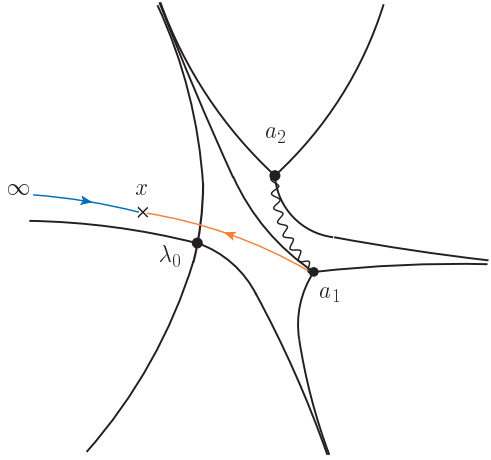}
 \end{center}
 \caption{Normalization of $\psi_{\pm, {\rm IM}}$.}
 \label{fig:normalization of psi_pm-0}
 \end{figure}
\end{rem}
\begin{proof}[Proof of Lemma \ref{lemma for C_pm}]
Expanding the both sides of \eqref{eq:C_pm} to the form like \eqref{eq:expansion of psi_pm,IM} 
and comparing the parts which do not contain $e^{k \eta \phi_{\rm II}}$ ($k \ge 1$), we obtain 
\begin{equation}
\psi_{\pm}^{(0)} = C_{\pm} \Bigl(\frac{\partial x^{(0)}_{\rm II}}{\partial x} \Bigr)^{-\frac{1}{2}} 
\tilde{\psi}_{\pm}(x^{(0)}_{\rm II}, t^{(0)}_{\rm II}, \eta).   \label{eq:psi(0) and psi-tilde(0)}
\end{equation}
Note that, since $C_{\pm}$ is independent of $t$, it does not contain 
$e^{k \eta \phi_{\rm II}}$ $(k \ge 1)$, either. 
By \eqref{eq:integral relation of top terms}, we have
\begin{eqnarray}
t_0(t,c) + x_0(x,t,c)^2 
& = & 
\frac{1}{2} \int_{\tau_1}^{t} \sqrt{\Delta(t,c)} dt + \int_{\lambda_0}^{x} \sqrt{Q_0(x,t,c)}dx  \nonumber \\
& = &
\frac{1}{2} \int_{\tau_1}^{t} \sqrt{\Delta(t,c)} dt + \int_{\lambda_0}^{a_1} \sqrt{Q_0(x,t,c)}dx + 
\int_{a_1}^{x} \sqrt{Q_0(x,t,c)}dx \nonumber \\
& = &
\int_{a_1}^{x} \sqrt{Q_0(x,t,c)}dx \hspace{+.2em} = 
\hspace{+.2em} \int_{a_1}^x S_{-1}(x,t,c) dx .
\label{eq:t_0 + x_0^2}
\end{eqnarray}
(As in the proof of Proposition \ref{integral of top terms}, we note that 
the sign of the right-hand side of \eqref{eq:integral relation of top terms} is $+$.)
Making use of \eqref{eq:psi(0)_pm,IM}, \eqref{eq:tildepsi(0)}, 
\eqref{eq:transformation relation for Riccati(0)}, \eqref{eq:tildeS(0)_odd full} 
and \eqref{eq:t_0 + x_0^2},  \eqref{eq:lemma for C_pm} is derived from 
\eqref{eq:psi(0) and psi-tilde(0)} directly. 
\end{proof}

Let $Z = Z(c,\eta)$ be a formal power series defined by 
\[
\frac{1}{2}U^{(0)} + \int_{\infty}^{x}\bigl( S^{(0)}_{\rm odd} - \eta S_{-1} \bigr) dx 
- \eta \bigl( t^{(0)}_{\rm I\hspace{-.1em}I} - t_0 \bigr) 
- \eta \bigl( {x^{(0)}_{\rm I\hspace{-.1em}I}}^2 - {x_0}^2 \bigr). 
\] 
In the case where a 1-parameter solution of ($H_{\rm II}$) 
is normalized at $t = \infty$, we can determine $Z$ explicitly 
with the aid of the results presented in Appendix.   
\begin{prop} \label{explicit form of C_pm}
Assume that the 1-parameter solution $\lambda_{\infty}$ normalized at $t = \infty$ 
is substituted into the coefficients of ($SL_{\rm II}$) and ($D_{\rm II}$). Then 
\begin{equation}
Z = 0 . \label{eq:Z = 0}
\end{equation}
\end{prop}
\begin{proof}
$Z$ has the form 
\[
Z(c,\eta) = Z_0 + Z_1 (c \hspace{+.1em} \eta)^{-1} 
+ Z_2 (c \hspace{+.1em} \eta)^{-2} + \cdots 
\]
for some $Z_{\ell} \in {\mathbb C}$ which is independent of $c$ because of 
\eqref{eq:homogenity of S_odd}, \eqref{eq:homogenity of U}  
and Proposition \ref{homogenity of transformation series} in Appendix A. 
If $\lambda_{\infty}(t,c,\eta;\alpha)$ is substituted 
into the coefficients of ($SL_{\rm II}$) and ($D_{\rm II}$), 
then the coefficients of $\eta^{- {\ell}}$ in $Z$ are holomorphic in $c$ for all ${\ell} \ge 0$ 
by Proposition \ref{holomorphic dependence of Z on c} in Appendix B. 
Thus we have $Z_{\ell} = 0$ (${\ell} \ge 1$). 
Furthermore, we can confirm that $Z_0 = 0$ easily.
\end{proof}

Therefore we have the following correspondence between 
$\psi_{\pm, {\rm IM}}$ and $\tilde{\psi}_{\pm}$ 
when the 1-parameter solution $\lambda_{\infty}$ normalized at $t = \infty$ is substituted: 
\begin{equation}
\psi_{\pm,\rm IM}(x,t,c,\eta;\alpha)   =  
 {e}^{\pm \eta t_{\rm I\hspace{-.1em}I}(t,c,\eta;\alpha)} 
\Bigl(\frac{\partial x_{\rm I\hspace{-.1em}I}}{\partial x}(x,t,c,\eta;\alpha) \Bigr)^{-\frac{1}{2}}
\tilde{\psi}_{\pm} \bigl(
x_{\rm I\hspace{-.1em}I}(x,t,c,\eta;\alpha),t_{\rm I\hspace{-.1em}I}(t,c,\eta;\alpha),\eta;\alpha \bigr) .   
\label{eq:full order relation between SL_II and Can}
\end{equation}
We then expect that the connection formulas for $\psi_{\pm, {\rm IM}}$ on Stokes curves 
emanating from the double turning point $x = \lambda_0$ should be derived from this relation 
\eqref{eq:full order relation between SL_II and Can} and the connection formulas for $\tilde{\psi}_{\pm}$. 
The latter is explicitly described in the following Proposition \ref{connection formula for tilde psi_pm} 
which is proved in \cite[$\S$3]{Takei Painleve}.
\begin{prop}[{[T1,Proposition 4]}] \label{connection formula for tilde psi_pm}
Let ${\tilde{\psi}}^{\rm J}_{\pm}$ be the Borel sum of $\tilde{\psi}_{\pm}$  
in the region $\rm J$ in Figure \ref{fig:Stokes curve of (Can)} 
($\rm  J = 0, I, I\hspace{-.1em}I, I\hspace{-.1em}I\hspace{-.1em}I$). 
Then on each Stokes curve the following connection formula holds: 
 \begin{figure}[h]
 \begin{center}
 \includegraphics[width=45mm]{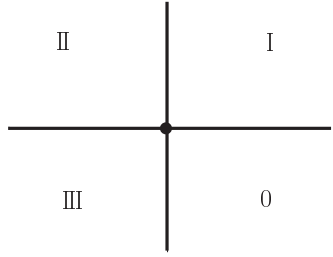}
 \end{center}
 \caption{Stokes curves of ($Can$).}
 \label{fig:Stokes curve of (Can)}
 \end{figure}
\begin{eqnarray}
\begin{cases}
\displaystyle  {\tilde{\psi}}^{\rm 0}_{+} = {\tilde{\psi}}^{\rm I}_{+} + \tilde{m}_{01} {\tilde{\psi}}^{\rm I}_{-} \\
\displaystyle {\tilde{\psi}}^{\rm 0}_{-} = {\tilde{\psi}}^{\rm I}_{-}
\end{cases}
\\
\begin{cases}
{\tilde{\psi}}^{\rm I}_{+} = {\tilde{\psi}}^{\rm I\hspace{-.1em}I}_{+} \\
{\tilde{\psi}}^{\rm I}_{-} = {\tilde{\psi}}^{\rm I\hspace{-.1em}I}_{-}
 + \tilde{m}_{12} {\tilde{\psi}}^{\rm I\hspace{-.1em}I}_{+}   \hspace{+.1em}
\end{cases}
\label{eq:connection formula on iR_+} \\
\begin{cases}
{\tilde{\psi}}^{\rm I\hspace{-.1em}I}_{+} = 
{\tilde{\psi}}^{\rm I\hspace{-.1em}I\hspace{-.1em}I}_{+}  
 + \tilde{m}_{23} {\tilde{\psi}}^{\rm I\hspace{-.1em}I\hspace{-.1em}I}_{-} \\
{\tilde{\psi}}^{\rm I\hspace{-.1em}I}_{-} = 
{\tilde{\psi}}^{\rm I\hspace{-.1em}I\hspace{-.1em}I}_{-}
\end{cases}
\label{eq:connection formula on R_-}   \\
\begin{cases}
{\tilde{\psi}}^{\rm I\hspace{-.1em}I\hspace{-.1em}I}_{+} =   
{\tilde{\psi}}^{\rm 0}_{+}    \\
{\tilde{\psi}}^{\rm I\hspace{-.1em}I\hspace{-.1em}I}_{-} =  
{\tilde{\psi}}^{\rm 0}_{-} + \tilde{m}_{30} {\tilde{\psi}}^{\rm 0}_{+}
\end{cases}
\end{eqnarray}
where
\begin{eqnarray}
\tilde{m}_{01} & = & - i \bigl( \tilde{\rho} + 2 \tilde{\sigma} \bigr) \sqrt{\frac{\pi}{2}} , \\
\tilde{m}_{12} & = & \bigl( \tilde{\rho} - 2 \tilde{\sigma} \bigr) \sqrt{\frac{\pi}{2}}  ,
\label{eq:connection multiplier on iR_+}   \\
\tilde{m}_{23} & = & i \bigl( \tilde{\rho} + 2 \tilde{\sigma} \bigr) \sqrt{\frac{\pi}{2}} ,
\label{eq:connection multiplier on R_-}      \\
\tilde{m}_{30} & = & - \bigl( \tilde{\rho} - 2 \tilde{\sigma} \bigr) \sqrt{\frac{\pi}{2}} .
\end{eqnarray}
\end{prop}
\begin{rem} \normalfont
In Proposition \ref{connection formula for tilde psi_pm} $\tilde{\sigma}$ and $\tilde{\rho}$ 
in the coefficients of ($Can$) should be understood as ordinary parameters 
(possibly depending on $t$ and $\eta$), although $\sigma = \sigma(t,c,\eta;\alpha)$ 
and $\rho = \rho(t,c,\eta;\alpha)$ are infinite series in our situation 
(cf. Remark \ref{Remark on the connection formula on the Stokes curves from double t-pt} below). 
\end{rem}


\subsection{Computation of the Stokes multipliers of ($SL_{\rm II}$)}

In this subsection we compute the Stokes multipliers of ($SL_{\rm II}$) around $x = \infty$ 
before and after the degeneration of the Stokes geometry 
observed when arg \hspace{-.5em} $c = \frac{\pi}{2}$. 
\begin{figure}[h]
 \begin{minipage}{0.5\hsize}
  \begin{center}
   \includegraphics[width=70mm]{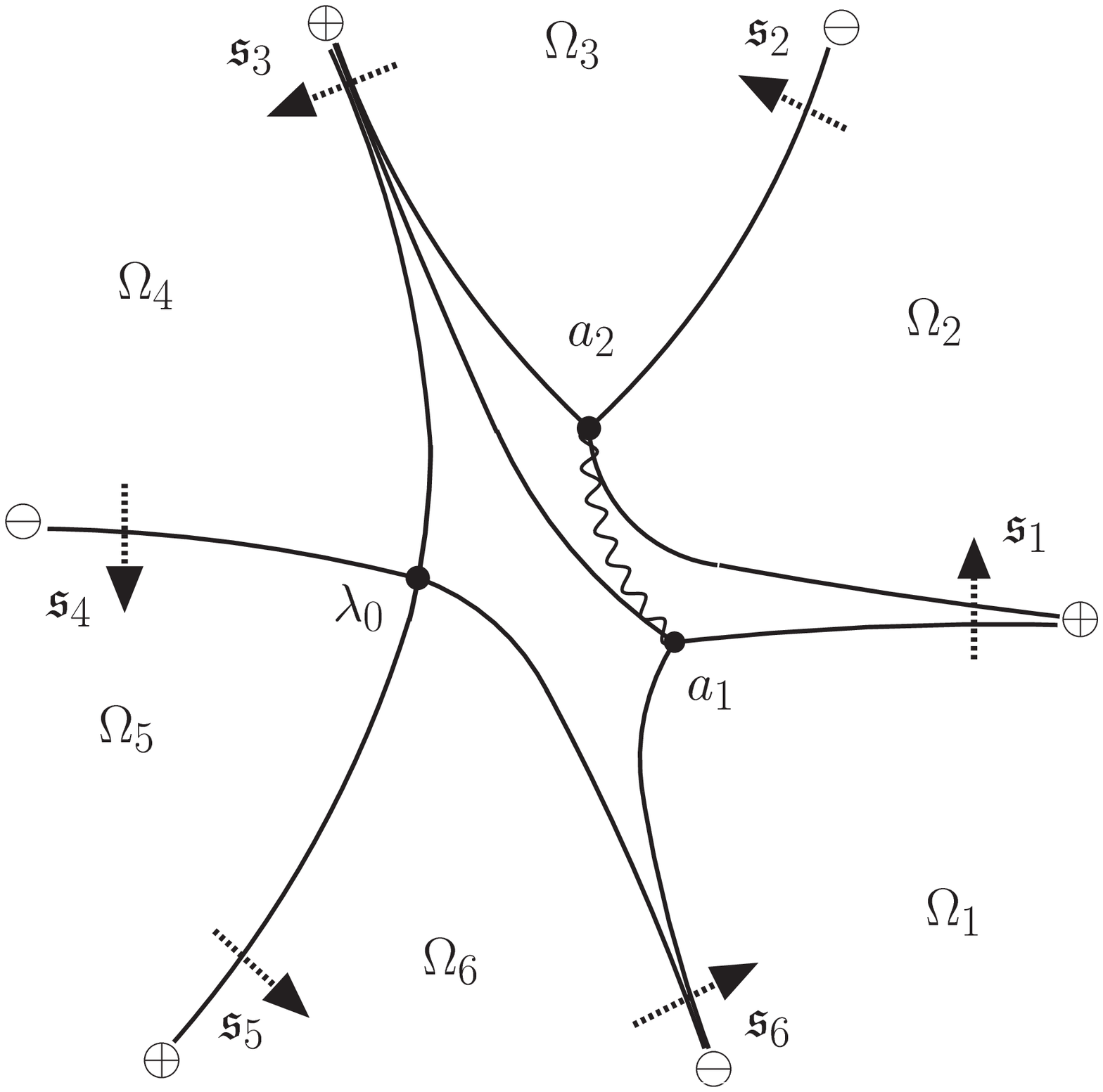}
  \end{center}
  \caption{arg \hspace{-.5em} $c = \frac{\pi}{2} - \varepsilon$}
  \label{fig:Stokes multipliers of SL_II minus}
 \end{minipage}
 \begin{minipage}{0.5\hsize}
  \begin{center}
   \includegraphics[width=70mm]{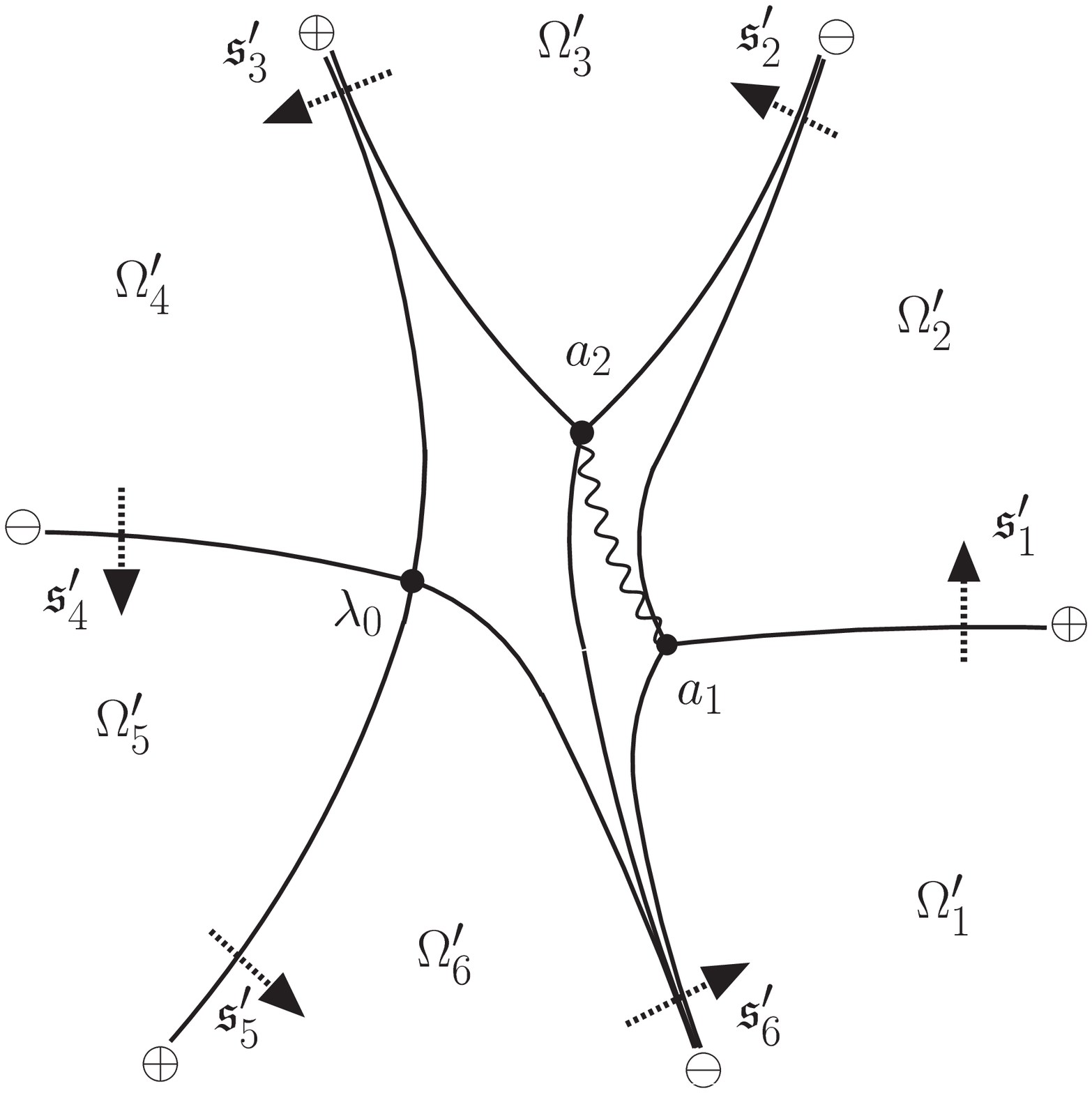}
  \end{center}
  \caption{arg \hspace{-.5em} $c = \frac{\pi}{2} + \varepsilon$}
  \label{fig:Stokes multipliers of SL_II plus}
 \end{minipage}
\end{figure}
The symbols $\oplus$ and $\ominus$ in Figures \ref{fig:Stokes multipliers of SL_II minus} 
and \ref{fig:Stokes multipliers of SL_II plus} 
designate the ``sign of Stokes curves" which is defined similarly to that for 
$P$-Stokes curves (cf.\ Remark \ref{remark for the branch of sqrt-Delta}). 
Note that it follows from Proposition \ref{integral of top terms} that 
${\rm Re} \int_{a_1}^{a_2} \sqrt{Q_0} \hspace{+.1em} dx < 0$  
holds near arg \hspace{-.5em} $c = \frac{\pi}{2}$. 
Since a neighborhood of $x = \infty$ is divided into the six regions 
$\Omega_j$ and $\Omega'_j$ ($1 \le j \le 6$) by Stokes curves 
as in Figures \ref{fig:Stokes multipliers of SL_II minus} and \ref{fig:Stokes multipliers of SL_II plus},  
we obtain six Stokes multipliers around $x = \infty$ 
for arg \hspace{-.5em} $c = \frac{\pi}{2} - \varepsilon$ and 
arg \hspace{-.5em} $c = \frac{\pi}{2} + \varepsilon$, respectively. 
Let ${\mathfrak s}_{j} = {\mathfrak s}_{j}(c;\alpha)$
(resp. ${{\mathfrak s}_j}' = {{\mathfrak s}_j}'(c;\alpha)$) be the Stokes multipliers corresponding 
to the analytic continuation from $\Omega_j$ to $\Omega_{j+1}$
(resp. from $\Omega'_j$ to $\Omega'_{j+1}$) ($1 \le j \le 6$). 
The results of the computations of the Stokes multipliers of ($SL_{\rm II}$)
by using the WKB solutions $(\psi_{+,\rm IM}, \psi_{-, \rm IM})$ 
are as follows: \\[-.5em]

\noindent
\textbf{Stokes multipliers of ($SL_{\rm II}$) around $x = \infty$.} 
\\[-1.em]

\noindent
\textit{(i) If the 1-parameter solution substituted into the coefficients of ($SL_{\rm II}$) 
and ($D_{\rm II}$) is normalized at $\infty$, then we obtain the following}:
\begin{align}
\begin{cases}
{\mathfrak s}_{1}  =  i \hspace{+.3em} (1 + {{e}}^{2 \pi i c \eta}) {{e}}^{U - 2V} \\
{\mathfrak s}_{2}  =  i \hspace{+.3em}{{e}}^{- 2 \pi i c \eta}  {{e}}^{2V - U} \\
{\mathfrak s}_{3}  = 
i \hspace{+.3em} (1 + {{e}}^{2 \pi i c \eta}) {{e}}^{- 2 \pi i c \eta} 
{{e}}^{U - 2V} \hspace{+2.3em} \\ 
{\mathfrak s}_{4}  =  - 2 \sqrt{\pi} \alpha \\
{\mathfrak s}_{5}  =  0 \\
{\mathfrak s}_{6}  =  2 \sqrt{\pi}  \alpha + i \hspace{+.3em} {{e}}^{2V - U} .
\end{cases}
 & \hspace{+.5em}
\begin{cases}
{\mathfrak s}_{1}'  =  i \hspace{+.3em} {{e}}^{U - 2V} \\
{\mathfrak s}_{2}'  =  
i \hspace{+.3em} (1 + {{e}}^{2 \pi i c \eta}) {{e}}^{- 2 \pi i c \eta} 
{{e}}^{2V - U} \hspace{+2.3em} \\
{\mathfrak s}_{3}'  = i \hspace{+.3em} {{e}}^{- 2 \pi i c \eta} {{e}}^{U - 2V} \\ 
{\mathfrak s}_{4}'  =  - 2 \sqrt{\pi} \alpha \\
{\mathfrak s}_{5}'  =  0 \\
{\mathfrak s}_{6}'  =  
2 \sqrt{\pi} \alpha + i \hspace{+.3em} (1 + {{e}}^{2 \pi i c \eta}) {{e}}^{2V - U} .
\end{cases}
\label{eq:list1}
\end{align}

\noindent
\textit{(ii) If the 1-parameter solution substituted into the coefficients of ($SL_{\rm II}$) 
and ($D_{\rm II}$) is normalized at $\tau_1$, 
then we obtain the following}:
\begin{align}
\begin{cases}
{\mathfrak s}_{1}  =  i \hspace{+.3em} (1 + {{e}}^{2 \pi i c \eta}) {{e}}^{U - 2V} \\
{\mathfrak s}_{2}  =  i \hspace{+.3em}{{e}}^{- 2 \pi i c \eta}  {{e}}^{2V - U} \\
{\mathfrak s}_{3}  = 
i \hspace{+.3em} (1 + {{e}}^{2 \pi i c \eta}) {{e}}^{- 2 \pi i c \eta} 
{{e}}^{U - 2V} \hspace{+2.3em} \\ 
{\mathfrak s}_{4}  =  - 2 \sqrt{\pi} \alpha {\rm e}^{W} \\
{\mathfrak s}_{5}  =  0 \\
{\mathfrak s}_{6}  =  2 \sqrt{\pi}  \alpha {\rm e}^{W} + i \hspace{+.3em} {{e}}^{2V - U} .
\end{cases}
 & \hspace{+.5em}
\begin{cases}
{\mathfrak s}_{1}'  =  i \hspace{+.3em} {{e}}^{U - 2V} \\
{\mathfrak s}_{2}'  =  
i \hspace{+.3em} (1 + {{e}}^{2 \pi i c \eta}) {{e}}^{- 2 \pi i c \eta} 
{{e}}^{2V - U} \hspace{+2.3em} \\
{\mathfrak s}_{3}'  = i \hspace{+.3em} {{e}}^{- 2 \pi i c \eta} {{e}}^{U - 2V} \\ 
{\mathfrak s}_{4}'  =  - 2 \sqrt{\pi} \alpha {\rm e}^{W} \\
{\mathfrak s}_{5}'  =  0 \\
{\mathfrak s}_{6}'  =  
2 \sqrt{\pi} \alpha {\rm e}^{W} + i \hspace{+.3em} (1 + {{e}}^{2 \pi i c \eta}) {{e}}^{2V - U} .
\end{cases}
\label{eq:list2}
\end{align}
\textit{Here $\alpha$ is the free parameter contained 
in the 1-parameter solution substituted into 
the coefficients of ($SL_{\rm II}$) and ($D_{\rm II}$), 
$U = U(t,c,\eta;\alpha)$ is given by \eqref{eq:U},
$V = V(t,c,\eta;\alpha)$ is the Voros coefficient \eqref{eq:SL_II Voros coeff} of ($SL_{\rm II}$) 
and $W = W(c,\eta)$ is the $P$-Voros coefficient \eqref{eq:PIIVoros}.} \\[-.5em]

To be precise, the Stokes multipliers are the Borel sum of 
${\mathfrak s}_j$ and ${{\mathfrak s}_j}'$ in \eqref{eq:list1} and \eqref{eq:list2}. 
From now on we demonstrate the computation of the Stokes multipliers of ($SL_{\rm II}$)  
when arg \hspace{-.5em} $c = \frac{\pi}{2} - \varepsilon$ and the 1-parameter solution 
substituted into the coefficients is normalized at $\infty$.  
We demonstrate only the computations of ${\mathfrak s}_1$ and ${{\mathfrak s}_4}$, 
since the other Stokes multipliers can be computed in similar ways.   
We note that, since the two normalizations of 1-parameter solutions introduced in Section 3    
are related as \eqref{eq:relation between two normalizations}, the result \eqref{eq:list2}  
of the computation when the substituted 1-parameter solution is normalized 
at $t = \tau_1$ is obtained from \eqref{eq:list1} by replacement 
$\alpha \mapsto \alpha \hspace{+.1em} e^{W}$. 

\begin{rem} \normalfont
In the computations of the Stokes multipliers the path of normalization of WKB solutions 
$\psi_{\pm, {\rm IM}}$ are taken as in Figures \ref{fig:normalization of psi_pm IM in case s_1} 
$\sim$ \ref{fig:normalization of psi_pm IM in case s_6}. 
The red and blue curves in Figures \ref{fig:normalization of psi_pm IM in case s_1} 
$\sim$ \ref{fig:normalization of psi_pm IM in case s_6} designate the paths of integration 
of $\int_{a_1}^x S_{-1} dx$ and $\int_{\infty}^x (S_{\rm odd} - \eta S_{-1}) dx$ respectively.
  \begin{figure}[h]
  \begin{minipage}{0.5\hsize}
  \begin{center}
  \includegraphics[width=60mm]{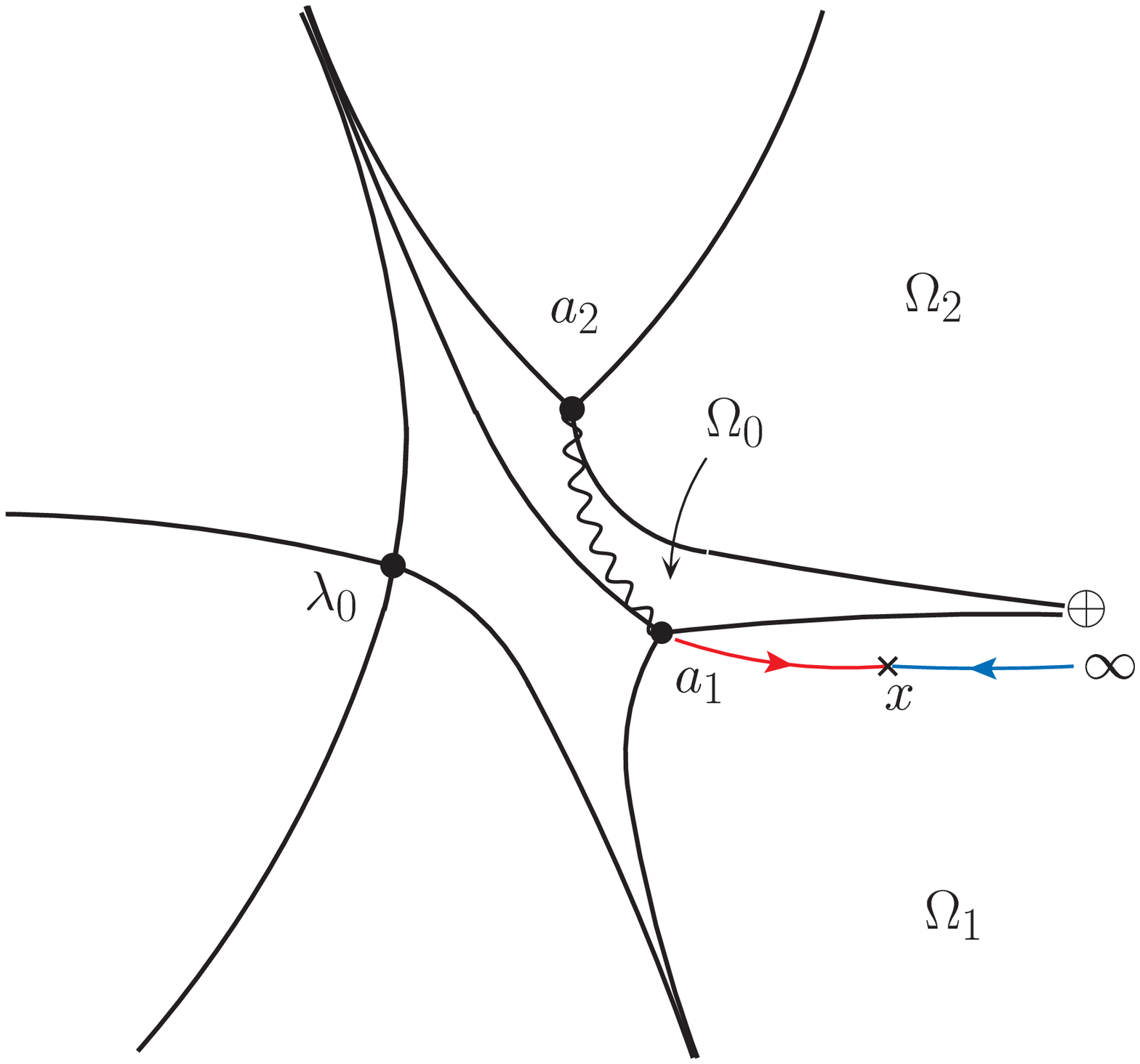}
  \end{center}
  \caption{Normalization path for ${\mathfrak s}_1$.}
  \label{fig:normalization of psi_pm IM in case s_1}
  \end{minipage}
  \begin{minipage}{0.5\hsize}
  \begin{center}
  \includegraphics[width=54mm]{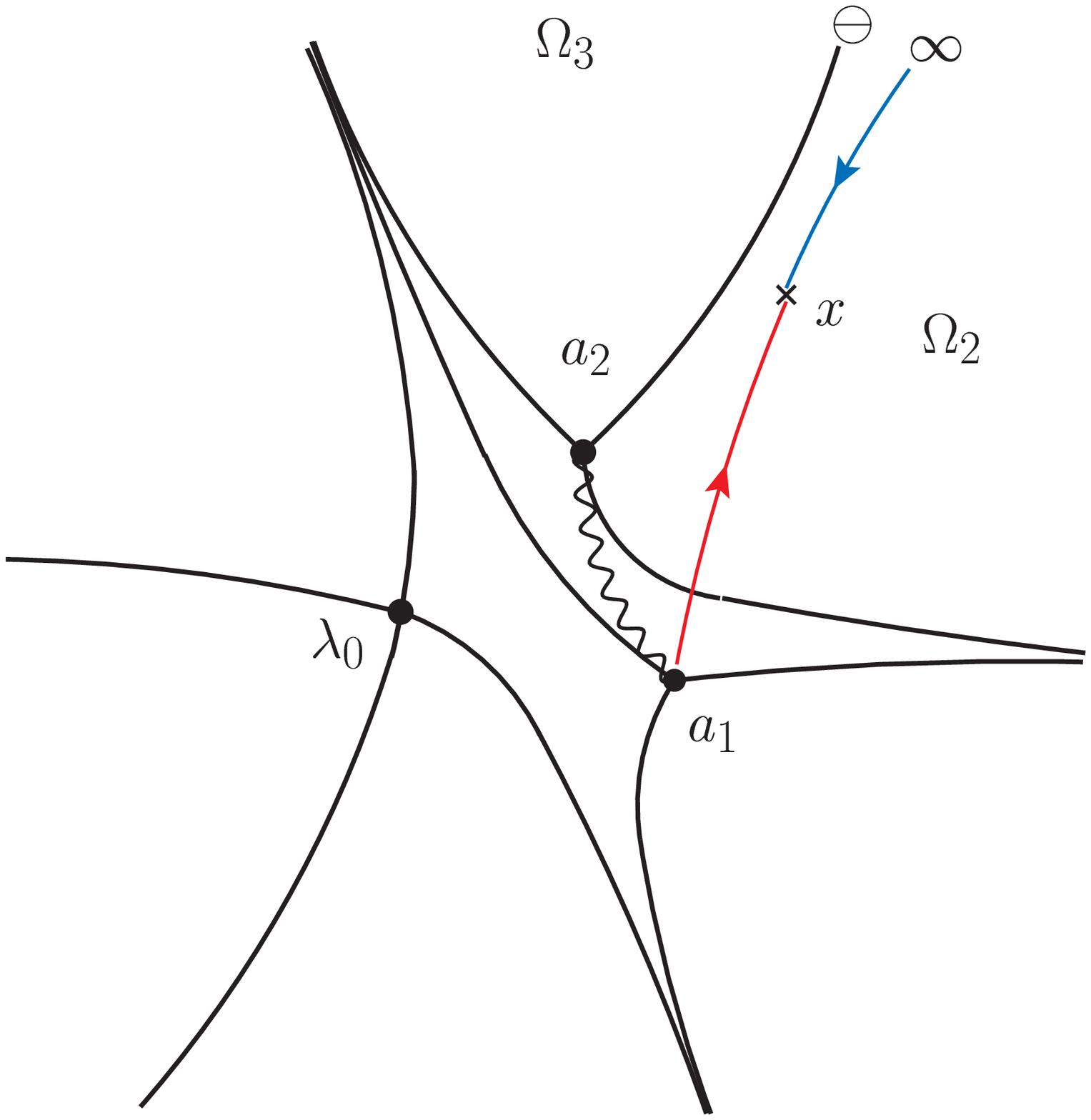}
  \end{center}
  \caption{Normalization path for ${\mathfrak s}_2$.}
  \label{fig:normalization of psi_pm IM in case s_2}
  \end{minipage}
  \end{figure}
  \begin{figure}[h]
  \begin{minipage}{0.5\hsize}
  \begin{center}
  \includegraphics[width=54mm]{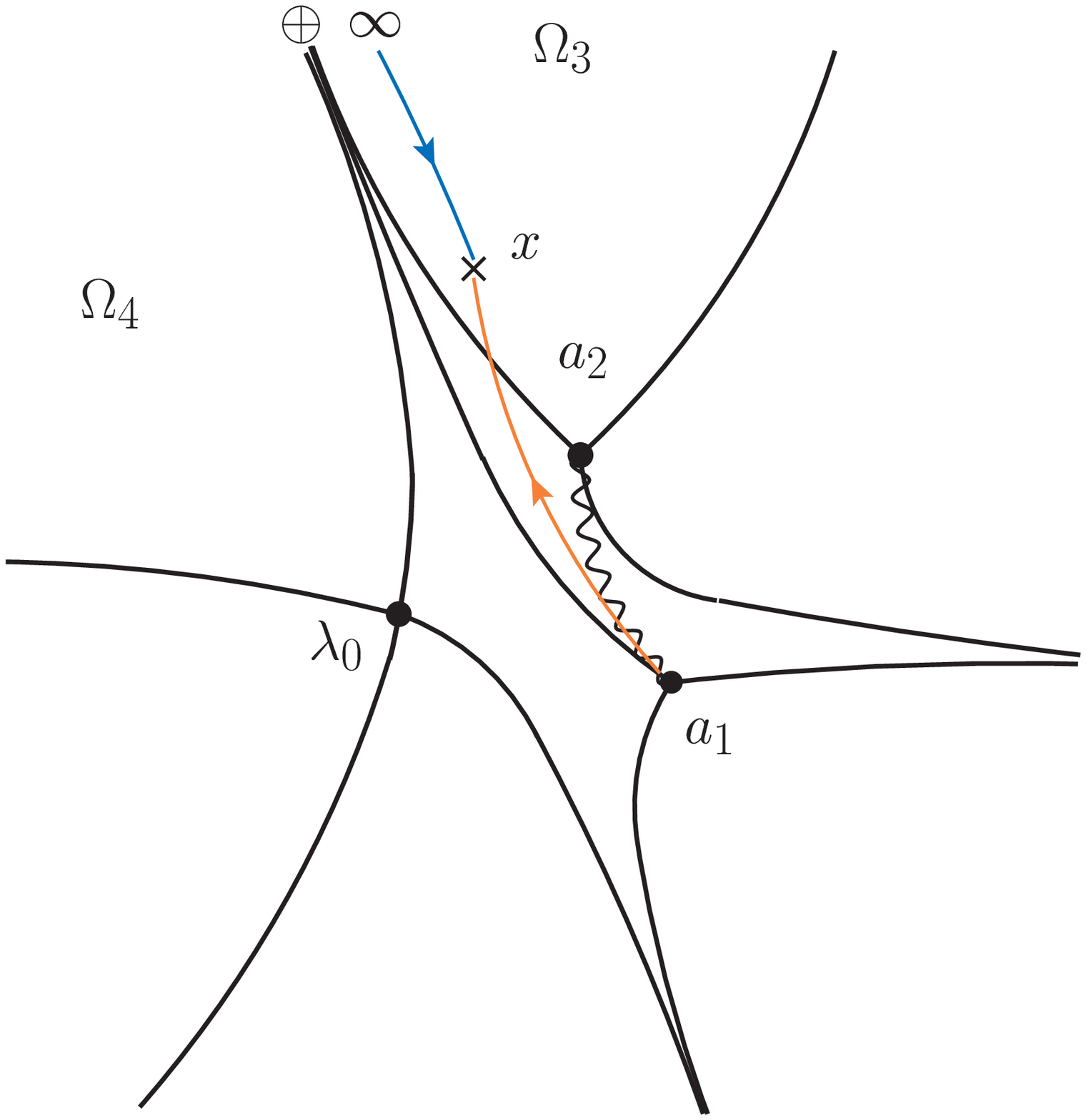}
  \end{center}
  \caption{Normalization path for ${\mathfrak s}_3$.}
  \label{fig:normalization of psi_pm IM in case s_3}
  \end{minipage}
  \begin{minipage}{0.5\hsize}
  \begin{center}
  \includegraphics[width=60mm]{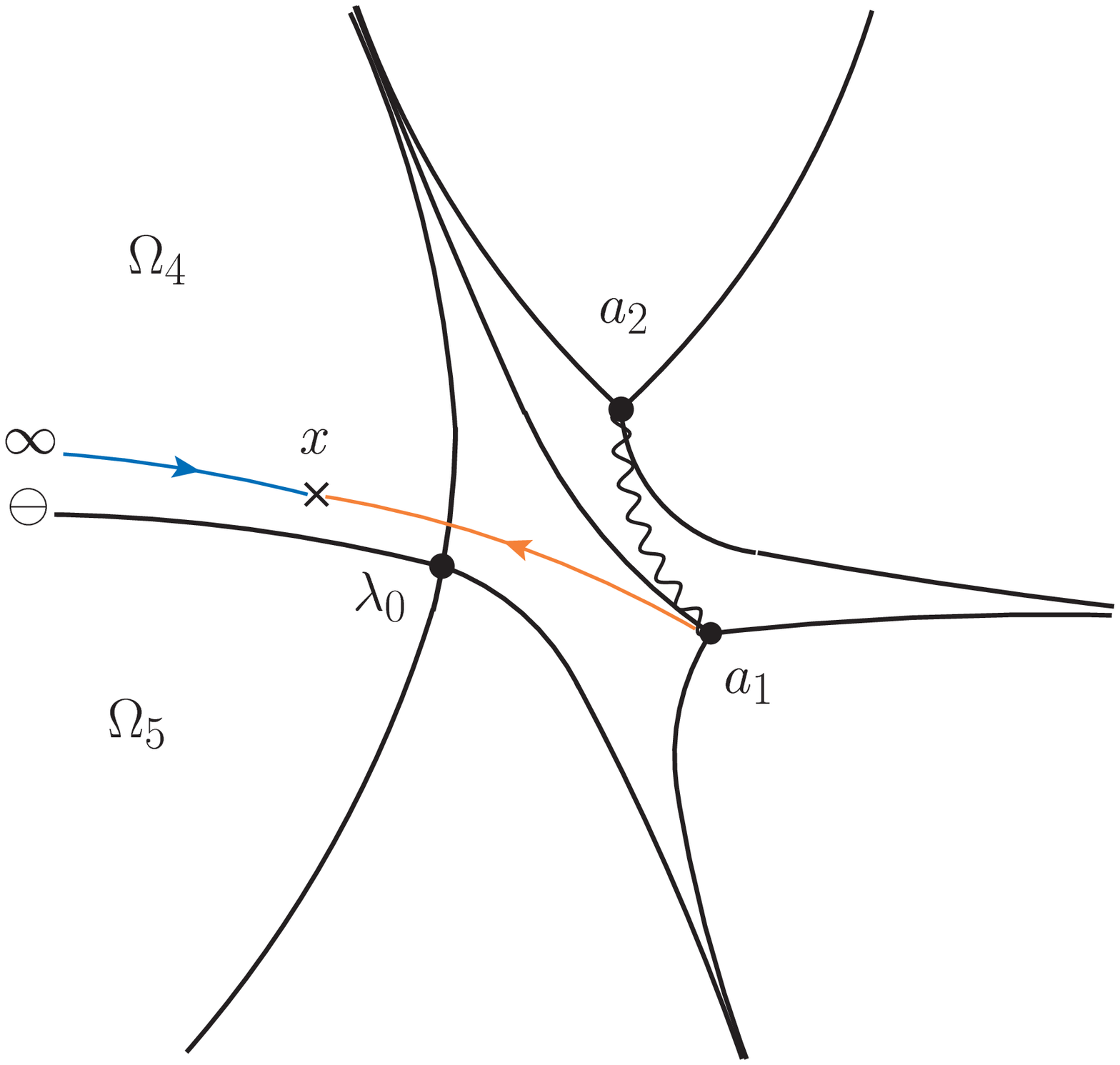}
  \end{center}
  \caption{Normalization path for ${\mathfrak s}_4$.}
  \label{fig:normalization of psi_pm IM in case s_4}
  \end{minipage}
  \end{figure}
  \begin{figure}[h]
  \begin{minipage}{0.5\hsize}
  \begin{center}
  \includegraphics[width=60mm]{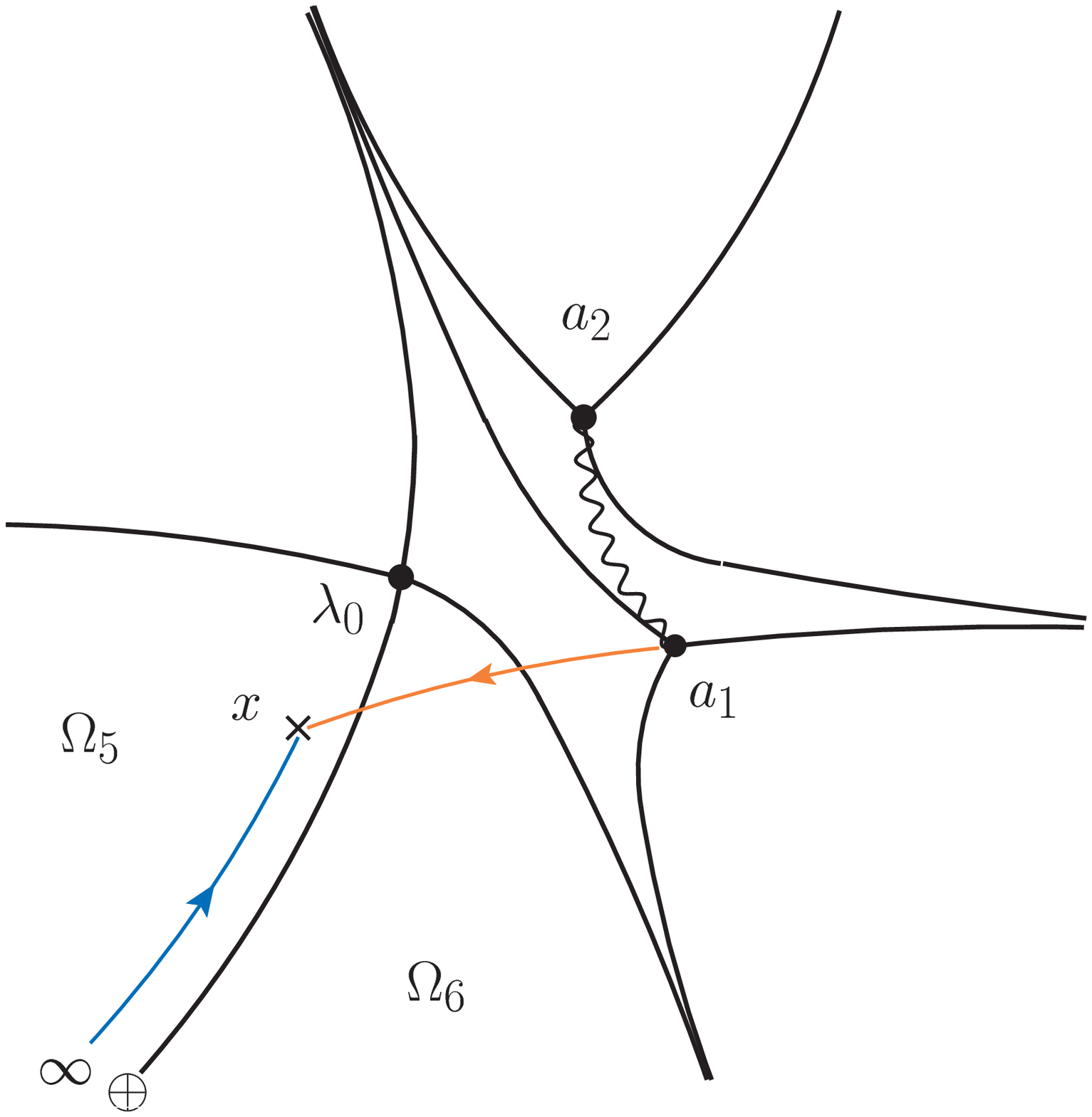}
  \end{center}
  \caption{Normalization path for ${\mathfrak s}_5$.}
  \label{fig:normalization of psi_pm IM in case s_5}
  \end{minipage}
  \begin{minipage}{0.5\hsize}
  \begin{center}
  \includegraphics[width=60mm]{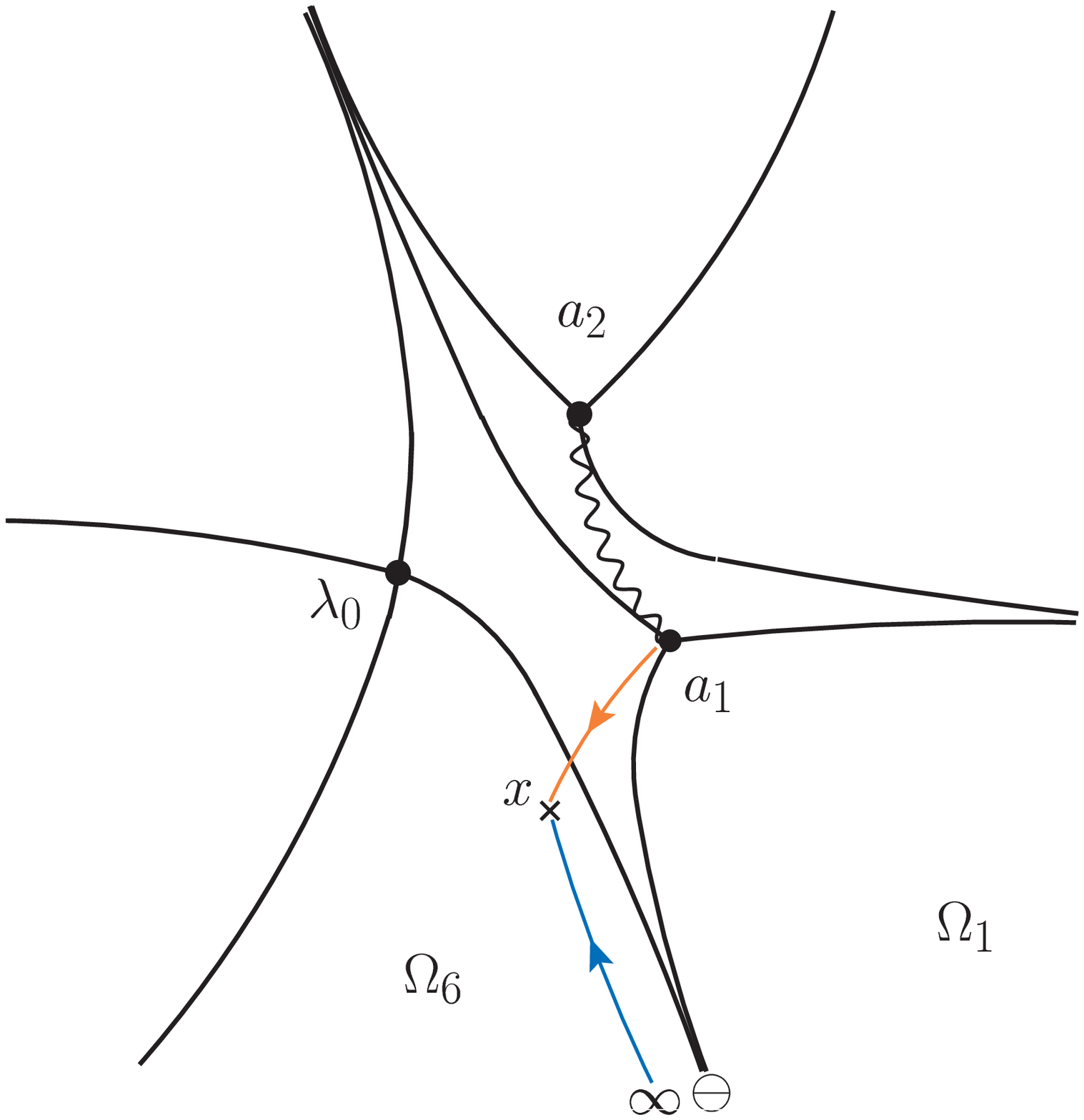}
  \end{center}
  \caption{Normalization path for ${\mathfrak s}_6$.}
  \label{fig:normalization of psi_pm IM in case s_6}
  \end{minipage}
  \end{figure}
\end{rem}

We denote the Borel sum of a WKB solution $\psi$ in a region $\Omega$ 
by $\psi^{\Omega}$. 

\noindent
$\bullet$ \underline{Computation of ${\mathfrak s}_1$.}\hspace{+.1em}
Let $\Omega_0$ be the region pinched by $\Omega_1$ and $\Omega_2$, 
and $\psi_{\pm, a_j}$ be a WKB solution of ($SL_{\rm II}$) normalized at $x = a_j$:  
\begin{eqnarray*}
\psi_{\pm, a_j} = \frac{1}{\sqrt{S_{\rm odd}}} \hspace{+.1em} {\rm exp} \Bigl( 
 \pm \int_{a_j}^{x} S_{\rm odd} dx \Bigr) \hspace{+1.em} (j = 1,2).
\end{eqnarray*}
Voros' connection formula (\cite[$\S$2, Theorem 2.23]{KT iwanami}, \cite{Voros}) 
says that the following relations for the Borel sums of $\psi_{\pm, a_j}$ 
hold on the Stokes curve emanating from $x = a_j$ ($j = 1,2$): 
\begin{eqnarray*}
\begin{cases}
\displaystyle \psi^{\Omega_1}_{+, a_1} = \psi^{\Omega_0}_{+, a_1} + i \psi^{\Omega_0}_{-, a_1} ,  \\[+.5em]
\displaystyle \psi^{\Omega_1}_{-, a_1} = \psi^{\Omega_0}_{-, a_1} ,
\end{cases} \hspace{+2.em}
\begin{cases}
\displaystyle \psi^{\Omega_0}_{+, a_2} = \psi^{\Omega_2}_{+, a_2} + i \psi^{\Omega_2}_{-, a_2} , \\[+.5em]
\displaystyle \psi^{\Omega_0}_{-, a_2} = \psi^{\Omega_2}_{-, a_2} ,
\end{cases}
\end{eqnarray*}
because the sign of the Stokes curves are $\oplus$.
It follows from \eqref{eq:integral of S_odd} that  
\begin{eqnarray*}
\psi_{\pm, a_1} & = & {\rm exp} \pm \Bigl( \int_{a_1}^{a_2}S_{\rm odd}dx \Bigr) \psi_{\pm, a_2} \\
& = & e^{\pm \pi i c \eta}  \psi_{\pm, a_2}  
\end{eqnarray*}
hold. 
Furthermore, since $\psi_{\pm, {\rm IM}}$ and $\psi_{\pm, a_1}$ are related as 
\[
\psi_{\pm,\rm IM} = {e}^{\pm(\frac{1}{2}U - V)} \psi_{\pm, a_1} ,
\]
we can derive the connection formula for $\psi_{\pm, {\rm IM}}$ 
by combining the above formulas: 
\begin{eqnarray*}
\begin{cases}
\psi^{\Omega_1}_{+,\rm IM} = \psi^{\Omega_2}_{+,\rm IM} + 
i (1 + e^{2 \pi i c \eta } ) {e}^{U - 2V} 
\psi^{\Omega_2}_{-,\rm IM}  ,  \\[+.5em]
\psi^{\Omega_1}_{-,\rm IM} = \psi^{\Omega_2}_{-,\rm IM} .  
\end{cases}
\end{eqnarray*}
Thus we have 
\begin{equation}
{\mathfrak s}_1 = i  (1 + {e}^{2 \pi i c \eta} ) {e}^{U - 2V} .
\end{equation}

\noindent
$\bullet$ \underline{Computation of ${\mathfrak s}_4$.}\hspace{+.1em} 
Because we have to consider a connection problem on a Stokes curve 
emanating from the double turning point $x = \lambda_0$, 
we use the transformation theory prepared in Subsections 6.1 and 6.2. 
In view of \eqref{eq:x_0} we fine that the leading term of the transformation 
$\tilde{x} = x_0(x,t,c)$ maps the Stokes curve in question 
to either $i {\mathbb R}_{> 0}$ or $i {\mathbb R}_{< 0}$ in Figure \ref{fig:Stokes curve of (Can)}, 
and it depends on the choice of the branch of $\Delta^{1/4}$. 
We now assume that the Stokes curve in question is mapped to 
$i {\mathbb R}_{>0}$. (In view of \eqref{eq:normalized at infinity} and 
\eqref{eq:normalized at infinity-2}, the change of the choice 
of the branch of $\Delta^{1/4}$ is equivalent to the replacement $\alpha \mapsto - \alpha$.)
Since the normalization of $\psi_{\pm, {\rm IM}}$ 
in Figure \ref{fig:normalization of psi_pm IM in case s_4} is the same as 
in Remark \ref{remark for normalization} and the substituted 1-parameter solution 
is normalized at $t = \infty$, we can use the relation 
\eqref{eq:full order relation between SL_II and Can} in this case.
Using the connection formula \eqref{eq:connection formula on iR_+} 
and \eqref{eq:connection multiplier on iR_+} for $\tilde{\psi}_{\pm}$ on $i {\mathbb R}_{>0}$, 
we derive the following connection formula for $\psi_{\pm, {\rm IM}}$: 
\begin{eqnarray}
\begin{cases}
\psi^{\Omega_4}_{+,\rm IM} = \psi^{\Omega_5}_{+,\rm IM}  ,  \\[+.5em]
\psi^{\Omega_4}_{-,\rm IM} = \psi^{\Omega_5}_{-,\rm IM}  
 + m \hspace{+.2em} \psi^{\Omega_5}_{+,\rm IM} , 
\label{eq:connection formula for the double turning point of SL=II}
\end{cases}
\end{eqnarray}
where
\begin{equation}
m = \bigl(\tilde{\rho}(t_{\rm II}) - 2 \tilde{\sigma}(t_{\rm II}) \bigr)
\sqrt{\frac{\pi}{2}} {e}^{- 2 \eta t_{\rm II}} .
\label{eq:connection multiplier for double turning point}
\end{equation}
Moreover, since 
\[
\tilde{\rho}(t_{\rm II}) - 2 \tilde{\sigma}(t_{\rm II}) = 
- 2 \sqrt{2} \alpha \hspace{+.2em} {e}^{2 \eta t_{\rm II}} 
\]
holds by \eqref{eq:canonical Painleve 1-parameter solution}, we have the required 
Stokes multiplier ${\mathfrak s}_4$: 
\begin{equation}
{\mathfrak s}_4 = m = - 2 \sqrt{\pi} \alpha .
\end{equation}


\begin{rem} \label{comments on Stokes multipliers1} \normalfont
All the Stokes multipliers ${\mathfrak s}_j$ and ${\mathfrak s}'_j$ ($1 \le j \le 6$) 
in \eqref{eq:list1} and \eqref{eq:list2} are independent of $t$. It is consistent with the 
theory of isomonodromic deformation. 
\end{rem}

\begin{rem} \label{comments on Stokes multipliers2} \normalfont
The Stokes multipliers ${\mathfrak s}_5$ and ${\mathfrak s}'_5$ in \eqref{eq:list1} and \eqref{eq:list2} 
are equal to 0. The reason is as follows. 
In the computation of ${\mathfrak s}_5$ we use the connection formula 
\eqref{eq:connection formula on R_-}, \eqref{eq:connection multiplier on R_-} 
for $\tilde{\psi}_{\pm}$, and hence we have 
\[
i \bigl(\tilde{\rho}(t_{\rm II}) + 2 \tilde{\sigma}(t_{\rm II}) \bigr)
\sqrt{\frac{\pi}{2}} \hspace{+.2em}  {e}^{2 \eta t_{\rm II}}  
\]
as the connection coefficient instead of \eqref{eq:connection multiplier for double turning point}. 
This quantity vanishes by \eqref{eq:canonical Painleve 1-parameter solution}. 
\end{rem}

\subsection{Derivation of the connection formula for the parametric Stokes phenomena  
through the Stokes multipliers of ($SL_{\rm II}$)}

Now we rederive the connection formulas describing the parametric Stokes phenomena  
for the 1-parameter solutions $\lambda_{\infty}(t,c,\eta;\alpha)$ and $\lambda_{\tau_1}(t,c,\eta;\alpha)$ 
of ($P_{\rm II}$) by using the explicit form of the Stokes multipliers of ($SL_{\rm II}$) 
computed in Subsection 6.3. 

If a true solution represented by a 1-parameter solution $\lambda(t,c,\eta;\alpha)$ for 
arg \hspace{-.3em} $c = \frac{\pi}{2} - \varepsilon$ and that by  $\lambda(t,c,\eta;\tilde{\alpha})$ for 
arg \hspace{-.3em} $c = \frac{\pi}{2} + \varepsilon$ coincide, then 
the corresponding Stokes multipliers ${\mathfrak s}_j(c;\alpha)$ and ${\mathfrak s}'_j(c;\tilde{\alpha})$ 
of ($SL_{\rm II}$) should coincide,  that is, 
\begin{equation}
{\mathcal S}[{\mathfrak s}_j(c;\alpha)] = 
{\mathcal S}[{{\mathfrak s}'_j}(c;\tilde{\alpha})] \hspace{+1.em} (1 \le j \le 6)
\end{equation}
should hold. 
Hence, comparing  ${\mathfrak s}_j(c;\alpha)$ and ${\mathfrak s}'_j(c;\tilde{\alpha})$ 
given by \eqref{eq:list1} and using  Corollary \ref{connection formula for 2V - U}, 
we find that
\begin{equation}
{\cal S}[{\mathfrak s}_j (c;\alpha)] = {\cal S}[{\mathfrak s}'_j (c;\tilde{\alpha})] \hspace{+.7em} 
\Rightarrow \hspace{+.7em} \tilde{\alpha} = \alpha  
\label{eq:conclusion from list1}
\end{equation}
holds in the case where the 1-parameter solution $\lambda_{\infty}(t,c,\eta;\alpha)$ 
is substituted into the coefficients of ($SL_{\rm II}$) and ($D_{\rm II}$).
This result \eqref{eq:conclusion from list1} is consistent with 
\eqref{connection formula for 1-parameter solution normalized at infinity}, 
that is, the parametric Stokes phenomenon does not occur to $\lambda_{\infty}(t,c,\eta;\alpha)$. 
Similarly, in the case of $\lambda_{\tau_1}(t,c,\eta;\alpha)$ being substituted,  
the comparison of ${\mathfrak s}_j(c;\alpha)$ and ${\mathfrak s}'_j(c;\tilde{\alpha})$ 
in \eqref{eq:list2} tells us that  
\begin{eqnarray}
{\cal S}[{\mathfrak s}_j (c;\alpha)] = {\cal S}[{\mathfrak s}'_j (c;\tilde{\alpha})] 
& \Rightarrow & 
\alpha \hspace{+.1em} {\cal S}\bigl[ {{e}}^W \bigl|_{{\rm{arg}}c = \frac{\pi}{2} - \varepsilon} \bigr] = 
\tilde{\alpha} \hspace{+.1em} {\cal S}\bigl[ {{e}}^W \bigl|_{{\rm{arg}}c = 
\frac{\pi}{2} + \varepsilon} \bigr]   \nonumber \\
& \Rightarrow & 
\tilde{\alpha} = (1 + {e}^{2 \pi i c \eta}) \hspace{+.2em} \alpha,  
\label{eq:conclusion from list2}
\end{eqnarray}
and this is consistent with \eqref{connection formula for 1-parameter solution normalized at tau_1}. 
Thus, we have rederived the connection formulas for the parametric Stokes phenomena for 
1-parameter solutions of ($P_{\rm II}$) through the computation of the Stokes multipliers 
of ($SL_{\rm II}$).


\appendix

\section{Homogenity}

The second Painlev\'e equation ($P_{\rm II}$) with a large parameter
is obtained by a change of variables   
\[
(w, z, a) :\Rightarrow (\eta^{\frac{1}{3}} \lambda, \eta^{\frac{2}{3}} t, \eta c)
\]
from the ``original" Painlev$\acute{\rm e}$ equation 
$\frac{d^2 w}{d z^2} = 2 w^3 + z w + a$, 
that is, if $w(z,a)$ is a solution of the original Painlev$\acute{\rm e}$ equation,  
then $\lambda(t,c,\eta)$ given by 
$\eta^{\frac{1}{3}} \lambda(t,c,\eta) = w(\eta^{\frac{2}{3}}t, \eta c)$ is a solution of ($P_{\rm II}$)). 
Hence various quantities which appeared in this paper have a homogenity 
with respect to the following scaling operation: 
\[
(x, t, c, \eta) \mapsto (r^{-\frac{1}{3}}x, r^{-\frac{2}{3}}t, r^{-1}c, r \eta) \hspace{+1.em} (r > 0).
\] 
For example, the homogenious degree of $\lambda_0$ which is an algebraic function defined by 
$2\lambda_0^3 + t \lambda_0 + c = 0$ is $-\frac{1}{3}$, that is, 
\[
\lambda_0(r^{-\frac{2}{3}}t, r^{-1}c, r \eta) = \eta^{-\frac{1}{3}} \lambda_0(t,c).
\]

We list the homogenious degrees of quantities below.
\begin{eqnarray}
\lambda^{(0)}_k(t,c) & : & \Bigl( k - \frac{1}{3} \Bigr) \hspace{+1.em}  (k \ge 0),   
       \label{eq:homogenity of lambda(0)_k}      \\
\nu^{(0)}_k(t,c) & : &  \Bigl( k - \frac{2}{3} \Bigr)  \hspace{+1.em}  (k \ge 0),   \\
\lambda^{(0)}(t,c,\eta) & : & - \frac{1}{3} , \\
\nu^{(0)}(t,c,\eta) & : & - \frac{2}{3} , \\
\Delta(t,c) & : & - \frac{2}{3} , \\
    \label{eq:homogenity of Delta}   
\tau_j(c) & : & - \frac{2}{3} \hspace{+1.em} (j =1,2,3),
\end{eqnarray}
\begin{eqnarray}
R_k(t,c) & : & \Bigl( k + \frac{2}{3} \Bigr)   
       \hspace{+1.em}   (k \ge 0) ,  \hspace{+.5em}  \\
R(t,c,\eta) & : & + \frac{2}{3} , \\
R_{\rm odd}(t,c,\eta) & : & + \frac{2}{3} , \\
\phi_{\rm II}(t,c) & : & -1,  
\end{eqnarray}
\begin{eqnarray}
\lambda^{(k)}_{\ell}(t,c) & : & \Bigl( - \frac{1}{3} + \frac{1}{2} k + \ell \Bigr)    
     \hspace{+1.em}  (k \ge 0 , \ell \ge 0) , \hspace{+2.0em}   \\
\lambda(t,c,\eta;\alpha) & : & - \frac{1}{3} ,  
     \label{eq:homogenity of lambda}   \\
\nu(t,c,\eta;\alpha) := \eta^{-1} \frac{d}{dt} \lambda(t,c,\eta;\alpha)
 & : & - \frac{2}{3} ,  
     \label{eq:homogenity of nu}   \\
W(c,\eta) & : & 0 ,
     \label{eq:homogenity of W}
\end{eqnarray}

\begin{eqnarray}
Q_{\rm II}(x,t,c,\eta;\alpha) & : & - \frac{4}{3} , 
      \label{eq:homogenity of Q_II} \\
A_{\rm II}(x,t,c,\eta;\alpha) & : & + \frac{1}{3} ,  \\
S(x,t,c,\eta;\alpha) & : & + \frac{1}{3} ,  \\
S_{\rm odd}(x,t,c,\eta;\alpha)& : & + \frac{1}{3} ,   
      \label{eq:homogenity of S_odd}  \\
a_j(t,c) & : & - \frac{1}{3}  \hspace{+1.em} (j =1,2) ,  \hspace{+2.8em} 
\end{eqnarray}
\begin{eqnarray}
\psi_{\pm,\infty}(x,t,c,\eta;\alpha) & : & - \frac{1}{6} ,  \hspace{+8.3em} \\
\psi_{\pm,\rm IM}(x,t,c,\eta;\alpha) & : & - \frac{1}{6} ,  \\
U(t,c,\eta;\alpha) & : & 0 ,  
     \label{eq:homogenity of U}  \\
V(t,c,\eta;\alpha) & : & 0 .  
     \label{eq:homogenity of V}
\end{eqnarray}
These facts can be easily verified by straightforward computations.
\begin{prop}\label{homogenity of transformation series}
The formal series below have the following homogenious degrees:
\begin{eqnarray}
x_{\rm II}(x,t,c,\eta;\alpha) & : & - \frac{1}{2} , \hspace{+6.8em}
     \label{eq:homogenity of x_II}    \\
\sigma(t,c,\eta;\alpha) & : & 0 ,
     \label{eq:homogenity of sigma}  \\
\rho(t,c,\eta;\alpha) & : & 0 ,
     \label{eq:homogenity of rho}  \\
t_{\rm II}(t,c,\eta;\alpha) & : & -1 .    
     \label{eq:homogenity of t_II}
\end{eqnarray}
\end{prop}
\begin{proof}
First we check the homogeneity of $x_{\rm II}$, $\sigma$ and $\rho$. 
Due to \eqref{eq:x_0} we know that 
the homogenious degrees of $x_0$, $\sigma$ and $\rho$ is $- \frac{1}{2}$, $0$ and $0$ respectively. 
In what follows we show that 
the homogenious degrees of $x^{(k)}_{\ell}$, ${\sigma}^{(k)}_{\ell}$ and ${\rho}^{(k)}_{\ell}$ are 
\begin{eqnarray}
x^{({k})}_{\ell} & : & \Bigl( - \frac{1}{2} + \frac{1}{2} k + \ell \Bigr) \hspace{+1.em}  (k \ge 0, \ell \ge 0) ,
     \label{eq:homogenity of x(k)_n} \\
\sigma^{(k)}_{\ell} & : & \Bigl( \frac{1}{2} k + \ell \Bigr) \hspace{+1.em}  (k \ge 0, \ell \ge 0)
     \label{eq:homogenity of sigma(k)_n} \\
\rho^{(k)}_{\ell} & : & \Bigl( \frac{1}{2} k + \ell \Bigr) \hspace{+1.em}  (k \ge 0, \ell \ge 0)
     \label{eq:homogenity of rho(k)_n}
\end{eqnarray}
by induction. 
We note that, since the claims \eqref{eq:homogenity of x(k)_n} $\sim$ 
\eqref{eq:homogenity of rho(k)_n} are true 
for $k \ge 1, \ell = 0$ by \eqref{eq:x(k)_0 = 0}, 
it suffices to confirm \eqref{eq:homogenity of x(k)_n} $\sim$ 
\eqref{eq:homogenity of rho(k)_n}
for $k =k', \ell = \ell'$ ($k', \ell' \ge 0$) under the assumptions that 
they are true for 
$0 \le k \le k' - 1, \ell \ge 0$ and $k = k',  0 \le \ell \le \ell' - 1$. 
In view of \eqref{eq:transformation relation for potential} the differential equation 
which determines $x^{(k')}_{\ell'}$ has the form 
\begin{equation}
8 x_0 \frac{\partial x_0}{\partial x} 
(x_0 \frac{\partial x^{(k')}_{\ell'}}{\partial x} + x^{(k')}_{\ell'} \frac{\partial x_0}{\partial x} )
 = r^{(k')}_{\ell'} (x,t,c), 
\label{eq:x(k)_n}
\end{equation}
and the homogenious degree of $r^{(k')}_{\ell'}$ is ($- \frac{4}{3} + \frac{1}{2} k' + \ell' $) 
by induction hypothesis. 
(We note that $r^{(k')}_{\ell'}$ ($k', \ell' \ge 0$) are holomorphic in $x$ and 
have a zero of order at least 1 at $x = \lambda_0$ because $\sigma$, $\rho$ and $E$ 
are determined by the equations \eqref{eq:sigma} $\sim$ \eqref{eq:E}.) 
Since $x^{(k')}_{\ell'}$ is given, as a unique holomorphic solution of \eqref{eq:x(k)_n} 
at $x = \lambda_0$, by 
\begin{equation} 
x^{(k')}_{\ell'} (x,t,c) = \frac{1}{x_0} \int_{\lambda_0}^{x}    
\frac{r^{(k')}_{\ell'}}{8 x_0 \frac{\partial x_0}{\partial x}} \hspace{+.2em} dx ,
\label{eq:expression of x(k')_ell'}
\end{equation}
it is homogenious with degree $(- \frac{1}{2} + \frac{1}{2} k' + \ell')$ . 
The homogenious degrees of $\sigma^{(k')}_{\ell'}$ and 
$\rho^{(k')}_{\ell'}$ can be computed from \eqref{eq:sigma}, \eqref{eq:rho} 
and \eqref{eq:homogenity of x(k)_n}.
Furthermore, 
taking into account that the homogenious degree of $t_0$ defined by \eqref{eq:t_0} 
is $-1$ and that the formal power series $t^{(k)}_{\rm II}$ ($k \ge 0$) 
are defined by \eqref{eq:definition of t(0)} and \eqref{eq:definition of t_II(k)}, 
we find that the homogenious degree of the formal series $t_{\rm II}$ is $-1$. 
\end{proof}


\section{Holomorphy of $Z$ at $c = 0$}

$Z = Z(c,\eta)$ introduced in Section 6 is a formal power series given by 
\begin{equation}
Z = \frac{1}{2}U^{(0)} + \int_{\infty}^{x}\bigl( S^{(0)}_{\rm odd} - \eta S_{-1} \bigr) dx 
- \eta \bigl( t^{(0)}_{\rm I\hspace{-.1em}I} - t_0 \bigr) 
- \eta \bigl( {x^{(0)}_{\rm I\hspace{-.1em}I}}^2 - {x_0}^2 \bigr).   
\label{eq:Z in Appendix}
\end{equation}
The right-hand side of \eqref{eq:Z in Appendix} is independent of 
both $x$ and $t$. (We can verify this fact by using \eqref{eq:relation for transformation series}.)
The aim of Appendix B is to show that 
the coefficients $Z_{\ell}(c)$ of $\eta^{-\ell}$ in $Z(c,\eta)$ are holomorphic 
at $c = 0$ (for all $\ell \ge 0$) when the 1-parameter solution $\lambda_{\infty}$ 
normalized at $t = \infty$ is substituted into the coefficients of ($SL_{\rm II}$) and ($D_{\rm II}$).

Let  $t_{\ast}$ be a point fixed in the domain in Figure \ref{fig:Riemann surface of sqrt Delta}, 
and $\delta_1, \delta_2, \delta_3$ be positive numbers satisfying that 
\begin{equation}
|c| < \delta_3  \hspace{+.4em}  \Rightarrow  \hspace{+.4em}
\{ t ; |t - t_{\ast}| \le \delta_2 \} \ni \hspace{-.9em}/ \hspace{+.4em} 
\tau_1(c) , \tau_2(c) , \tau_3(c) ,
\label{eq:condition 1}
\end{equation}
\begin{equation}
 |c| < \delta_3  \hspace{+.4em} ,
\hspace{+.4em} |t - t_{\ast}| < \delta_2   
 \hspace{+.4em}   \Rightarrow   \hspace{+.4em}
\{x ; |x - \lambda_0(t,c)| \le \delta_1 \}  \ni \hspace{-.9em}/ \hspace{+.4em} a_1(t,c) , a_2(t,c)  .
\label{eq:condition 2}
\end{equation} 
Note that, since all the three $P$-turning points $\tau_{j}(c)$ of ($P_{\rm II}$) 
tend to $t = 0$ in the $t$-plane and   
the two simple turning points $x = a_{j}$ of ($SL_{\rm II}$) 
tend to $x = - \lambda_{0}(t,0) = + \frac{i}{\sqrt{2}} t^{1/2}$ in the $x$-plane 
when $c$ tends to $0$, we can take such positive numbers $\delta_{j}$. 
For the above $\delta_1, \delta_2, \delta_3$ and $t_{\ast}$,  we define domains 
${\cal D}_1, {\cal D}_2, {\cal D}_3$ and $\cal D$ by
\[
{\cal D}_3 := \{ c ; |c| < \delta_3 \} , 
\]
\[
{\cal D}_2 := \{ t ; |t - t_{\ast}| < \delta_2 \},
\]
\[
{\cal D}_1(t,c) := \{x ; |x - \lambda_0(t,c)| < \delta_1 \}, 
\]
\[
{\cal D} := \{ (x,t,c) ; (t,c) \in {\cal D}_2 \times {\cal D}_3 , x \in {\cal D}_1(t,c)  \}.
\]
We will verify that all the coefficients of the formal power series which appear in the right-hand 
side of \eqref{eq:Z in Appendix} such as $U^{(0)}$, $t^{(0)}_{\rm II}$, $x^{(0)}_{\rm II}$, etc.\ 
\hspace{-.8em} , are holomorphic on $\cal D$. 
($\delta_1, \delta_2$, $\delta_3$ may be chosen sufficiently small again if necessary.)

\begin{lemm} \label{holomorphic dependence of U on c}
All the coefficients of the formal power series 
\[
U^{(0)}(t,c,\eta) = \int_{\infty}^{t} \bigl( \lambda^{(0)}(t,c,\eta) - \lambda_0(t,c) \bigr) dt
\]
are holomorphic on ${\cal D}_2 \times {\cal D}_3$.
\end{lemm}
\begin{proof}
Let ${{\cal D}_2}'$ be a domain in $t$-plane given by 
\[
{\cal D}'_2 = \bigcup_{t'} \{ t ; |t - t'| < \delta_2 \} \hspace{+.5em}
(\supset {\cal D}_2) .
\]
 \begin{figure}[h]
 \begin{center}
 \includegraphics[width=55mm]{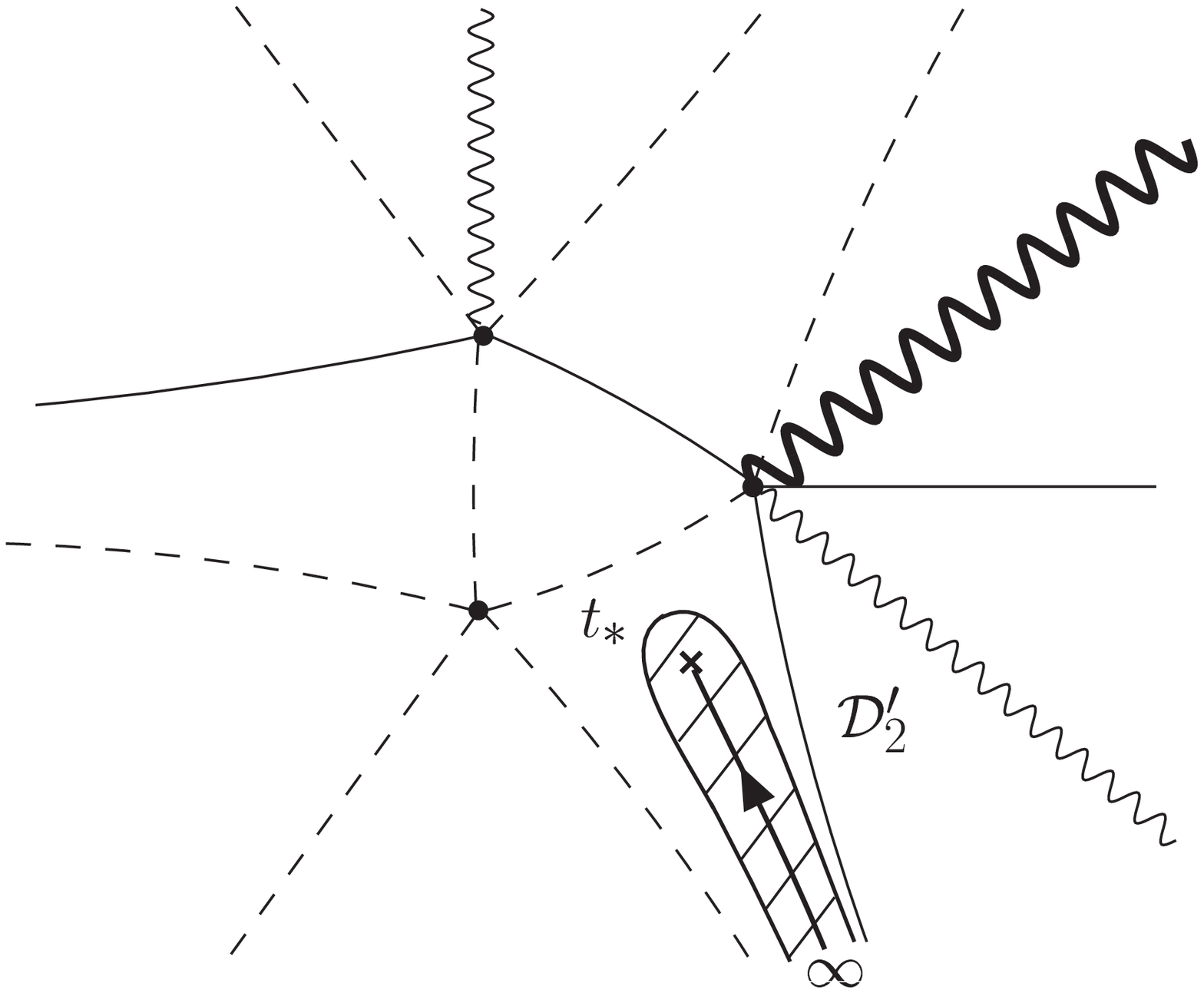}
 \end{center}
 \caption{Domain ${\cal D}'_2$.}
 \end{figure}
(Here $t'$ runs over all points on the path of integration of $U^{(0)}(t_{\ast},c,\eta)$. 
See Figure \ref{fig:normalization path of lambda(1)_infinity} for the choice of the path of integration.)
We show that all the coefficients of $\lambda^{(0)} - \lambda_0$ are holomorphic 
on ${{\cal D}_2}' \times {\cal D}_3$ and integrable at $t = \infty$.
($\delta_2$ and $\delta_3$ should be chosen smaller so that  
\[
|c| < \delta_3  \Rightarrow
\overline{{\cal D}'_2}
 \ni \hspace{-.9em}/ \hspace{+.4em} 
\tau_1(c) , \tau_2(c) , \tau_3(c)
\]
holds since the coefficients of $\lambda^{(0)}$ are singular at $P$-turning points. )

Since the discriminant of the algebraic equation $2 \lambda^3 + t \lambda + c = 0$ 
for $\lambda$ never vanishes on ${{\cal D}_2}' \times {\cal D}_3$, 
$\lambda_0$ is holomorphic on ${{\cal D}_2}' \times {\cal D}_3$ and
\[
\Delta = 6 \lambda_0^2 + t = \frac{4\lambda_0^3 - c}{\lambda_0}
\]
is also holomorphic and never vanishes on ${{\cal D}_2}' \times {\cal D}_3$. 
The holomorphy of all the coefficients on ${{\cal D}_2}' \times {\cal D}_3$ follows from 
these facts and the recursive relations \eqref{eq:recurrence relation of lambda(0)}. 
Indeed, by induction, we can confirm that the coefficient $\lambda^{(0)}_k$ 
is identically $0$ when $k$ is an odd number and 
has the form
\begin{equation}
\lambda^{(0)}_{2n}(t,c) = \frac{p_{2n}(\lambda_0,c)}{(4 \lambda_0^3 - c)^{5n - 1}} 
\label{eq:expression of lambda(0)_2n}
\end{equation}
when $k = 2n$ ($n \ge 1$) is an even number, where  
$p_{2n} \in {\mathbb C}[\lambda_0,c]$ is a polynomial with 
${\rm deg}_{\lambda_0} (p_{2n}) \le 9n - 2$. 
(Here ${\rm deg}_{\lambda_0} (p_{2n})$ is the degree of the polynomial $p_{2n}$ when 
it is considered a polynomial of $\lambda_0$.) 
Thus the holomorphy of $\lambda^{(0)}_k$ is obvious.
In view of \eqref{eq:expression of lambda(0)_2n} and 
\[
\frac{\partial}{\partial c}\lambda_0 = - \frac{1}{\Delta}, 
\]
the $c$-derivative of $\lambda^{(0)}_{2n}$ is also represented as 
\begin{equation}
\frac{\partial}{\partial c}\lambda^{(0)}_{2n}(t,c) =      
\frac{\hat{p}_{2n}(\lambda_0,c)}{(4 \lambda_0^3 - c)^{5n + 1}}
\hspace{+1.em} (n \ge 1) ,
\label{eq:expression of c-derivation of lambda(0)_2n}
\end{equation}
where $\hat{p}_{2n}(\lambda_0,c) \in {\mathbb C}[\lambda_0, c]$ is also a polynomial with 
${\rm deg}_{\lambda_0} (\hat{p}_{2n}) \le 9n + 4$. 
Since the behavior of $\lambda_0$ when $t \rightarrow \infty$ 
is given by \eqref{eq:behavior of lambda_0}, 
\eqref{eq:expression of c-derivation of lambda(0)_2n} implies that 
$\frac{\partial}{\partial c}\lambda^{(0)}_{2n}(t,c)$ is integrable at $t = \infty$ 
uniformly with respect to $c \in {\cal D}_3$. 
Therefore, 
\[
\int_{\infty}^t \lambda^{(0)}_{2n}(t,c) \hspace{+.1em} dt 
\]
is holomorphic on ${{\cal D}_2}' \times {\cal D}_3 (\supset {\cal D}_2 \times {\cal D}_3)$
for all $n \ge 1$.
\end{proof}

\begin{lemm} \label{holomorphic dependence of 1-parameter solution on c}
For the 1-parameter solution 
\begin{eqnarray*}
\left\{
\begin{array}{lll}
\displaystyle
\lambda_{\infty}(t,c,\eta;\alpha) & = & \lambda^{(0)}(t,c,\eta) + 
\alpha \eta^{-\frac{1}{2}} \lambda_{\infty}^{(1)}(t,c,\eta) {\rm{e}}^{\eta \phi_{\rm{I\hspace{-.1em}I}}} + 
(\alpha \eta^{-\frac{1}{2}})^2 \lambda_{\infty}^{(2)}(t,c,\eta) {\rm{e}}^{2 \eta \phi_{\rm{I\hspace{-.1em}I}}} 
+ \cdots \\[+.8em]
\displaystyle
\nu_{\infty}(t,c,\eta;\alpha)  & = & 
\displaystyle
\eta^{-1} \frac{d}{dt} \lambda_{\infty}(t,c,\eta;\alpha) \\[+1.em] 
& = &
\displaystyle
\nu^{(0)}(t,c,\eta) + 
\alpha \eta^{-\frac{1}{2}} \nu_{\infty}^{(1)}(t,c,\eta) {\rm{e}}^{\eta \phi_{\rm{I\hspace{-.1em}I}}} + 
(\alpha \eta^{-\frac{1}{2}})^2 \nu_{\infty}^{(2)}(t,c,\eta) {\rm{e}}^{2 \eta \phi_{\rm{I\hspace{-.1em}I}}} 
+ \cdots 
\end{array} \right. 
\end{eqnarray*}
of ($H_{\rm II}$) normalized at $t = \infty$, all the coefficients of 
the formal power series $\lambda_{\infty}^{(k)}(t,c,\eta)$, $\nu_{\infty}^{(k)}(t,c,\eta)$ 
($k \ge 0$) are holomorphic on ${\cal D}_2 \times {\cal D}_3$.
\end{lemm}
\begin{proof}
The holomorphy of all the coefficients of $\lambda^{(0)}$ has been  already 
shown in the proof of Lemma \ref{holomorphic dependence of U on c}.
Next we consider 
\[
\lambda^{(1)}_{\infty} = \frac{1}{\sqrt{\eta^{-1} R_{\rm odd}}} 
\hspace{+.1em} {\rm exp} \Bigl(
\int_{\infty}^t \bigl( R_{\rm odd} - \eta R_{-1} \bigr) dt 
\Bigr) .
\]
It is clear that $R_{-1}(t,c) = \sqrt{\Delta(t,c)}$ is holomorphic on ${{\cal D}_2}' \times {\cal D}_3$.
Due to the recursive relations \eqref{eq:R_k+1} for $R_{k}(t,c)$, 
we can obtain the following expression by induction:
\begin{eqnarray}
R_{2n}(t,c) & = & \frac{q_{2n}(\lambda_0,c)}{(4 \lambda_0^3 - c)^{5n + 2}} \hspace{+1.em} (n \ge 0) ,     
\label{eq:expression of R_2n} \\
R_{2n + 1}(t,c) & = & \frac{\lambda_0^{\frac{1}{2}} q_{2n+1}(\lambda_0,c)} 
{(4 \lambda_0^3 - c)^{5n + \frac{9}{2}}}      \hspace{+1.em} (n \ge 0) , 
\label{eq:expression of R_2n + 1} 
\end{eqnarray}
where $q_{2n}(\lambda_0,c), q_{2n + 1}(\lambda_0,c) \in {\mathbb C}[\lambda_0,c]$
are polynomials satisfying ${\rm deg}_{\lambda_0} ({q}_{2n}) \le 9n + 4$ and  
${\rm deg}_{\lambda_0} ({q}_{2n + 1}) \le 9n + 8$. 
Similarly to the proof of Lemma \ref{holomorphic dependence of U on c}, 
we can show that 
\[
\int_{\infty}^{t} R_{2n + 1}(t,c) \hspace{+.1em} dt
\]
is holomorphic on ${{\cal D}_2}' \times {\cal D}_3(\supset {\cal D}_2 \times {\cal D}_3)$ 
by using the expression \eqref{eq:expression of R_2n + 1} for all $n \ge 0$. 
Thus the holomorphy of the coefficients of $\lambda_{\infty}^{(1)}$ is verified.
The holomorphy of the coefficients of $\lambda_{\infty}^{(k)}$ ($k \ge 2$) 
can be confirmed from the recursive relations \eqref{eq:lambda_(k)l}, and 
the holomorphy of the coefficients of $\nu^{(0)}$ and $\nu_{\infty}^{(k)}$ ($k \ge 1$) 
can also be shown by using the relation 
$\nu_{\infty}(t,c,\eta;\alpha) = \eta^{-1} \frac{d}{dt} \lambda_{\infty}(t,c,\eta;\alpha)$.
\end{proof}

\begin{rem} \normalfont \label{defference of two normalization in the holomorphy}
When $c$ tends to 0, the three $P$-turning points merge to $t = 0$ simultaneously. 
Hence, if we consider a 1-parameter solution normalized at a $P$-turning point $t = \tau_1$, 
we can not expect that Lemma \ref{holomorphic dependence of 1-parameter solution on c} 
holds because the integration path of $\int_{\tau_1}^t R_{\rm odd} \hspace{+.2em}dt$ is 
pinched by the turning points and all the coefficients of the formal power series $R_{\rm odd}$ 
have a singularity at $P$-turning points. The assumption that a 1-parameter solution is 
normalized at $t = \infty$ is essential. 
\end{rem}

It follows from Lemma \ref{holomorphic dependence of 1-parameter solution on c} that, 
if $\lambda_{\infty}$ is substituted into $Q_{\rm II}$, then all the coefficients $Q_{\ell}^{(k)}$ 
are holomorphic on ${\cal D} \setminus \{ x = \lambda_0 \}$ and 
bounded as $x$ tends to $\infty$ when $(k,\ell) \ne (0,0)$. 
(Note that $Q^{(0)}_0 = Q_0 = (x - \lambda_0)^2 (x^2 + 2 \lambda_0 x + 3 \lambda_0^2 + t)$ 
for $(k,\ell) = (0,0)$.)
More precisely, $Q^{(k)}_{\ell}$ are represented as
\begin{equation}
Q^{(k)}_{\ell}(x,t,c) = \frac{u(x,t,c)}{(x - \lambda_0)^{m}}   
\label{eq:expression of Q(k)_ell}
\end{equation}
for some integer $m \ge 0$ and a polynomial $u(x,t,c)$ of $x$ 
with coefficients being holomorphic on ${\cal D}_2 \times {\cal D}_3$ 
and ${\rm deg}_{x} (u) \le m$.

\begin{lemm} \label{holomorphic dependence of integration of S_odd on c}
Assume that the 1-parameter solution $\lambda_{\infty}$ normalized at $t = \infty$ 
is substituted into $Q_{\rm II}$. 
Then all the coefficients of the formal power series  
\[
\int_{\infty}^{x} \bigl( S^{(0)}_{\rm odd}(x,t,c,\eta) - \eta S_{-1}(x,t,c) \bigr) dx    
\]
are holomorphic on ${\cal D}$ and all the coefficients of 
\[
\int_{\infty}^{x} S^{(k)}_{\rm odd}(x,t,c,\eta) dx
\]
are holomorphic on ${\cal D} \setminus \{ x = \lambda_0 \}$. 
Here the paths of integration of the above integrals are taken as in 
Figure \ref{fig:normalization of psi_pm-0}. 
\end{lemm}
\begin{proof}
For $(t,c) \in {\cal D}_2 \times {\cal D}_3$ 
let ${{\cal D}_1}'(t,c)$ be a domain in the $x$-plane defined by 
\[
{\cal D}'_1(t,c) =  \bigcup_{x'} \{ x ; |x - x'| < \delta_1 \} \hspace{+.5em}
(\supset {\cal D}_1(t,c)) , 
\]
where $x'$ runs over all points on the path of integration of 
$\int_{\infty}^{\lambda_0}(S^{(0)}_{\rm odd} - \eta S_{-1})dx$, 
and let ${\cal D}'$ be the following domain: 
\[
{\cal D}' := \{ (x,t,c) ; (t,c) \in {\cal D}_2 \times {\cal D}_3 , x \in {\cal D}'_1(t,c) \} 
\hspace{+.5em} (\supset {\cal D}) .
\] 
 \begin{figure}[h]
 \begin{center}
 \includegraphics[width=55mm]{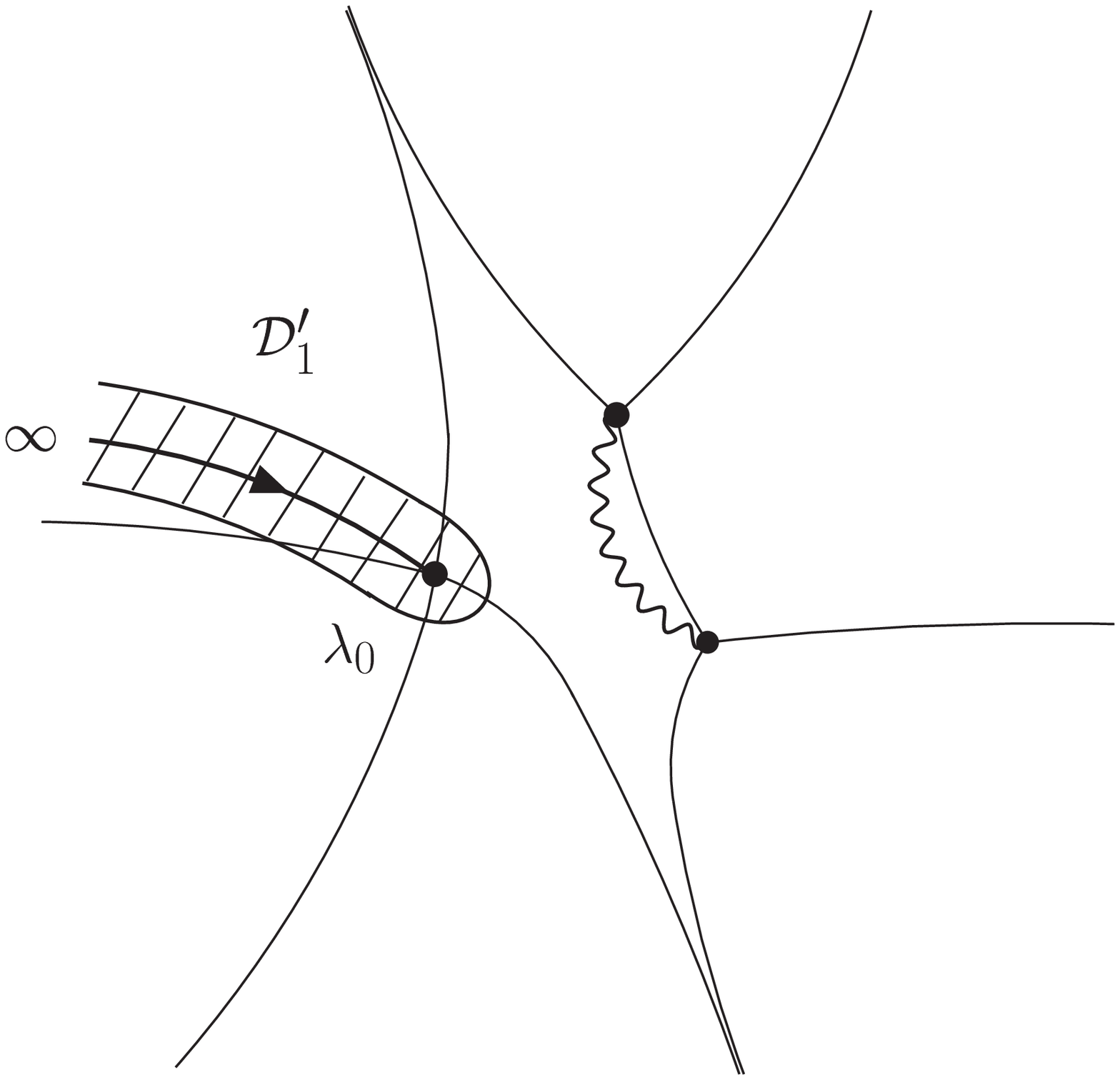}
 \end{center}
 \caption{Domain ${\cal D}'_1$.}
 \end{figure}
We prove the holomorphy and integrability at $x = \infty$ of all the coefficients
of $S_{\rm odd}^{(0)} - \eta S_{-1}$ and $S_{\rm odd}^{(k)}$. 
($\delta_1,\delta_2,\delta_3$ should be chosen smaller so that  
\[
(t,c) \in {\cal D}_2 \times {\cal D}_3 \Rightarrow  
\overline{{\cal D}'_1(t,c)}  \ni \hspace{-.9em}/ \hspace{+.4em} a_1(t,c), a_2(t,c)
\]
holds, if necessary, since the coefficients of $S_{\rm odd}^{(k)}$ are singular at turning points.
The above condition guarantees that $x^2 + 2 \lambda_0 x + 3 \lambda_0^2 + t \ne 0$ on ${\cal D}'$.)

By the recursive relations \eqref{eq:S^(k)_ell} and \eqref{eq:expression of Q(k)_ell}, 
for $(k,\ell) \ne (0,0)$, we obtain the following expression of $S^{(k)}_{\rm odd ,\ell}$: 
\begin{equation}
S^{(k)}_{\rm odd ,\ell}(x,t,c) = \frac{v(x,t,c)}
{(x - \lambda_0)^{n_1}(x^2 + 2 \lambda_0 x + 2 \lambda_0^2 + t)^{\frac{n_2}{2}}} , 
    \label{eq:expression of S(k)_ell} \\
\end{equation}
where $n_1 , n_2 \ge 0$ are some integers and $v(x,t,c)$ is a polynomial of $x$ 
whose coefficients are all holomorphic on ${\cal D}_2 \times {\cal D}_3$
and which satisfies ${\rm deg}_{x} (v) \le n_1 + n_2 - 2$.
Due to \eqref{eq:expression of S(k)_ell}, we have 
\begin{equation}
\frac{\partial}{\partial t} S^{(k)}_{\rm odd ,\ell}(x,t,c) = 
\frac{\tilde{v}(x,t,c)}
{(x - \lambda_0)^{n_1 + 1}(x^2 + 2 \lambda_0 x + 2 \lambda_0^2 + t)^{\frac{n_2}{2} + 1}} ,     
\label{eq:expression of t-derivation of S(k)_ell} 
\end{equation}
where  $\tilde{v}(x,t,c)$ is also a polynomial of $x$ 
whose coefficients are all holomorphic on ${\cal D}_2 \times {\cal D}_3$  
and which satisfies ${\rm deg}_x (\tilde{v}) \le n_1 + n_2 + 1$.
The $c$-derivative of $S^{(k)}_{\rm odd ,\ell}$ also has a form similar to 
\eqref{eq:expression of t-derivation of S(k)_ell}. 
These facts imply that 
$\frac{\partial}{\partial t} S^{(k)}_{\rm odd ,\ell}$ and 
$\frac{\partial}{\partial c} S^{(k)}_{\rm odd ,\ell}$ are both 
integrable at $x = \infty$ uniformly with respect to 
$(t,c) \in {\cal D}_2 \times {\cal D}_3$. 
Therefore all the coefficients of  
$\int_{\infty}^{x} \bigl( S^{(0)}_{\rm odd} - \eta S_{-1} \bigr) dx $ and 
$\int_{\infty}^{x} S^{(k)}_{\rm odd} dx $ ($k \ge 1$) 
are holomorphic on 
${\cal D}'\setminus\{ x = \lambda_0 \} \hspace{+.3em} (\supset {\cal D}\setminus\{x = \lambda_0\})$.
Furthermore, as noted in Lemma \ref{lemma for S(0)_odd}(ii), 
all the coefficients of $S^{(0)}_{\rm odd}$ are holomorphic at $x = \lambda_0$.
Thus we have the holomorphy on ${\cal D}' \hspace{+.3em} (\supset {\cal D})$ 
of the coefficients of $\int_{\infty}^{x} \bigl( S^{(0)}_{\rm odd} - \eta S_{-1} \bigr) dx $.
\end{proof}

Next we discuss the holomorphy of the coefficients of $x^{(k)}_{\rm II}$ 
and $t^{(k)}_{\rm II}$ on ${\cal D}$.
\begin{lemm} \label{holomorphic dependence of x_II on c}
Assume that the 1-parameter solution $\lambda_{\infty}$ normalized at $t = \infty$ 
is substituted into the coefficients of ($SL_{\rm II}$) and ($D_{\rm II}$). 
Then all the coefficients of the formal power series 
\[
x^{(k)}_{\rm II}(x,t,c,\eta) \hspace{+1.em} (k \ge 0) 
\]
are holomorphic on ${\mathcal D}$.
\end{lemm}
\begin{proof}
We prove this lemma by induction. Due to the condition \eqref{eq:condition 2}, 
$x_0(x,t,c) = \bigl[ \int_{\lambda_0}^x \sqrt{Q_0(x,t,c)} dx \bigr]^{\frac{1}{2}}$ 
is holomorphic on ${\mathcal D}$. The $x$-derivative 
$\frac{\partial x_0}{\partial x}$ of $x_0$ behaves as  
\[
\frac{\partial x_0}{\partial x} = \frac{1}{\sqrt{2}}\Delta^{\frac{1}{4}} + {\cal O}(x - \lambda_0)
\]
when $x$ tends to $\lambda_0$. 
Hence, since $\Delta \ne 0$ on ${\cal D}_2 \times {\cal D}_3$, we can assume that 
$\frac{\partial x_0}{\partial x} \ne 0$ holds on ${\mathcal D}$ 
(by taking sufficiently small $\delta_1 > 0$). 

Since the claims are true 
for $k \ge 1, \ell = 0$ by \eqref{eq:x(k)_0 = 0}, 
it suffices to confirm the holomorphies of $x^{(k')}_{\ell'}$ ($k', \ell' \ge 0$) 
under the assumptions that  
$x^{(k)}_{\ell}$ are holomorpic on ${\cal D}$ 
for $0 \le k \le k' - 1, \ell \ge 0$ and $k = k',  0 \le \ell \le \ell' - 1$. 
As noted in the proof of Proposition \ref{homogenity of transformation series}, 
the differential equation 
which determines $x^{(k')}_{\ell'}$ is given by \eqref{eq:x(k)_n}  
and $r^{(k')}_{\ell'}$ is holomorphic on ${\mathcal D} \setminus \{ x = \lambda_0 \}$ 
under the assumptions of induction. 
(Note that $1/x_0$ and $Q^{(k)}_{\ell}$ have a singularity at $x = \lambda_0$.)
However, it follows from the proof of \cite[Theorem3.1]{AKT Painleve WKB} 
that $r^{(k')}_{\ell'}$ is holomorphic and has a zero of order 1 at $x =\lambda_0$ 
because $\sigma, \rho$ and $E$ satisfy \eqref{eq:sigma} $\sim$ \eqref{eq:E}.
Thus $x^{(k')}_{\ell'}$ given by \eqref{eq:expression of x(k')_ell'}  
is holomorphic on ${\cal D}$.
\end{proof}

\begin{lemm} \label{holomorphic dependence of t_II on c}
Assume that the 1-parameter solution $\lambda_{\infty}$ normalized at $t = \infty$ 
is substituted into the coefficients of ($SL_{\rm II}$) and ($D_{\rm II}$). 
Then all the coefficients of the formal power series 
\[
t^{(0)}_{\rm II}(t,c,\eta) - t_0(t,c) ,
\]
\[
t^{(k)}_{\rm II}(t,c,\eta) \hspace{+1.em}  (k \ge 1) ,
\]
are holomorphic on ${\mathcal D}_2 \times {\mathcal D}_3$. 
\end{lemm}
\begin{proof}
By Lemma \ref{holomorphic dependence of x_II on c}, 
\[
\frac{\partial^{n} x^{(k)}_{\ell}}{\partial x^{n}}(\lambda_0,t,c)
\]
is holomorphic on ${\cal D}_2 \times {\cal D}_3$ for $k,\ell,n \ge 0$. 
Moreover, since 
\[
\frac{\partial x_0}{\partial x} (\lambda_0, t, c) \ne 0
\]
holds on ${\cal D}_2 \times {\cal D}_3$, all the coefficients of 
$\sigma^{(k)}(t,c,\eta)$ ($k \ge 0$) are holomorphic on ${\cal D}_2 \times {\cal D}_3$
by the definition \eqref{eq:sigma} of $\sigma$. 
Therefore, the holomorphy of all the coefficients of 
$t^{(0)}_{\rm II}(t,c,\eta) - t_0(t,c)$ and $t^{(k)}_{\rm II}(t,c,\eta)$ ($k \ge 1$) 
immediately follows from \eqref{eq:definition of t(0)} and \eqref{eq:definition of t_II(k)}.
\end{proof}

Thus we finally obtain the following: 
\begin{prop} \label{holomorphic dependence of Z on c}
Assume that the 1-parameter solution $\lambda_{\infty}$ normalized at $t = \infty$ 
is substituted into the coefficients of ($SL_{\rm II}$) and ($D_{\rm II}$). 
Then all the coefficients of the formal power series $Z(c,\eta)$ which is given by 
\eqref{eq:Z in Appendix} are holomorphic at $c = 0$.
\end{prop}
\begin{proof}
By Lemmas \ref{holomorphic dependence of U on c} $\sim$ 
\ref{holomorphic dependence of t_II on c}, we have shown that 
all the coefficients of $Z$ are holomorphic on ${\cal D}$ 
if $\lambda_{\infty}$ is substituted. Moreover, since $Z$ is independent of $x$ and $t$, 
they are holomorphic functions of $c$ on ${\cal D}_3$, which is a neighborhood of $c = 0$. 
\end{proof}


\end{document}